\documentclass[twoside]{article}

\usepackage[accepted]{aistats2024}
\usepackage[utf8]{inputenc} 
\usepackage[T1]{fontenc}    
\usepackage{hyperref}       
\usepackage{url}            
\usepackage{booktabs}       
\usepackage{amsfonts}       
\usepackage{nicefrac}       
\usepackage{microtype}      

\usepackage{natbib}
\usepackage{sidecap}

\usepackage{amssymb, amsmath, amsthm, latexsym}
\usepackage{url}
\usepackage{algorithm}
\usepackage{algorithmic}
\usepackage{tabularx}
\usepackage{paralist}
\usepackage{mathtools}
\usepackage{comment}
\usepackage{mathrsfs}

\usepackage{bbm} 

\usepackage{makecell}
\usepackage{multirow}
\usepackage{booktabs}

\usepackage{nicefrac} 

\usepackage[flushleft]{threeparttable} 

\usepackage{caption}
\usepackage{multirow}
\usepackage[dvipsnames]{xcolor}
\usepackage{colortbl}
\definecolor{bgcolor}{rgb}{0.8,1,1}
\definecolor{bgcolor2}{rgb}{0.8,1,0.8}
\definecolor{niceblue}{rgb}{0.0,0.19,0.56}

\hypersetup{colorlinks,linkcolor={blue},citecolor={niceblue},urlcolor={blue}}

\usepackage[dvipsnames]{xcolor}
\usepackage{pifont}

\newtheorem{theorem}{Theorem}[section]
\newtheorem{lemma}[theorem]{Lemma}

\newtheorem{definition}[theorem]{Definition}
\newtheorem{proposition}[theorem]{Proposition}
\newtheorem{assumption}[theorem]{Assumption}
\newtheorem{corollary}[theorem]{Corollary}
\newtheorem{remark}[theorem]{Remark}
\theoremstyle{plain}

\usepackage{xspace}

\newcommand{\algname}[1]{{\sf  #1}\xspace}

\newcommand{\circledOne}{\text{\ding{172}}}
\newcommand{\circledTwo}{\text{\ding{173}}}
\newcommand{\circledThree}{\text{\ding{174}}}
\newcommand{\circledFour}{\text{\ding{175}}}
\newcommand{\circledFive}{\text{\ding{176}}}

\usepackage[colorinlistoftodos,bordercolor=orange,backgroundcolor=orange!20,linecolor=orange,textsize=scriptsize]{todonotes}

\renewcommand{\endproof}{\hfill$\square$}

\newcommand{\1}{\mathbbm 1}
\newcommand{\Exp}{\mathbb{E}}
\newcommand{\Prob}{\mathbb{P}}
\newcommand{\R}{\mathbb{R}}
\newcommand{\eqdef}{\stackrel{\text{def}}{=}}

\def\<#1,#2>{\left\langle #1,#2\right\rangle}
\renewcommand{\leq}{\leqslant}
\renewcommand{\le}{\leqslant}
\renewcommand{\geq}{\geqslant}
\renewcommand{\ge}{\geqslant}

\newcommand{\Tr}{\mathrm{Tr}}

\newcommand{\rmd}{\mathrm d}
\newcommand{\Med}{\mathtt{Med}}
\newcommand{\Mean}{\mathtt{Mean}}
\newcommand{\SMM}{\mathtt{SMoM}}

\newcommand{\cG}{{\cal G}}
\newcommand{\cH}{{\cal H}}

\newcommand{\cN}{{\cal N}}
\newcommand{\cO}{{\cal O}}

\newcommand{\sfF}{\mathsf F}

\newcommand{\sfG}{\mathsf G}
\newcommand{\sfp}{\mathsf p}
\newcommand{\sfP}{\mathsf P}
\newcommand{\sfq}{\mathsf q}
\newcommand{\sfr}{\mathsf r}
\newcommand{\sfs}{\mathsf s}

\newcommand{\scH}{\mathscr H}

\newcommand{\mA}{{\bf A}}

\newcommand{\mI}{{\bf I}}

\newcommand{\clip}{\texttt{clip}}

\newcommand{\EE}{\mathbb{E}}

\newcommand{\tX}{\widetilde{X}}

\newcommand{\tnabla}{\widetilde{\nabla}}

\makeatletter
\newtheorem*{rep@theorem}{\rep@title}
\newcommand{\newreptheorem}[2]{%
\newenvironment{rep#1}[1]{%
 \def\rep@title{#2 \ref{##1}}%
 \begin{rep@theorem}}%
 {\end{rep@theorem}}}
\makeatother

\newreptheorem{theorem}{Theorem}
\newreptheorem{lemma}{Lemma}

\graphicspath{{plots/}}

\usepackage{makecell}

\usepackage{accents}
\newlength{\dhatheight}

\usepackage{pgfplotstable} 
\usetikzlibrary{automata, positioning, arrows, shapes, fit, calc, intersections}
\usepgfplotslibrary{statistics}
\pgfplotsset{compat=1.15}

\begin{document}

\runningtitle{Breaking the Heavy-Tailed Noise Barrier in Stochastic Optimization Problems}

\runningauthor{Puchkin, Gorbunov, Kutuzov, Gasnikov}

\twocolumn[

\aistatstitle{Breaking the Heavy-Tailed Noise Barrier \\in Stochastic Optimization Problems}

\aistatsauthor{
Nikita Puchkin
\And Eduard Gorbunov \And Nikolay Kutuzov
\And Alexander Gasnikov
}

\aistatsaddress{
HSE University,\\
IITP RAS
\And 
MBZUAI \And MIPT
\And
University Innopolis,\\ MIPT, ISP RAS
}
]

\begin{abstract}
We consider stochastic optimization problems with heavy-tailed noise with structured density. For such problems, we show that it is possible to 
get faster rates of convergence than $\cO(K^{\nicefrac{-2(\alpha - 1)}{\alpha}})$,
when the stochastic gradients have finite moments of order $\alpha \in (1, 2]$. In particular, our analysis allows the noise norm to have an unbounded expectation. To achieve these results, we stabilize stochastic gradients, using smoothed medians of means. We prove that the resulting estimates have negligible bias and controllable variance. This allows us to carefully incorporate them into \algname{clipped-SGD} and \algname{clipped-SSTM} and derive new high-probability complexity bounds in the considered setup.
\end{abstract}

\section{INTRODUCTION}

Stochastic optimization problems with heavy-tailed noise have been gaining a lot of attention in the machine learning community. This phenomenon can be partially explained due to the growing popularity of large language models \citep{brown2020language, gpt4} where stochastic gradients are often far from being well-concentrated \citep{zhang2020adaptive}. In theoretical studies, such behaviour is reflected in the so-called bounded (central) $\alpha$-th moment assumption with $\alpha \in (1,2]$ \citep{nemirovskij1983problem, zhang2020adaptive}, written as
\begin{equation}
	\EE \big[ \|g(x) - \nabla f(x)\|^\alpha \big] \leq \sigma^\alpha, \label{eq:bounded_alpha_moment} 
\end{equation}
where $\nabla f(x)$ is the gradient of the objective function $f(x)$, $g(x)$ is the stochastic gradient, and $\sigma \geq 0$. 
While in the classical literature on stochastic optimization the authors usually require the noise to have bounded variance (see, for instance, \citep{nemirovski2009robust, ghadimi2013stochastic}), many recent results on high-probability/in-expectation rates of convergence were obtained under a strictly weaker condition $1 < \alpha < 2$.

In deep learning and machine learning communities, one of the most popular techniques to deal with heavy-tailed noise is gradient clipping. In \citep{pascanu2013difficulty}, the authors showed that such a simple trick helps to stabilize neural network training via stochastic gradient descent. Their algorithm called \algname{clipped-SGD} and clipping in general were then studied in a series of papers including \citep{abadi2016deep, zhang2019gradient, chen2020understanding, zhang2020adaptive, mai2021stability, karimireddy2021learning}.
In particular, for strongly convex functions \citet{zhang2020adaptive} showed that the expected error of \algname{clipped-SGD} \citep{pascanu2013difficulty} decreases as $\cO(K^{\nicefrac{-2(\alpha - 1)}{\alpha}})$ when the number of iterations $K$ grows. In \citep{sadiev2023high}, the authors extended this result and proved that a similar bound (up to logarithmic factors) holds with high probability. According to \citep[Theorem 5]{zhang2020adaptive}, the rate of convergence $\cO(K^{\nicefrac{-2(\alpha - 1)}{\alpha}})$ is tight and cannot be improved if no assumptions, except for \eqref{eq:bounded_alpha_moment}, are made. However, this rate deteriorates when $\alpha$ is close to $1$, and, if $\alpha = 1$, the convergence is not even guaranteed. Fortunately, authors usually have to construct quite specific families of discrete distributions to attain the lower bound $\Omega(K^{\nicefrac{-2(\alpha - 1)}{\alpha}})$ (see, e. g., \citep{nemirovskij1983problem, devroye2016sub, zhang2020adaptive, cherapanamjeri2022optimal, vural2022mirror}). Such an extreme situation is unlikely to hold in practice and we can hope for more optimistic error guarantees. This brings us to a natural question: \emph{is it possible to achieve better rates of convergence in stochastic optimization problems with heavy-tailed noise under refined assumptions on its structure?} In this paper, we give an affirmative answer to this question.

\paragraph{Contribution.}

We consider a novel stochastic convex optimization setup with smooth (quasi-strongly/strongly) convex objective and structured noise (see Assumption~\ref{as:convolution} below), going beyond the standard bounded $\alpha$-th moment condition with $\alpha \in (1, 2]$. We provide new high-probability upper bounds on the error of versions of \algname{clipped-SGD} and its accelerated variant called \algname{clipped-SSTM} in smooth (quasi-strongly/strongly) problems, properly tailored to our setting. In particular, we do not assume boundedness of $\alpha$-th moments and get $\widetilde{\cO}(K^{-1/2})$ bound, which outperforms $\cO(K^{\nicefrac{-2(\alpha - 1)}{\alpha}})$ for $\alpha < \nicefrac{4}{3}$. Moreover, for symmetric noise distributions, we obtain rates of convergence, which match (up to logarithmic factors) the state-of-the-art ones derived under the bounded variance assumption \citep{nazin2019algorithms, davis2021low, gorbunov2020stochastic}. In particular, for smooth strongly convex problems, the dominating term in our upper bound decreases as $\widetilde{\cO}(K^{-1})$. Our approach relies on new non-asymptotic results on the performance of smoothed median of means.

\paragraph{Paper structure.}
The rest of the paper is organized as follows. In Section \ref{sec:setup}, we introduce our notation and formulate problem setup. Section \ref{sec:related_work} is devoted to an overview of related work. In Sections \ref{sec:smoothed_mom} and \ref{sec:main_results}, we present our main results and illustrate the performance of suggested algorithms in Section \ref{sec:numerical}. Many technical details are deferred to Appendix.

\section{SETUP AND NOTATION}
\label{sec:setup}
Before we formulate our main contributions, we need to introduce the notation and formalize the problem and setup we focus on. 

\paragraph{Notation.} Throughout the paper, we denote the standard Euclidean norm in $\R^d$ as $\|\cdot\|$. To simplify the bounds in the main text, we use $\widetilde{\cO}(\cdot)$ notation that hides constant and polylogarithmic factors. For any $x_1, \dots, x_n \in \R^d$, we denote $\Mean(x_1, \dots, x_n) = (x_1 + \ldots + x_n) / n$. For any random vectors $\xi_1, \dots, \xi_{2m + 1}$, $\Med(\xi_1, \dots, \xi_{2m + 1})$ stands for the $(m + 1)$-th order statistic (also called the median), taken in the component-wise fashion. For any non-zero $x \in \R^d$ and $\lambda > 0$, $\clip(x, \lambda) = \min\{1, \lambda / \|x\|\} x$ denotes the clipping operator. We also define $\clip(0, \lambda) = 0$ for all $\lambda > 0$.
For any $\theta > 0$, 
\[
	\Phi_\theta(t) = \frac1{\sqrt{2\pi} \theta} \int\limits_{-\infty}^t e^{-u^2 / (2 \theta^2)} \, \rmd u
\]
stands for the CDF of a Gaussian random variable with zero mean and variance $\theta^2$. Sometimes, we use the notation $a \land b$ and $a \lor b$, instead of $\min\{a, b\}$ and $\max\{a, b\}$, respectively. Along with the standard $\cO(\cdot)$ notation, we use 
the relations $g \lesssim h$ and $h \gtrsim g$, which are equivalent to $g = \cO(h)$.
Finally, for any functions $g: \R^d \rightarrow \R $ and $h: \R^d \rightarrow \R$, their convolution is denoted as
$
    g * h(x) = \int_{\R^d} g(x - y) h(y) \, \rmd y.
$
We also adopt the notation $g^{*k}(x) = \underbrace{g * \ldots * g}_{\text{$k$ times}}(x)$.

\paragraph{Setup.} We consider an unconstrained smooth convex optimization problem
\begin{equation}
    \min\limits_{x\in\R^d}f(x), \label{eq:minimization_problem}
\end{equation}
where the function $f:\R^d \to \R$ is accessible through the stochastic first-order oracle $\cG:\R^d \to \R^d$ that for given point $x\in\R^d$ returns some estimate of $\nabla f_{\xi}(x)$ of the true gradient $\nabla f(x)$. We make the following assumption about the distribution of $\nu = \nabla f_{\xi}(x) - \nabla f(x)$.

\begin{assumption}
    \label{as:convolution}
    For any $x \in \R^d$ and each $j \in \{1, \dots, d\}$, the marginal density $\sfp_j$ of the $j$-th component of the noise $\nu = \nabla f_\xi(x) - \nabla f(x)$ satisfies the following conditions:
    \begin{itemize}
        \item there exists $M_j > 0$, such that $\sfr_j(u) = (\sfp_j(u) - \sfp_j(-u)) / 2$ fulfils
        \[
            \int\limits_{-\infty}^{+\infty} u \sfr_j(u) \, \rmd u = 0
            \quad \text{and} \quad 
            \int\limits_{-\infty}^{+\infty} u^2 \left|\sfr_j(u)\right| \, \rmd u \leq M_j.
        \]
        \item there are $B_j > 0$ and $\beta_j \geq 1$, such that, for any $k \in \mathbb N$,
        \[
            \sfs_j^{*k}(u) \leq \frac{B_j k}{k^{\nicefrac{(\beta_j + 1)}{\beta_j}} + |u|^{1 + \beta_j}}, 
        \]
        where $\sfs_j(x) = (\sfp_j(x) + \sfp_j(-x)) / 2$.
    \end{itemize}
\end{assumption}

In Assumption \ref{as:convolution}, we split marginal  densities of noise components into a sum of symmetric and antisymmetric parts. For each $j \in \{1, \dots, d\}$, the antisymmetric remainder $\sfr_j(u)$ is a signed density with a finite second moment. However, the symmetric term $\sfs_j(u)$ may decay much slower, than $\sfr_j(u)$. As a result, the density $\sfp_j(u)$ has finite moments up to order $\alpha < (\beta_j \land 2)$. Note that if $\beta_j = 1$ for some $j \in \{1, \dots, d\}$, then the noise norm $\|\nu\|$ may have no expectation.

We proceed with several examples of $\sfs_j(u)$, satisfying Assumption \ref{as:convolution}. Obviously, if the $j$-th component of $\nu(x)$ has a standard Cauchy distribution, that is, $\sfs_j(u) = 1 / \pi \cdot 1 / (1 + u^2)$, then, for any $k \in \mathbb N$,
$
    \sfs_j^{*k}(u) = 1/\pi \cdot k/(k^2 + u^2),
$
and Assumption \ref{as:convolution} is fulfilled with $B_j = 1 / \pi$ and $\beta_j = 1$. This is a particular example of a symmetric $\alpha$-stable distribution with parameter $\alpha = 1$. All symmetric $\alpha$-stable distributions have a characteristic function of the form $\varphi(y) = e^{-|y / \sigma|^\alpha}$, where $\sigma > 0$ and $\alpha \in (0, 2]$ (see, e.g., \citep[Chapter XVII, Sections 5-6]{feller1971}). If $1 \leq \alpha \leq 2$, they also satisfy Assumption \ref{as:convolution} with $\beta_j = \alpha$ and some $B_j > 0$. In general, if $\sfs_j(u) \sim B / |u|^{1 + \beta_j}$, then it is known from probability theory that $\sfs_j^{*k}(u) \sim B k / |u|^{1 + \beta_j}$ for any $k \in \mathbb N$. Assumption \ref{as:convolution} can be viewed as a non-asymptotic version of this property.

We also make several standard assumptions about function $f$ itself. Similarly to \citep{gorbunov2021near, sadiev2023high}, it is sufficient for our analysis to make all the assumptions only on some compact subset (ball) of $\R^d$ since we show that with high probability the considered methods do not leave this compact. We start with the standard smoothness assumption.

\begin{assumption}\label{as:L_smoothness}
    There exists a set $Q\subseteq \R^d$ and constant $L > 0$ such that for all $x, y \in Q$
    \begin{eqnarray}
        \|\nabla f(x) - \nabla f(y)\| &\leq& L\|x - y\|, \label{eq:L_smoothness}\\
        \|\nabla f(x)\|^2 &\leq& 2L\left(f(x) - f_*\right), \label{eq:L_smoothness_cor_2}
    \end{eqnarray}
    where $f_* = \inf_{x \in Q}f(x) > -\infty$.
\end{assumption}

When $Q =\R^d$, \eqref{eq:L_smoothness_cor_2} follows from \eqref{eq:L_smoothness}. However, when $Q \neq \R^d$, condition \eqref{eq:L_smoothness_cor_2} can be derived from \eqref{eq:L_smoothness}, if the latter is assumed on a slightly larger set, see \citep[Appendix B]{sadiev2023high} for additional discussion.

We assume convexity or strong convexity of $f$ for the results with accelerated rates.

\begin{assumption}\label{as:str_cvx}
    There exists set $Q \subseteq \R^d$ and constant $\mu \geq 0$ such that $f$ is $\mu$-strongly convex, i.e., for all $x, y \in Q$, it holds that
    \begin{equation}
        f(y) \geq f(x) + \langle \nabla f(x), y - x \rangle + \frac{\mu}{2}\|y - x\|^2. \label{eq:str_cvx}
    \end{equation}
\end{assumption}

Finally, for non-accelerated case, it is sufficient to assume a relaxed condition called quasi-strong convexity.

\begin{assumption}\label{as:QSC}
    There exists set $Q \subseteq \R^d$ and constant $\mu \geq 0$ such that $f$ is $\mu$-quasi-strongly convex, that is, for all $x \in Q$ and $x^* = \arg\min\limits_{x\in \R^d} f(x)$
    \begin{equation}
        f(x^*) \geq f(x) + \langle \nabla f(x), x^* - x \rangle + \frac{\mu}{2}\|x - x^*\|^2. \label{eq:QSC}
    \end{equation}
\end{assumption}

The above assumption belongs to the class of conditions on structured non-convexity. When $\mu > 0$ it (together with smoothness) implies linear convergence for Gradient Descent \citep{necoara2019linear}.

\section{RELATED WORK}
\label{sec:related_work}

\paragraph{High-probability complexity bounds.} Under sub-Gaussian noise assumption, optimal (up to logarithmic factors) high-probability complexity bounds\footnote{Such results establish upper bounds for the number of oracle calls needed for a method to find point $x$ such that $f(x) - f(x^*)$ or $\|x - x^*\|^2$ or $\|\nabla f(x)\|^2$ are less than $\varepsilon$ with probability at least $\delta$, where $x^*$ is a solution of \eqref{eq:minimization_problem}.} are proven by \citet{nemirovski2009robust} for (strongly) convex non-smooth problems with bounded sub-gradients, by \citet{ghadimi2012optimal} for (strongly) convex smooth problems, and by \citet{li2020high} for smooth non-convex problems. These results are achieved for the same methods that are optimal in terms of the in-expectation convergence. However, when the noise has just a finite variance, some algorithmic changes seem to be necessary, e.g., as it is shown in \citep[Section 2]{sadiev2023high}, standard \algname{SGD} has provably bad (inverse-power instead of $\mathrm{poly}(\log(\nicefrac{1}{\delta}))$) dependence on the confidence level $\delta$ in this case.

A popular tool for overcoming this issue is gradient clipping, i.e., the application of the clipping operator to the gradient estimator. A version of gradient clipping is used by \citet{nazin2019algorithms} who derive the first (non-accelerated) high-probability complexity bounds for smooth (strongly) convex problems on compact domains with logarithmic dependence on $\nicefrac{1}{\delta}$ under bounded variance assumption. Accelerated results are obtained by \citet{davis2021low} and \citet{gorbunov2020stochastic} for smooth strongly convex and smooth convex problems respectively. \citet{gorbunov2021near} generalize these results to the case of problems with H\"older continuous gradients.

State-of-the-art high-probability complexity bounds are derived under bounded (central) $\alpha$-th moment assumption \eqref{eq:bounded_alpha_moment}. The first work in this direction is \citep{cutkosky2021high} where the authors derived optimal (up to logarithmic factors) bounds in the smooth non-convex regime with the additional assumption of boundedness of the gradients. Without this assumption, a worse bound is derived by \citet{sadiev2023high}, and the optimal one is obtained by \citet{nguyen2023high}. For (strongly) convex problems the results for \algname{clipped-SGD} and its accelerated version \algname{clipped-SSTM} \citep{gorbunov2020stochastic} are derived by \citet{sadiev2023high}. Up to logarithmic factors, these results match the known lower bounds in the strongly convex case \citep{zhang2020adaptive}. \citet{nguyen2023improved} improve the logarithmic factors in the upper bounds from \citet{sadiev2023high, nguyen2023high}. 
Recently, the generalization of the results from \citep{sadiev2023high} to the case of composite and distributed optimization were obtained by \citet{gorbunov2023high}.

\paragraph{Other results under heavy-tailed noise.} Although in our work we primarily focus on high-probability convergence results, we briefly discuss here other existing works devoted to the convergence of stochastic methods under heavy-tailed noise assumption. For convex functions with bounded gradients, \citet{nemirovskij1983problem} show $\cO(K^{\nicefrac{-(\alpha - 1)}{\alpha}})$ in-expectation convergence rate for Mirror Descent and \citet{vural2022mirror} propose an extension of this result for uniformly convex functions. For strongly convex functions with bounded gradients, \citet{zhang2020adaptive} show $\cO(K^{-\nicefrac{2(\alpha-1)}{\alpha}})$ in-expectation rate of convergence. In the smooth non-convex case, $\cO(K^{-\nicefrac{2(\alpha-1)}{(3\alpha-2)}})$  in-expectation convergence rate is achieved by \citet{zhang2020adaptive} who also derive a matching lower bound.

In the case of the noise with symmetric density function and bounded first moment, \citet{jakovetic2023nonlinear} derive $\cO(K^{-\zeta})$ in-expectation convergence rate for $\mu$-strongly convex $L$-smooth functions and \algname{SGD}-type methods with general non-linearities. However, parameter $\zeta$ is proportional to $\nicefrac{\mu}{L\sqrt{d}}$ in the worst case. Therefore, this rate can be much slower than $\cO(K^{\nicefrac{-2(\alpha-1)}{\alpha}})$ for ill-conditioned/large-scale problems though \citet{jakovetic2023nonlinear} do not assume \eqref{eq:bounded_alpha_moment} and consider general class of non-linearities. Our analysis also does not rely on \eqref{eq:bounded_alpha_moment}, but we additionally allow non-symmetric noise distributions and do not assume the existence of the finite first moment of the noise.

\paragraph{Median estimates.}

Median, median of means, and smoothed median were extensively used in the problems of robust mean estimation and robust machine learning (see, for instance, \citep{nemirovskij1983problem, minsker15, devroye2016sub, lugosi2019sub, lugosi2020, lecue2020, cherapanamjeri2022optimal}). A reader is referred to a comprehensive survey of \citet{lugosi2019} on this topic. Usually, the authors use median of means or its modifications to get sub-Gaussian rates of convergence, assuming the existence of only two moments. In \cite{cherapanamjeri2022optimal}, the authors went further and derived a minimax optimal upper bound  in the problem of mean estimation when observations have finite moments of order $\alpha \in (1, 2]$. However, the authors faced the same problem as \citet{zhang2020adaptive}: the rate of convergence became very slow when $\alpha$ approached $1$. This happens, because the family of distributions of interest is extremely large if one assumes the existence of $\alpha$-th moment only. In our paper, we exploit the special noise structure, described in Section \ref{sec:setup}. Under Assumption \ref{as:convolution}, we derive new non-asymptotic bounds on the performance of the smoothed median of means, which do not deteriorate even if the underlying density has quite heavy tails.

\section{SMOOTHED MEDIAN OF MEANS AND ITS PROPERTIES}
\label{sec:smoothed_mom}

In this section, we describe how to get reliable gradient estimates from noisy stochastic gradients given by the first-order oracle. Let us start with a simple example. Fix an arbitrary $x \in \R^d$ and assume that the noise $\nu = \nabla f_\xi(x) - \nabla f(x) \in \R^d$ has a symmetric absolutely continuous distribution. For any $j \in \{1, \dots, d\}$, let $\sfp_j(u)$ be the marginal density of $\nu_j$, the $j$-th component of $\nu$. Then the following proposition holds true.

\begin{proposition}
    \label{prop:median_symmetric_case}
    Fix any $j \in \{1, \dots, d\}$ and assume that the marginal density of $\nu_j$ is symmetric, that is, $\sfp_j(u) = \sfp_j(-u)$ for all $u \in \R$. Suppose that there exist positive numbers $B_j$ and $\beta_j$, such that
    \[
        \sfp_j(u) \leq \frac{B_j}{1 \vee |u|^{\beta_j + 1}},
        \quad \text{for all $u \in \R$.}
    \]
    Let $\nu_{j, 1}, \dots, \nu_{j, (2m + 1)}$ be independent copies of $\nu_j$. If $m > 3 / \beta_j$, then $\Exp \, \Med(\nu_{j, 1}, \dots, \nu_{j, (2m + 1)}) = 0$ and
    $\Exp \, \Med\left(\nu_{j, 1}, \dots, \nu_{j, (2m + 1)} \right)^2$ is finite.
\end{proposition}

The proof of the proposition with an explicit bound on the variance of $\Med(\nu_{j, 1}, \dots, \nu_{j, (2m + 1)})$ is postponed to Appendix.
Proposition \ref{prop:median_symmetric_case} shows that, despite the heavy tails of the underlying density $\sfp$, $m > \max\{3 / \beta_j : 1 \leq j \leq d\}$ oracle calls are enough to produce an unbiased estimate $g(x) = \Med(\nabla f_{\xi_1}(x), \dots, \nabla f_{\xi_m}(x))$ of $\nabla f(x)$ with a finite variance. After that, we can use the standard clipping technique to solve the optimization problem \eqref{eq:minimization_problem}.

Unfortunately, the symmetry assumption, which played the central role in Proposition \ref{prop:median_symmetric_case}, is rather restrictive. To deal with asymmetric distributions, we use more sophisticated gradient estimates, based on smoothed median of means.

\begin{definition}
    \label{def:smom}
    Let $\zeta$ be a random element in $\R^d$ and let $\theta > 0$ be an arbitrary number. For any positive integers $m$ and $n$, the smoothed median of means $\SMM_{m, n}(\zeta, \theta)$ is defined as follows:
    \[
        \SMM_{m, n}(\zeta, \theta)
        = \Med\left(\upsilon_1, \dots, \upsilon_{2m +1}\right),
    \]  
    where, for each $j \in \{0, \dots, 2m\}$,
    \[
        \upsilon_j = \Mean(\zeta_{j n + 1}, \dots, \zeta_{(j + 1) n}) + \theta \, \eta_{j + 1},
    \]
    $\zeta_1, \dots, \zeta_{(2m + 1) n}$ are i.i.d. copies of $\zeta$, and
    $\eta_1, \dots, \eta_{2m + 1} \sim \cN(0, \mI_d)$ are independent standard Gaussian random vectors.
\end{definition}

Let us briefly describe the idea behind our approach. Assuming that $\nabla f_\xi(x) - \nabla f(x)$ at a point $x \in \R^d$ has a density $\sfp(u)$, we represent the latter in the following form:
\[
    \sfp(u) = \sfs(u) + \sfr(u),
\]
where $\sfs(u) = (\sfp(u) + \sfp(-u)) / 2$ is a symmetric part and $\sfr(u) = (\sfp(u) - \sfp(-u)) / 2$ is an antisymmetric remainder. If the tails of the remainder $\sfr(u)$ are much lighter than the ones of $\sfp(u)$, we can make $n$ oracle calls at the point $x$ and take the average of $\nabla f_{\xi_1}(x), \dots, \nabla f_{\xi_n}(x)$. If $n$ is large enough, then the distribution of $\Mean(\nabla f_{\xi_1}(x), \dots, \nabla f_{\xi_n}(x))$ is almost symmetric. Hence, we can use the same trick as in Proposition \ref{prop:median_symmetric_case} to get an estimate of $\nabla f(x)$ with a finite variance. We add small Gaussian noise to ensure that the density of our estimate is infinitely differentiable, as we need it for technical purposes. Note that, in general, the expectation of $\SMM_{m, n}(\nabla f_\xi(x) - \nabla f(x), \theta)$ is not equal to zero, and it is a challenging task to show that it is sufficiently small.

Before we move to the heavy-tailed setup, let us illustrate the efficiency of our approach in the case when the stochastic gradients have a finite second moment.

\begin{lemma}
    \label{lem:smom_moments_light_tails}
    Assume that stochastic gradient $\nabla f_{\xi}(x)$ at a point $x \in \R^d$ has an absolutely continuous distribution and a finite second moment $\Exp (\nabla f_{\xi}(x) - \nabla f(x))  (\nabla f_{\xi}(x) - \nabla f(x))^\top = \Sigma$. Then, for any positive integer $m$ and $n$, it holds that
    \begin{align*}
        &
        \Exp \left\| \SMM_{m, n}\big(\nabla f_{\xi}(x), \theta \big) - \nabla f(x) \right\|^2
        \\&
        \leq 4 (2m + 1) \left( \frac{\Tr(\Sigma)}n + \theta^2 d \right).
    \end{align*}
    If, in addition, $m \geq 3$ and $n\theta^2 \gtrsim m \|\Sigma\|$, then
    \begin{align*}
        \left\| \Exp \SMM_{m, n}\big(\nabla f_{\xi}(x), \theta \big) - \nabla f(x) \right\|
        \lesssim \frac{m} {\theta n} \sqrt{\Tr(\Sigma^2)}.
    \end{align*}
\end{lemma}

\begin{remark}
    Lemma \ref{lem:smom_moments_light_tails} also yields that
    \begin{align*}
        &
        \big\| \Exp \SMM_{m, n}\big(\nabla f_{\xi}(x), \theta \big) - \nabla f(x) \big\|
        \\&
        \lesssim \sqrt{m} \left( \sqrt{\frac{\Tr(\Sigma)}n} + \theta \sqrt{d} \right).
    \end{align*}
    Though this bound is enough for our purposes, nevertheless, we find it useful to prove a dimension-free $\cO(1/n)$ upper bound on expectation of the smoothed median of means, which follows from the analysis of impact of antisymmetric density part on the expectation (see Lemma \ref{lem:expectation_comparison} in Appendix). 
\end{remark}

Lemma \ref{lem:smom_moments_light_tails} shows that, if the noise vector has a finite second moment, then the smoothed median of means has a small controllable shift and bounded variance. In this case, it behaves similarly to clipping. However, if we deal with heavy-tailed noise, the standard clipping technique fails, while the
smoothed median of means still has a small bias and finite variance.
We proceed with the main result of this section.

\begin{lemma}
    \label{lem:smom_moments_heavy_tails}
    Assume that the stochastic gradient $\nabla f_{\xi}(x)$ at a point $x \in \R^d$ has an absolutely continuous distribution. Suppose that, for any $j \in \{1, \dots, d\}$ and any $x \in \R^d$, the density of $\nu_j = \nabla f_{\xi}(x) - \nabla f(x)$ meets Assumption \ref{as:convolution}. Then, if $m > 2 + 3 / \beta_j$ and $\theta^2 n \geq (2 \lor m^2) M_j$ for all $j \in \{1, \dots, d\}$, it holds that
    \begin{align*}
        &
        \Exp \big\| \SMM_{m, n}\big(\nabla f_{\xi}(x), \theta \big) - \nabla f(x) \big\|^2
        \\&
        \lesssim m \Bigg\{ (1 + \theta^2) d + \sum\limits_{j = 1}^d \Bigg[ \left( \frac{M_j}{\theta n} \right)^2 + \left( \frac{2^{\beta_j} B_j}{\beta_j n^{\beta_j - 1}} \right)^{2 / \beta_j}
        \\&\hspace{4.5cm}
        + \left( \frac{B_j M_j}{\theta n^{\beta_j}} \right)^{2 / (\beta_j + 1)} \Bigg] \Bigg\}
    \end{align*}
    and
    \begin{align*}
        &
        \left\| \Exp \SMM_{m, n}\big(\nabla f_{\xi}(x), \theta \big) - \nabla f(x) \right\|
        \\&
        \lesssim \frac{m (1 + \theta)}{\theta^2 n} \sqrt{ \sum\limits_{j = 1}^d M_j^2} +
        \frac{m}{\theta^2 n}  \sqrt{\sum\limits_{j = 1}^d M_j^2 \left( \frac{2^{\beta_j} B_j}{n^{\beta_j - 1}} \right)^{2 / \beta_j}}.
    \end{align*}
\end{lemma}

\begin{remark}
    The bounds on $\big\| \Exp \SMM_{m, n}\big(\nabla f_{\xi}(x), \theta \big) - \nabla f(x) \big\|$ in Lemma \ref{lem:smom_moments_light_tails} and Lemma \ref{lem:smom_moments_heavy_tails} rely on the fact that convolution with the Gaussian density is infinitely differentiable. However, if $\sfs_1, \dots, \sfs_d$ are also sufficiently smooth, then one can take $\theta = 0$ and apply a similar technique as in Lemma \ref{lem:expectation_comparison} and Lemma \ref{lem:f-g_non-uniform_bound}.
\end{remark}

The proof of Lemma \ref{lem:smom_moments_heavy_tails} is moved to Appendix. It shows that the bias of the smoothed median of means decays rapidly with the growth of the batch size, though the noise may have extremely heavy tails. The reason for that is the special noise structure guaranteed by Assumption \ref{as:convolution}. This favourable property allows us to obtain faster rates of convergence in heavy-tailed stochastic convex optimization problems.

\section{MAIN RESULTS FOR STOCHASTIC OPTIMIZATION}
\label{sec:main_results}

In view of the results of the previous section, Assumption~\ref{as:convolution} allows constructing an estimator with bounded bias and variance using the smoothed median of means. Since in the analysis of the stochastic first-order methods we only use these two properties, we formulate them as a separate assumption for convenience.

\begin{assumption}\label{as:bounded_bias_and_variance}
    There exists $N \in \mathbb{N}$, aggregation rule $\mathcal{R}$ and (possibly dependent on $N$) constants $b \geq 0, \sigma \geq 0$ such that for an $x \in \R^d$ i.i.d.\ samples $\nabla f_{\xi_1}(x),\ldots, \nabla f_{\xi_N}(x)$ from the oracle $\cG(x)$ satisfy the following relations:
    \begin{eqnarray}
        \left|\EE[\nabla f_{\Xi}(x)] - \nabla f(x)\right| &\leq& b, \label{eq:bias_bound}\\
        \EE\left[\left\|\nabla f_{\Xi}(x) - \EE[\nabla f_{\Xi}(x)]\right\|^2\right] &\leq& \sigma^2, \label{eq:variance_bound}
    \end{eqnarray}
    where $\nabla f_{\Xi}(x) = \mathcal{R}(\nabla f_{\xi_1}(x),\ldots, \nabla f_{\xi_N}(x))$ and expectations are taken w.r.t.\ $\nabla f_{\xi_1}(x),\ldots, \nabla f_{\xi_N}(x)$.
\end{assumption}

We emphasize once again that Assumption~\ref{as:bounded_bias_and_variance} holds whenever Assumption~\ref{as:convolution} is satisfied. Indeed, we can take $\nabla f_{\Xi}(x) = \SMM_{m, n}\big( \nabla f_{\xi}(x), \theta \big)$ with parameters $m$, $n$, and $\theta$, satisfying the conditions of Lemma \ref{lem:smom_moments_heavy_tails}. In this case, the batch size is equal to $N = (2m + 1)n$.

\subsection{Convergence of \algname{clipped-SGD}}

We start with \algname{clipped-SGD} defined as follows:
\begin{equation}
    x^{k+1} = x^k - \gamma_k \clip(\nabla f_{\Xi^k}(x^k), \lambda_k), \label{eq:clipped_SGD}
\end{equation}
where $\nabla f_{\Xi^k}(x^k)$ is an estimator satisfying Assumption~\ref{as:bounded_bias_and_variance} sampled independently from previous iterations. Below we formulate the main convergence result for \algname{clipped-SGD} in the quasi-convex case.

\begin{theorem}\label{thm:clipped_SGD_main_cvx}
    Let Assumptions~\ref{as:L_smoothness} and \ref{as:QSC} with $\mu = 0$ hold on $Q = B_{2R}(x^*)$, where $R \geq \|x^0 - x^*\|$. Suppose that $\nabla f_{\Xi^k}(x^k)$ satisfies Assumption~\ref{as:bounded_bias_and_variance} with parameters $b_k, \sigma_k$ for $k = 0,1,\ldots,K$, $K > 0$ and $\gamma_k = \gamma = \Theta(\min\{\nicefrac{1}{L A}, \nicefrac{R}{\sigma  \sqrt{KA}}, \nicefrac{R}{bA}, \nicefrac{R}{b(K+1)}\}),$ 
    $
    \lambda_{k} \equiv \lambda  = \Theta(\nicefrac{R}{\gamma A}),
    $ 
    where $A = \ln(\nicefrac{4(K+1)}{\delta})$ and $b = \max_{k=0,1,\ldots,K} b_k$, $\sigma = \max_{k=0,1,\ldots,K} \sigma_k$. Then the iterates produced by \algname{clipped-SGD} after $K$ iterations with probability at least $1-\delta$ satisfy
    \begin{equation}
        f(\overline{x}^K) - f(x^*) = \widetilde\cO\left(\max\left\{\frac{LR^2}{K}, \frac{\sigma R}{\sqrt{K}}, bR\right\}\right), \notag
    \end{equation}
    where $\overline{x} = \frac{1}{K+1}\sum_{k=0}^K x^k$.
\end{theorem}

The rate of convergence in the above result matches (up to logarithmic factors) the best-known one for \algname{clipped-SGD} under bounded variance assumption \citep{gorbunov2021near}. Due to systematic bias bounded by $b$ the method reaches only $\widetilde{\cO}(bR)$ error after a sufficiently large number of steps. When the bias is just bounded and cannot be controlled, this situation is standard \citep{devolder2014first}. The proof of the above result follows the ones given in \citep{gorbunov2021near, sadiev2023high}: using the induction argument, we show that under a proper choice of parameters, the iterates stay in a bounded set with high probability, which allows us to apply standard Bernstein inequality (see Lemma~\ref{lem:Bernstein_ineq}). In particular, this proof technique differs from the standard ones that rely on the boundedness of the noise \citep{rakhlin2011making} or on the assumption that the noise is sub-Gaussian \citep{harvey2019simple}.

However, in our setup, we can control the bias. For example, if the distribution is symmetric and has a bounded moment of the order $\beta$ for some $\beta > 0$, then according to Proposition~\ref{prop:median_symmetric_case}, it is sufficient to use coordinate-wise median estimator to get $\nabla f_{\Xi}(x)$ satisfying Assumption~\ref{as:bounded_bias_and_variance} with $b = 0$ and $\sigma^2 = d(2m + 1) \left(1 \vee \frac{4 B_j}{\beta_j} \right)^{2 / \beta_j}$ (see \eqref{eq:var_of_median}) using $\cO(\nicefrac{1}{\beta})$ samples of $\nabla f_\xi(x)$. In this case, we have the following result.

\begin{corollary}[Symmetric noise]\label{cor:SGD_symmetric_noise}
    Let the assumptions of Theorem~\ref{thm:clipped_SGD_main_cvx} hold and for all $x \in \R^d$ the noise $\nu = \nabla f_{\xi}(x) - \nabla f(x)$ satisfies the conditions from Proposition~\ref{prop:median_symmetric_case}. Then the iterates produced after $K$ iterations of \algname{clipped-SGD} with $\nabla f_{\Xi^k}(x^k)$ being a coordinate-wise median of $2m+1$ samples $\nabla f_{\xi}(x^k)$ with $m > \max\{3 / \beta_j : 1 \leq j \leq d\}$ and $\gamma_k = \gamma = \Theta(\min\{\nicefrac{1}{L A}, \nicefrac{R}{\sigma  \sqrt{KA}}\}),$ 
    $
    \lambda_{k} \equiv \lambda  = \Theta(\nicefrac{R}{\gamma A})
    $ for $\sigma^2 = d(2m + 1) \left(1 \vee \frac{4 B_j}{\beta_j} \right)^{2 / \beta_j}$ and 
    where $A = \ln(\nicefrac{4(K+1)}{\delta})$ with probability at least $1-\delta$ satisfy
    \begin{equation}
        f(\overline{x}^K) - f(x^*) = \widetilde\cO\left(\max\left\{\frac{LR^2}{K}, \frac{\sigma R}{\sqrt{K}}\right\}\right) \notag
    \end{equation}
    and the overall number of stochastic oracle calls equals $(2m+1)K = \cO(K\max\{1/\beta_j: 1 \leq j \leq d\})$.
\end{corollary}

This result implies that as long as the distribution is symmetric, its tails can be even heavier than the ones of Cauchy distribution, i.e., moments of order larger than $\beta$ for some $\beta \in (0,1]$ can be unbounded, but \algname{clipped-SGD} with coordinate-wise median estimator inside will still converge as in the case when the stochastic gradients are unbiased and have bounded variance. We emphasize that the existing state-of-the-art high probability convergence results \citep{sadiev2023high, nguyen2023high, nguyen2023improved} have a slower decreasing main term (of the order $\widetilde{\cO}(K^{\nicefrac{-(\alpha-1)}{\alpha}})$) and are derived for much lighter tails. However, in contrast to Corollary~\ref{cor:SGD_symmetric_noise}, the mentioned results do not rely on the symmetry.

Finally, we consider the general case, when the noise satisfies Assumption~\ref{as:convolution}, i.e., the noise also has a non-symmetric component. In this case, Lemma~\ref{lem:smom_moments_heavy_tails} implies that the smoothed median of means gives an estimator $\nabla f_{\Xi}(x)$ satisfying Assumption~\ref{as:bounded_bias_and_variance} with 
\begin{align}
    b &= \cO(\nicefrac{C}{n}),\quad \sigma^2 = \cO(d(1+\theta^2)+D), \label{eq:bias_var_general_case}\\
    C &= \frac{(1 + \theta)}{\theta^2} \sqrt{ \sum_{j = 1}^d M_j^2}\notag\\
    &\qquad+ \frac{1}{\theta^2}  \sqrt{\sum_{j = 1}^d M_j^2 \left( \frac{2^{\beta_j} B_j}{n^{\beta_j - 1}} \right)^{2 / \beta_j}}, \label{eq:general_case_C_constant}\\
    D &= \sum_{j = 1}^d \left( \left( \frac{M_j}{\theta n} \right)^2 + \left( \frac{2^{\beta_j} B_j}{\beta_j n^{\beta_j - 1}} \right)^{2 / \beta_j}\right)\notag\\
    &\qquad+ \sum_{j = 1}^d\left( \frac{B_j M_j}{\theta n^{\beta_j}} \right)^{2 / (\beta_j + 1)} \label{eq:general_case_D_constant}
\end{align}
using $\cO(n)$ samples $\nabla f_\xi(x)$ (when $m = \cO(1)$). Together with Theorem~\ref{thm:clipped_SGD_main_cvx} this implies the following result.

\begin{corollary}[General noise]\label{cor:SGD_biased_heavy_noise}
    Let the assumptions of Theorem~\ref{thm:clipped_SGD_main_cvx} hold and for all $x \in \R^d$ the noise $\nu = \nabla f_{\xi}(x) - \nabla f(x)$ satisfies Assumption~\ref{as:convolution}. Then \algname{clipped-SGD} with $\nabla f_{\Xi^k}(x^k)$ being the smoothed median of means of $\cO(n)$ samples $\nabla f_{\xi}(x^k)$ and $\gamma_k = \gamma = \Theta(\min\{\nicefrac{1}{L A}, \nicefrac{R}{\sigma  \sqrt{KA}}, \nicefrac{R}{bA}, \nicefrac{R}{b(K+1)}\}),$ 
    $\lambda_{k} \equiv \lambda  = \Theta(\nicefrac{R}{\gamma A})$ for $b$ and $\sigma$ defined in \eqref{eq:bias_var_general_case} with probability at least $1-\delta$ after $K$ iterations ensures that $f(\overline{x}^K) - f(x^*)$ equals
    \begin{align*}
        \widetilde\cO\!\left(\!\max\!\left\{\frac{LR^2}{K}, \frac{\sqrt{(1 +  \theta^2) d + D} R}{\sqrt{K}}, \frac{(1 + \theta)CR}{\theta^2n}\right\}\!\right)\!, \notag
    \end{align*}
    where $n \geq \Omega(\nicefrac{\max_{j\in [d]} M_j}{\theta^2})$ and $C$, $D$ are defined in \eqref{eq:general_case_C_constant}-\eqref{eq:general_case_D_constant}. The overall number of oracle calls equals $\cO(nK)$.
\end{corollary}

Due to the heavy-tailedness of the symmetric part of the noise distribution and the presence of bias, the variance term does not necessarily improve with the growth of the number of samples. However, the bias term still can be smaller than any predefined level via increasing $n$. When the bias is large, i.e., the noise is sufficiently non-symmetric, then one can take $n = \widetilde\cO(\varepsilon^{-1})$ and $K = \widetilde\cO(\varepsilon^{-2})$ to guarantee $f(\overline{x}^K) - f(x^*) \leq \varepsilon$ with probability at least $1-\delta$, i.e., the total oracle complexity is $\widetilde\cO(\varepsilon^{-3})$. However, when the bias is small, i.e., $\max_{j\in [d]}M_j$ are sufficiently small\footnote{In practice, this can happen when a non-symmetric noise is added to the stochastic gradients with symmetric noise, e.g., this can happen in some mechanisms for ensuring differential privacy such as the one from \citep{guo2023privacy} (the non-symmetric part of the noise in the resulting vector after averaging over multiple clients can have a small variance when the number of clients is large, which is typical for modern Federated Learning applications \citep{kairouz2021advances}).}, then for $\varepsilon = \Theta(\nicefrac{(1+\theta)CR}{\theta^2})$ one can achieve even $K = \widetilde\cO(\varepsilon^{-2})$ total oracle complexity that matches (up to logarithmic factors and constants related to the variance) the main term in the optimal complexities under bounded variance assumption \citep{gorbunov2020stochastic}. However, in contrast to the existing results, we do not require the noise to have a finite first moment.

The following theorem gives a convergence rate for \algname{clipped-SGD} in the quasi-strongly convex case.

\begin{theorem}\label{thm:clipped_SGD_main_str_cvx}
    Let Assumptions~\ref{as:L_smoothness} and \ref{as:QSC} with $\mu > 0$ hold on $Q = B_{2R}(x^*)$, where $R \geq \|x^0 - x^*\|$. Suppose that $\nabla f_{\Xi^k}(x^k)$ satisfies Assumption~\ref{as:bounded_bias_and_variance} with parameters $b_k, \sigma_k$ for $k = 0,1,\ldots,K$, $K > 0$ and $\gamma_k = \gamma = \Theta(\min\{\nicefrac{1}{LA}, \nicefrac{\min\{\ln B_K, \ln C_K, \ln D\}}{\mu(K+1)}\}),$ 
    $
    \lambda_k = \Theta(\nicefrac{\exp(-\gamma\mu(1 + \nicefrac{k}{2}))R}{\gamma A}),
    $ 
    where $A = \ln(\nicefrac{4(K+1)}{\delta})$ and $B_K, C_K, D$ are some parameters dependent on $b = \max_{k=0,1,\ldots,K} b_k$, $\sigma = \max_{k=0,1,\ldots,K} \sigma_k$, $\mu$, $R$, and $K$. Then the final iterate produced by \algname{clipped-SGD} after $K$ iterations with probability at least $1-\delta$ satisfies
    \begin{align*}
        \|x^{K} &- x^*\|^2
        \\&
        = \widetilde\cO \left( \max\left\{ R^2\exp\left(- \frac{\mu K}{L \ln \tfrac{K}{\delta}} \right), \frac{\sigma^2}{\mu K}, \frac{bR}{\mu} \right\} \right).
    \end{align*}
\end{theorem}

Similarly to the convex case, in the case of symmetric noise with bounded $\beta$-th moment for some $\beta > 0$, the above result matches the best-known one for \algname{clipped-SGD} under bounded variance and strong convexity \citep{gorbunov2021near}. In particular, for such noise distributions, condition \eqref{eq:bounded_alpha_moment} is not necessarily satisfied, and even if it is satisfied, our rate $\widetilde\cO(K^{-1})$ is better than the lower bound $\Omega(K^{-\nicefrac{2(\alpha-1)}{\alpha}})$ under condition \eqref{eq:bounded_alpha_moment} for $\alpha \in (1,2)$ \citep{zhang2020adaptive}. However, it is worth mentioning that we do rely on the symmetry of the noise distribution to achieve this rate, while the lower bound holds for any distributions satisfying \eqref{eq:bounded_alpha_moment}.

In the non-symmetric case, we combine Theorem~\ref{thm:clipped_SGD_main_str_cvx} with Lemma~\ref{lem:smom_moments_heavy_tails} and get the following result.

\begin{corollary}[General noise]\label{cor:SGD_biased_heavy_noise_str_cvx}
    Let the assumptions of Theorem~\ref{thm:clipped_SGD_main_str_cvx} hold and for all $x \in \R^d$ the noise $\nu = \nabla f_{\xi}(x) - \nabla f(x)$ satisfies Assumption~\ref{as:convolution}. Then the iterates produced after $K$ iterations of \algname{clipped-SGD} with $\nabla f_{\Xi^k}(x^k)$ being the smoothed median of means of $\cO(n)$ samples $\nabla f_{\xi}(x^k)$ and $\gamma_k = \gamma = \Theta(\min\{\nicefrac{1}{LA}, \nicefrac{\min\{\ln B_K, \ln C_K, \ln D\}}{\mu(K+1)}\}),$  $\lambda_k = \Theta(\nicefrac{\exp(-\gamma\mu(1 + \nicefrac{k}{2}))R}{\gamma A}),$ 
    where $A = \ln(\nicefrac{4(K+1)}{\delta})$ and $B_K, C_K, D$ are some parameters dependent on $b$ and $\sigma$ defined in \eqref{eq:bias_var_general_case}, with probability at least $1-\delta$ satisfy
    \begin{align*}
        \|x^{K} &- x^*\|^2
        = \widetilde\cO\left( R^2\exp\left(- \frac{\mu K}{L \ln \tfrac{K}{\delta}}\right) \right)
        \\&
        + \widetilde\cO\left(\!\max\!\left\{\!\frac{(1 + \theta^2) d + D}{\mu K}, \frac{(1 + \theta) CR}{\theta^2 n\mu}\!\right\}\!\right),
    \end{align*}
    where $n \geq \Omega(\nicefrac{\max_{j\in [d]} M_j}{\theta^2})$ and $C$, $D$ are defined in \eqref{eq:general_case_C_constant}-\eqref{eq:general_case_D_constant}. The overall number of oracle calls equals $\cO(nK)$.
\end{corollary}

Taking $n = \widetilde\cO(\varepsilon^{-1})$ and $K = \widetilde\cO(\varepsilon^{-1})$, one can guarantee $\|x^{K} - x^*\|^2 \leq \varepsilon$ with probability at least $1-\delta$, i.e., the total oracle complexity is $\widetilde\cO(\varepsilon^{-2})$. This result is worse than the one for the case of symmetric distribution, but it still does not require the existence of the expectation of the noise and is better than $\widetilde\cO(\varepsilon^{-\nicefrac{\alpha}{2(\alpha-1)}})$, which is known to be optimal (up to logarithmic factors) under assumption \eqref{eq:bounded_alpha_moment}, when $\alpha < \nicefrac{4}{3}$.

\subsection{Convergence of \algname{clipped-SSTM}}

Next, we consider an accelerated variant of \algname{clipped-SGD} called \algname{clipped-SSTM} \citep{gorbunov2020stochastic, gasnikov2016universal}:
\begin{eqnarray}
    x^{k+1} &=& \frac{A_ky^k + \alpha_{k+1} z^k}{A_{k+1}}, \label{eq:x_SSTM}\\
    z^{k+1} &=& z^k - \alpha_{k+1} \clip(\nabla f_{\Xi^k}(x^{k+1}), \lambda_{k+1}), \label{eq:z_SSTM}\\
    y^{k+1} &=& \frac{A_ky^k + \alpha_{k+1} z^{k+1}}{A_{k+1}}, \label{eq:y_SSTM}
\end{eqnarray}
where $z^0 = y^0 = x^0$, $\alpha_{0} = A_0$, $\alpha_{k+1} = \frac{k+2}{2aL}$ for some parameter $a_{k+1}> 0$, $A_{k+1} = A_k + \alpha_{k+1}$, and $\nabla f_{\Xi^k}(x^{k+1})$ is an estimator satisfying Assumption~\ref{as:bounded_bias_and_variance} sampled independently from previous iterations. Below we formulate the main convergence result for \algname{clipped-SSTM} in the convex case.

\begin{theorem}\label{thm:clipped_SSTM_main_cvx}
    Let Assumptions~\ref{as:L_smoothness} and \ref{as:str_cvx} with $\mu = 0$ hold on $Q = B_{3R}(x^*)$, where $R \geq \|x^0 - x^*\|$. Suppose that $\nabla f_{\Xi^k}(x^{k+1})$ satisfies Assumption~\ref{as:bounded_bias_and_variance} with parameters $b_k, \sigma_k$ for $k = 0,1,\ldots,K$, $K > 0$ and $a = \Theta(\min\{A^2, \nicefrac{\sigma (K+1)^{\nicefrac{3}{2}} \sqrt{A}}{LR}, \nicefrac{b(K+2)^2}{LR}\}),$ 
    $
    \lambda_{k}  = \Theta(\nicefrac{R}{\alpha_{k+1} A}),
    $ 
    where $A = \ln(\nicefrac{4(K+1)}{\delta})$ and $b = \max_{k=0,1,\ldots,K} b_k$, $\sigma = \max_{k=0,1,\ldots,K} \sigma_k$. Then the iterates produced by \algname{clipped-SSTM} after $K$ iterations with probability at least $1-\delta$ satisfy
    \begin{equation}
        f(y^K) - f(x^*) = \widetilde\cO\left(\max\left\{\frac{LR^2}{K^2}, \frac{\sigma R}{\sqrt{K}}, bR\right\}\right). \notag
    \end{equation}
\end{theorem}

As expected for accelerated methods, the above bound has a better $\widetilde\cO(K^{-2})$ ``deterministic'' term in contrast to the $\widetilde\cO(K^{-1})$ corresponding term in the upper bound for \algname{clipped-SGD}. When the noise is symmetric and has bounded $\beta$-th moment for some $\beta > 0$ (not necessarily larger than $1$), then one can construct an estimator with $b = 0$ and finite $\sigma$ (see Proposition~\ref{prop:median_symmetric_case}). In this case, the result matches (up to logarithmic factors) the optimal ones derived under bounded variance \citep{gorbunov2020stochastic} or sub-Gaussian noise \citep{ghadimi2012optimal} assumptions. The improvement of the deterministic part can be utilized when parallel independent computations of the estimator are possible with marginal overheads (e.g., on communications/aggregation of the results of parallel computations).

Finally, for non-symmetric noise distributions satisfying Assumption~\ref{as:convolution}, Theorem~\ref{thm:clipped_SSTM_main_cvx} with Lemma~\ref{lem:smom_moments_heavy_tails} imply the following result.

\begin{corollary}[General noise]\label{cor:SSTM_biased_heavy_noise_str_cvx}
    Let the assumptions of Theorem~\ref{thm:clipped_SSTM_main_cvx} hold and for all $x \in \R^d$ the noise $\nu = \nabla f_{\xi}(x) - \nabla f(x)$ satisfies Assumption~\ref{as:convolution}. Then \algname{clipped-SSTM} with $\nabla f_{\Xi^k}(x^k)$ being the smoothed median of means of $\cO(n)$ samples $\nabla f_{\xi}(x^k)$ and $a = \Theta(\min\{A^2, \nicefrac{\sigma (K+1)^{\nicefrac{3}{2}} \sqrt{A}}{LR}, \nicefrac{b(K+2)^2}{LR}\}),$ $\lambda_{k}  = \Theta(\nicefrac{R}{\alpha_{k+1} A}),$ 
    where $A = \ln(\nicefrac{4(K+1)}{\delta})$ and $b$ and $\sigma$ defined in \eqref{eq:bias_var_general_case}, with probability at least $1-\delta$ after $K$ iterations ensures that $f(y^K) - f(x^*)$ equals
    \begin{align*}
        \widetilde\cO\!\left(\!\max\!\left\{\frac{LR^2}{K^2}, \frac{\sqrt{(1 + \theta^2) d + D} R}{\sqrt{K}}, \frac{(1 + \theta) CR}{\theta^2n}\right\}\!\right)\!,
    \end{align*}
    where $n \geq \Omega(\nicefrac{\max_{j\in [d]} M_j}{\theta^2})$ and $C$, $D$ are defined in \eqref{eq:general_case_C_constant}-\eqref{eq:general_case_D_constant}. The overall number of oracle calls equals $\cO(nK)$.
\end{corollary}

When the non-symmetric part is large, then the same comments are valid as the ones we make after Corollary~\ref{cor:SGD_biased_heavy_noise}. However, when the non-symmetric part is small, then there are regimes when the effect of acceleration is noticeable (for small enough $\theta$).

For the strongly convex problems, we consider a restarted version of \algname{clipped-SSTM}. We provide the results for this method in Appendix~\ref{appendix:restarted_clipped_sstm}.

\section{NUMERICAL EXPERIMENTS}
\label{sec:numerical}

\begin{figure*}[t]
    \includegraphics[width=\textwidth]{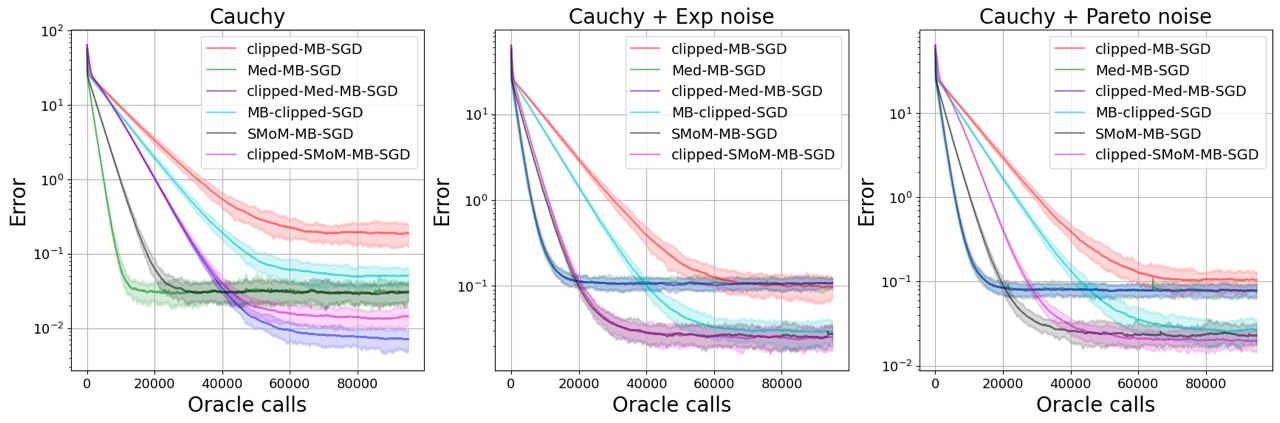}
    \centering
    \caption{Dependence of the mean error on the oracle calls number with a 95th and 5th percentile bounds.}
    \label{fig:errors}
\end{figure*}
In this section, we illustrate the performance of \algname{clipped-SGD} combined with the median and smoothed median of means on a simple quadratic problem:
\begin{equation}
    \min\limits_{x\in \R^d}\frac{1}{2}x^\top \mA x, \label{eq:quadratic_problem}
\end{equation}
where $\mA \in \R^{d\times d}$ is a randomly generated symmetric positive definite matrix. The code of numerical experiments is available on GitHub\footnote{\href{https://github.com/Kutuz4/AISTATS2024_SMoM}{\texttt{https://github.com/Kutuz4/AISTATS2024\_SMoM}}}. We consider stochastic gradients of the form $\nabla f_{\xi}(x) = \mA x + \xi$, where $\xi$ is an artificial noise following one of the following distributions.

\noindent
\emph{Example 1.}\quad Cauchy distribution with the density $\sfp_{c}(x)=\frac{1}{\pi (1 + x ^ 2)}$.

\noindent
\emph{Example 2.}\quad The mixture of Cauchy and exponential distributions with the density $\sfp(x) = 0.7 \cdot \sfp_{c}(x) + 0.3 \cdot e ^ {- (x + 1)} \cdot \1\{x \geq -1\}$.

\noindent
\emph{Example 3.}\quad The mixture of Cauchy and Pareto distributions with the density $\sfp(x) = 0.7 \cdot \sfp_{c}(x) + 0.3 \cdot \frac{3}{(x + 1.5) ^ 4} \cdot \1\{x \geq -1.5\}$.

The experiments check the ability of median and smoothed median of means to deal with symmetric and asymmetrical heavy-tailed noise. We consider two examples of asymmetrical distributions with rapidly (Example 2) and slowly (Example 3) decaying antisymmetric part to examine its influence on the performance of optimization procedures. 

We compare the following baselines:
\begin{itemize}
    \item \algname{clipped-MB-SGD}: \algname{clipped-SGD}, where clipping is taken after mini-batching/averaging;
    \item \algname{MB-clipped-SGD} (mini-batched \algname{clipped-SGD}, where averaging is taken after clipping);
    \item \algname{Med-MB-SGD} (mini-batched \algname{SGD} with median instead of averaging);
    \item \algname{clipped-Med-MB-SGD} (mini-batched \algname{SGD} with median instead of averaging and clipping operation after median);
    \item \algname{SMoM-MB-SGD} (mini-batched \algname{SGD} with the smoothed median of means);
    \item \algname{clipped-SMoM-MB-SGD} (mini-batched \algname{SGD} with clipping of the smoothed median of means).
\end{itemize}

For all methods, except for \algname{SMoM-MB-SGD} and \algname{clipped-SMoM-MB-SGD}, the batch size is 5, while for \algname{SMoM-MB-SGD} and \algname{clipped-SMoM-MB-SGD} we took $\SMM_{m, n}$ with $m = n = 2$. We have chosen $x_0 = 8 / \sqrt{d} \cdot (1, 1, 1, ..., 1)^\top$, where $d = 50$, as a starting point, launched all the methods $50$ times and computed the average errors. The results are displayed in Figure \ref{fig:errors}.

In the case of a symmetric distribution, \algname{Med-MB-SGD} and \algname{clipped-Med-MB-SGD} perform better than  \algname{clipped-SMoM-MB-SGD} due to lower oracle calls count for one iteration. However, as we expected, \algname{Med-MB-SGD} and \algname{clipped-Med-MB-SGD} cannot achieve high accuracy in the case of asymmetric distributions due to the presence of a bias, while the smoothed median of means successfully adapts to this situation. Suddenly, averaging gradients after clipping has good performance in the asymmetric case, but it still converges slower compared to \algname{clipped-SMoM-MB-SGD}. In terms of the number of steps, \algname{clipped-SMoM-MB-SGD} converges much faster than other methods on asymmetric noise because it needs $(2m + 1)n = 10$ oracle calls on each iteration. We also see that \algname{SMoM-MB-SGD} has a similar convergence rate to \algname{clipped-SMoM-MB-SGD} in the case of distributions with less heavy tails.

\section{CONCLUSION}

In this work, we show that under some structural assumptions on the noise distribution with heavy tails, one can achieve faster convergence in solving of stochastic optimization problems. The key instrument we use is the smoothed median of means, which provably has a small bias and a finite variance for quite a wide class of distributions. Although our results are given for smooth convex/strongly convex problems, using similar technique, one can derive high-probability convergence results for smooth non-convex problems \citep{sadiev2023high, nguyen2023high}, non-smooth convex and strongly convex problem \citep{gorbunov2021near}, variational inequalities under some structured non-monotonicity assumptions \citep{gorbunov2022clipped}, and composite and distributed optimization problems \citep{gorbunov2023high}. One can also improve the logarithmic factors in our results using the technique from \citep{nguyen2023improved}.

\subsubsection*{Acknowledgements}

The work of A. Gasnikov was supported by a grant for research centers in the field of artificial
intelligence, provided by the Analytical Center for the Government of the Russian Federation in
accordance with the subsidy agreement (agreement identifier 000000D730321P5Q0002) and the
agreement with the Ivannikov Institute for System Programming of the Russian Academy of Sciences
dated November 2, 2021 No. 70-2021-00142.

\bibliography{refs}
\newpage

\section*{Checklist}

 \begin{enumerate}

 \item For all models and algorithms presented, check if you include:
 \begin{enumerate}
   \item A clear description of the mathematical setting, assumptions, algorithm, and/or model. \textcolor{blue}{Yes, see Sections \ref{sec:setup}, \ref{sec:smoothed_mom} and \ref{sec:main_results}.}
   \item An analysis of the properties and complexity (time, space, sample size) of any algorithm. \textcolor{blue}{Yes, see our results in Sections \ref{sec:smoothed_mom} and \ref{sec:main_results}.}
   \item (Optional) Anonymized source code, with specification of all dependencies, including external libraries. \textcolor{blue}{Yes, we provide the source code with supplementary materials.}
 \end{enumerate}

 \item For any theoretical claim, check if you include:
 \begin{enumerate}
   \item Statements of the full set of assumptions of all theoretical results. \textcolor{blue}{Yes.}
   \item Complete proofs of all theoretical results. \textcolor{blue}{Yes, the proofs are collected in Appendix.}
   \item Clear explanations of any assumptions. \textcolor{blue}{Yes.}   
 \end{enumerate}

 \item For all figures and tables that present empirical results, check if you include:
 \begin{enumerate}
   \item The code, data, and instructions needed to reproduce the main experimental results (either in the supplemental material or as a URL). \textcolor{blue}{Yes.}
   \item All the training details (e.g., data splits, hyperparameters, how they were chosen). \textcolor{blue}{Yes, see Section \ref{sec:numerical} and Appendix.}
    \item A clear definition of the specific measure or statistics and error bars (e.g., with respect to the random seed after running experiments multiple times). \textcolor{blue}{Yes.}
    \item A description of the computing infrastructure used. \textcolor{blue}{Not Applicable. We do not use any computing infrastructure, except for an ordinary laptop.}
 \end{enumerate}

 \item If you are using existing assets (e.g., code, data, models) or curating/releasing new assets, check if you include:
 \begin{enumerate}
   \item Citations of the creator If your work uses existing assets. \textcolor{blue}{Not Applicable.}
   \item The license information of the assets, if applicable. \textcolor{blue}{Not Applicable.}
   \item New assets either in the supplemental material or as a URL, if applicable. \textcolor{blue}{Not Applicable.}
   \item Information about consent from data providers/curators. \textcolor{blue}{Not Applicable.}
   \item Discussion of sensible content if applicable, e.g., personally identifiable information or offensive content. \textcolor{blue}{Not Applicable.}
 \end{enumerate}

 \item If you used crowdsourcing or conducted research with human subjects, check if you include:
 \begin{enumerate}
   \item The full text of instructions given to participants and screenshots. \textcolor{blue}{Not Applicable.}
   \item Descriptions of potential participant risks, with links to Institutional Review Board (IRB) approvals if applicable. \textcolor{blue}{Not Applicable.}
   \item The estimated hourly wage paid to participants and the total amount spent on participant compensation. \textcolor{blue}{Not Applicable.}
 \end{enumerate}

 \end{enumerate}

\newpage

\appendix
\onecolumn

\tableofcontents
\newpage

\section{SMOOTHED MEDIAN OF MEANS ESTIMATOR}

\subsection{Proof of Proposition \ref{prop:median_symmetric_case}}

Let $\sfP_j$ be the cumulative distribution function of $\nu_j$:
\[
    \sfP_j(t) = \int\limits_{-\infty}^t \sfp_j(u) \rmd u,
    \quad \text{for all $t \in \R$.}
\]
Then the probability density of the median $\Med(\nu_{j, 1}, \dots, \nu_{j, 2m + 1})$ is given by
\[
    (2m + 1) \binom{2m}{m} \sfP_j(t)^m \big(1 - \sfP_j(t)\big)^m \sfp_j(t).
\]
Let us prove that the variance is bounded. To be more precise, we are going to show that
\[
    \int\limits_{-\infty}^{+\infty} t^2 \, \sfP_j(t)^m \big(1 - \sfP_j(t)\big)^m \sfp_j(t) \, \rmd t \leq \left(1 \vee \frac{4 B_j}{\beta_j} \right)^{2 / \beta_j} \cdot 4^{-m}.
\]
Since, according to the conditions of the proposition, $\sfp_j(u) \leq B_j / (1 \vee |u|)^{1 + \beta_j}$ for any $u \in \R$, it holds that
\[
    1 - \sfP_j(t)
    \leq \int\limits_t^{+\infty} \frac{B_j}{u^{1 + \beta_j}} \rmd u
    = \frac{B_j}{\beta_j} \cdot \frac1{t^{\beta_j}}
    \quad \text{for any $t \geq 1$.}
\]
Similarly, for any $t \leq -1$, we have
\[
    \sfP_j(t) \leq \frac{B_j}{\beta_j} \cdot \frac1{|t|^{\beta_j}}.
\]
Hence, for any $t \in \R$, $\sfP_j(t) \big( 1 - \sfP_j(t) \big)$ satisfies the inequality
\begin{equation}
    \label{eq:p_tail}
    \sfP_j(t) \big( 1 - \sfP_j(t) \big)
    \leq
    \begin{cases}
        4^{-m} \land \left( \frac{B_j}{\beta_j |t|^{\beta_j}} \right)^m
        \qquad \text{if $|t| \geq 1$},\\
        4^{-m} \qquad \text{otherwise.}
    \end{cases}
\end{equation}
This implies that the integral of $t^2 \, \sfP_j(t)^m (1 - \sfP_j(t))^m \sfp_j(t)$ over the real line does not exceed
\begin{align*}
    \int\limits_{-\infty}^{+\infty} t^2 \, \sfP_j(t)^m \big(1 - \sfP_j(t)\big)^m \sfp_j(t) \, \rmd t
    &
    \leq \sup\limits_{t \in \R} \left\{ t^2 \, \sfP_j(t)^m \big(1 - \sfP_j(t)\big)^m \right\} \cdot \int\limits_{-\infty}^{+\infty} \sfp_j(u) \, \rmd u
    \\&
    = \sup\limits_{t \in \R} \left\{ t^2 \, \sfP_j(t)^m \big(1 - \sfP_j(t)\big)^m \right\}
    \\&
    = \max\left\{\sup\limits_{|t| \leq 1} \left\{ t^2 \, \sfP_j(t)^m \big(1 - \sfP_j(t)\big)^m \right\},  \sup\limits_{|t| \geq 1} \left\{ t^2 \, \sfP_j(t)^m \big(1 - \sfP_j(t)\big)^m \right\} \right\}
    \\&
    \leq \max\left\{4^{-m},  \sup\limits_{|t| \geq 1} \left\{ t^2 \, \sfP_j(t)^m \big(1 - \sfP_j(t)\big)^m \right\}\right\}.
\end{align*}
We use \eqref{eq:p_tail} to bound the supremum in the right-hand side:
\begin{align*}
    &
    \sup\limits_{|t| \geq 1} \left\{ t^2 \, \sfP_j(t)^m \big(1 - \sfP_j(t)\big)^m \right\}
    \leq \sup\limits_{|t| \geq 1} \left\{ t^2 \left( \frac14 \land \frac{B_j}{\beta_j |t|^{\beta_j}} \right)^m \right\}
    \\&
    \leq \max\left\{ \sup\limits_{1 \leq |t| \leq (4B_j / \beta_j)^{1 / \beta_j}} \left\{ t^2 \cdot 4^{-m} \right\}, \sup\limits_{|t| \geq (4B_j / \beta_j)^{1 / \beta_j}} \left\{ t^2 \left( \frac{B_j}{\beta_j |t|^{\beta_j}} \right)^m \right\} \right\}.
\end{align*}
If $m \beta_j \geq 2$, then we have
\[
    \sup\limits_{|t| \geq (4B_j / \beta_j)^{1 / \beta_j}} \left\{ t^2 \left( \frac{B_j}{\beta_j |t|^{\beta_j}} \right)^m \right\}
    = \left( \frac{4 B_j}{\beta_j} \right)^{2 / \beta_j} \cdot 4^{-m},
\]
and, hence,
\[
    \sup\limits_{|t| \geq 1} \left\{ t^2 \, \sfP_j(t)^m \big(1 - \sfP_j(t)\big)^m \right\}
    \leq \sup\limits_{|t| \geq 1} \left\{ t^2 \left( \frac14 \land \frac{B_j}{\beta_j |t|^{\beta_j}} \right)^m \right\}
    \leq \left(1 \vee \frac{4 B_j}{\beta_j} \right)^{2 / \beta_j} \cdot 4^{-m}.
\]
Thus, we obtain that
\begin{align*}
    \int\limits_{-\infty}^{+\infty} t^2 \, \sfP_j(t)^m \big(1 - \sfP_j(t)\big)^m \sfp_j(t) \, \rmd t
    &
    \leq \sup\limits_{t \in \R} \left\{ t^2 \, \sfP_j(t)^m \big(1 - \sfP_j(t)\big)^m \right\}
    \leq \left(1 \vee \frac{4 B_j}{\beta_j} \right)^{2 / \beta_j} \cdot 4^{-m}.
\end{align*}
It only remains to note that
\[
    \binom{2m}{m}
    = \frac{(2m)!}{m! \cdot m!}
    = \prod\limits_{j = 1}^m \frac{2j}{j} \cdot \prod\limits_{j = 1}^m \frac{2j - 1}{j}
    \leq 2^m \cdot 2^m
    = 4^m
\]
to derive the desired bound
\begin{equation}
    \Exp \, \Med\left(\nu_{j, 1}, \dots, \nu_{j, 2m + 1} \right)^2
    \leq (2m + 1) \left(1 \vee \frac{4 B_j}{\beta_j} \right)^{2 / \beta_j}. \label{eq:var_of_median}
\end{equation}
Concerning the expectation of $\Med\left(\nu_{j, 1}, \dots, \nu_{j, 2m + 1} \right)$, we point out that it is finite, because $\Med\left(\nu_{j, 1}, \dots, \nu_{j, 2m + 1} \right)$ has a finite second moment. Moreover, due to the symmetry of $\sfp_j$, we have $\sfP_j(-t) = 1 - \sfP_j(t)$ and, thus,
\[
    (-t) \sfP_j(-t)^m \big(1 - \sfP_j(-t) \big)^m \sfp_j(-t)
    = - t \big(1 - \sfP_j(t) \big)^m \sfP_j(t)^m \sfp_j(t)
    \quad \text{for all $t \in \R$.}
\]
Hence, it holds that
\[
    \Exp \, \Med\left(\nu_{j, 1}, \dots, \nu_{j, 2m + 1} \right)
    = \int\limits_{-\infty}^{+\infty} t \sfP_j(t)^m \big(1 - \sfP_j(t) \big)^m \sfp_j(t) \, \rmd t
    = 0.
\]
\endproof

\subsection{Proof of Lemma \ref{lem:smom_moments_light_tails}}

Let $\Sigma_{11}, \dots, \Sigma_{dd}$ be the diagonal elements of $\Sigma$. Denote the difference $\nabla f_\xi(x) - \nabla f(x)$ by $\nu = (\nu_1, \dots, \nu_d)^\top$.
It is enough to show that
\begin{align*}
    \left| \Exp \, \SMM_{m, n}(\nu_j, \theta) \right|
    &
    \leq (2m + 1) \cdot \frac{\Sigma_{jj}}{\theta^2 n} \sqrt{\theta^2 + \frac{\|\Sigma\|}n} \cdot \left[ \frac{4m}{(2m - 1) \sqrt{2 \pi e}} + \frac{m(2m - 1)}{(2m - 2) \pi e} \cdot \frac{\|\Sigma\|}{\theta^2 n} + 32\left(\frac{m \|\Sigma\|}{\theta^2 n}\right)^2 \right]
    \\&
    \lesssim (2m + 1) \cdot \frac{\Sigma_{jj}}{\theta n} \sqrt{1 + \frac{\|\Sigma\|}{\theta^2 n}} \cdot \left[ 1 + \frac{m \|\Sigma\|}{\theta^2 n} + \left(\frac{m \|\Sigma\|}{\theta^2 n}\right)^2\right]
\end{align*}
and
\[
    \Exp \, \SMM_{m, n}(\nu_j, \theta)^2
    \leq 4 (2m + 1) \left( \frac{\Sigma_{jj}}n + \theta^2 \right).
\]
for all $j \in \{1, \dots, d\}$. We start with the upper bound on the second moment. We split the rest of the proof into several steps for convenience.

\medskip

\noindent\textbf{Step 1: bound on the second moment.}
\quad
For a fixed $j \in \{1, \dots, d\}$, let  $\sfp_j(u)$ be the marginal density of $\nu_j$ and let
\[
    \sfF(t)
    = \int\limits_{-\infty}^{+\infty} \Phi_\theta\left(t - \frac un \right) \sfp_j^{*n}(u) \, \rmd u
\]
stand for the cumulative distribution function of $\Mean(\nu_{j, 1}, \dots, \nu_{j, n}) + \theta \, \eta_j$, where $\nu_{j, 1}, \dots, \nu_{j, n}$ are i.i.d. copies of $\nu_j$.
Then the density of $\SMM_{m, n}(\nu_j, \theta)$ is equal to
\[
    (2m + 1) \binom{2m}{m} \sfF(t)^m \, \big(1 - \sfF(t) \big)^m \, \sfF'(t).
\]
If we manage to prove that
\[
    \sup\limits_{t \in \R} \left\{ t^2 \, \sfF(t)^m \, \big(1 - \sfF(t) \big)^m \right\}
    \leq \frac{4 (\Sigma_{jj} / n + \theta^2)}{4^m},
\]
then we immediately obtain 
\[
    \int\limits_{-\infty}^{+\infty} t^2 \, \sfF(t)^m \, \big(1 - \sfF(t) \big)^m \, \sfF'(t) \, \rmd t
    \leq \sup\limits_{t \in \R} \left\{ t^2 \, \sfF(t)^m \, \big(1 - \sfF(t) \big)^m \right\} \cdot \int\limits_{-\infty}^{+\infty} \sfF'(u) \, \rmd u
    = \sup\limits_{t \in \R} \left\{ t^2 \, \sfF(t)^m \, \big(1 - \sfF(t) \big)^m \right\}.
\]
Let $\nu_{j, 1}, \dots, \nu_{j, n}$ be independent copies of $\nu_j$ and let $\eta_j \sim \cN(0, 1)$ be a Gaussian random variable, which is independent of $\nu_{j, 1}, \dots, \nu_{j, n}$. Then, according to the definition of $\sfF(t)$, for any $t \in \R$, it holds that
\[
    1 - \sfF(t)
    = \Prob\left( \frac{\nu_{j, 1} + \dots + \nu_{j, n}}n + \theta \eta_j \geq t \right).
\]
Since $\nu_{j, 1}, \dots, \nu_{j, n}$ and $\eta_j$ have finite variance, we can apply Chebyshev's inequality to derive an upper bound on the right tail of $\sfF(t)$:
\[
    1 - \sfF(t)
    \leq \frac{\Exp \left(\nu_{j, 1} + \dots + \nu_{j, n} \right)^2 / n^2 + \theta^2 \Exp \eta_j^2}{t^2}
    = \frac{\Sigma_{jj}}{n t^2} + \frac{\theta^2}{t^2}
    \quad \text{for all $t > 0$.}
\]
Similarly, for any $t < 0$, we have
\[
    \sfF(t) \leq \frac{\Sigma_{jj}}{n t^2} + \frac{\theta^2}{t^2}.
\]
Combining these bounds with the inequality $\sfF(t) (1 - \sfF(t)) \leq 1/4$, which holds for any $t \in \R$, we deduce that
\begin{equation}
    \label{eq:cdf_upper_bound}
    \sfF(t) \big(1 - \sfF(t)\big)
    \leq \frac14 \land \left(\frac{\Sigma_{jj}}{n t^2} + \frac{\theta^2}{t^2} \right) =
    \begin{cases}
        \frac14, \quad \text{if $|t| \leq 2 \sqrt{\frac{\Sigma_{jj}}n + \theta^2}$,}\\
        \\
        \frac{\Sigma_{jj}}n + \theta^2, \quad \text{otherwise.}
    \end{cases}
\end{equation}
Hence, for any $m \geq 1$,
\[
    \sup\limits_{t \in \R} \left\{ t^2 \, \sfF(t)^m \, \big(1 - \sfF(t) \big)^m \right\}
    \leq \left( \max\limits_{t^2 \leq 4 \Sigma_{jj} / n + 4 \theta^2} \frac{t^2}{4^m} \right) \lor \left( \max\limits_{t^2 \geq 4 \Sigma_{jj} / n + 4 \theta^2} t^{2} \left( \frac{\Sigma_{jj}}{n t^2} + \frac{\theta^2}{t^2} \right)^{m} \right)
    = \frac{\Sigma_{jj} / n + \theta^2}{4^{m - 1}},
\]
as we announced. This implies that
\[
    \Exp \, \SMM_{m, n}(\nu_j, \theta)^2
    \leq (2m + 1) \binom{2m}{m} \int\limits_{-\infty}^{+\infty} t^2 \, \sfF(t)^m \, \big(1 - \sfF(t) \big)^m \, \sfF'(t) \, \rmd t
    \leq (2m + 1) \binom{2m}{m} \cdot \frac{\Sigma_{jj} / n + \theta^2}{4^{m - 1}}.
\]
Similarly to the proof of Proposition \ref{prop:median_symmetric_case}, we use the inequality
\[
    \binom{2m}{m} = \frac{(2m)!}{m! \cdot m!} = \prod\limits_{j = 1}^m \frac{2j}{j} \cdot \prod\limits_{j = 1}^m \frac{2j - 1}{j} \leq 4^m,
\]
which yields that
\[
    \Exp \, \SMM_{m, n}(\nu_j, \theta)^2
    \leq 4 (2m + 1) \left( \frac{\Sigma_{jj}}n + \theta^2 \right).
\]

\medskip

\noindent
\textbf{Step 2: bound on the expectation.}
\quad
The rest of the proof is devoted to an upper bound on the expectation of $\SMM_{m, n}(\nu_j, \theta)$. One could apply the Cauchy-Schwarz inequality to show that $\Exp \, \SMM_{m, n}(\nu_j, \theta)$ decreases with the growth of $n$. However, we are going to prove a stronger bound. Our approach is based on decomposition of $\sfp_j(u)$ into the sum of symmetric and antisymmetric part:
\[
    \sfp_j(u) = \sfs_j(u) + \sfr_j(u),
    \quad \text{where} \quad
    \sfs_j(u) = \frac{\sfp_j(u) + \sfp_j(-u)}2
    \quad \text{and} \quad
    \sfr_j(u) = \frac{\sfp_j(u) - \sfp_j(-u)}2.
\]
If $\sfr_j$ was equal to zero, we could say that $\Exp \, \SMM_{m, n}(\nu_j, \theta) = 0$ as well. However, in a general situation, has some impact on the mean of $\SMM_{m, n}(\nu_j, \theta)$. To quantify it, we compare the integrals
\[
    \int\limits_{-\infty}^{+\infty} t \, \sfF(t)^m \, \big(1 - \sfF(t) \big)^m \, \sfF'(t) \rmd t
    \quad \text{and} \quad
    \int\limits_{-\infty}^{+\infty} t \, \sfG(t)^m \, \big(1 - \sfG(t) \big)^m \, \sfG'(t) \rmd t,
\]
where $\sfG$ is a cumulative distribution function defined as
\[
    \sfG(t) = \int\limits_{-\infty}^{+\infty} \Phi_\theta\left(t - \frac un \right) \sfs_j^{*n}(u) \, \rmd u.
\]
In other words, $\sfG$ corresponds to the CDF of $\Mean(\nu_{j, 1}, \dots, \nu_{j, n}) + \theta \eta_j$, where $\nu_{j, 1}, \dots, \nu_{j, n}$ are i.i.d. copies of $\nu_j$ and $\eta_j \sim \cN(0, 1)$, in the symmetric case. We are going to show that $\sfr_j$ has minor influence on the expectation of the smoothed median of means if $m$ and $n$ are sufficiently large.

First, let us show that the cumulative distribution functions $\sfF(t)$ and $\sfG(t)$ are close to each other. It is straightforward to check that
\begin{equation}
    \label{eq:phi_second_derivative_bound}
    \sup\limits_{x \in \R} \left| \Phi_\theta''(x) \right|
    = \frac{1}{\sqrt{2 \pi}} \cdot \frac1{\theta^2} \cdot \sup\limits_{x \in \R} \left| -\frac x\theta e^{-x^2 / (2 \theta^2)} \right|
    \leq \frac{1}{\sqrt{2 \pi} \theta^2} \cdot \left.\left( y e^{-y^2 / 2} \right)\right|_{y = 1}
    = \frac{1}{\sqrt{2 \pi e}} \cdot \frac1{\theta^2}.
\end{equation}
Then Lemma \ref{lem:expectation_comparison} (see Appendix \ref{sec:smoothmed_aux} below) implies that
\begin{equation}
    \label{eq:f-g_sup_norm}
    \left| \sfF(t) - \sfG(t) \right| \leq \frac{1}{2 \sqrt{2 \pi e}} \cdot \frac{\Sigma_{jj}}{\theta^2 n}
    \quad \text{for all $t \in \R$.}
\end{equation}
In view of \eqref{eq:cdf_upper_bound}, for any $m \geq 3$ it holds that
\[
    |t| \sfF(t)^m \big(1 - \sfF(t) \big)^m \rightarrow 0
    \quad \text{and} \quad
    |t| \sfF(t)^{m - 1} \big(1 - \sfF(t) \big)^{m - 1} \rightarrow 0
    \quad \text{as $t \rightarrow \infty$.}
\]
Then, according to Lemma \ref{lem:f-g_comparison}, we have
\begin{align*}
    &
    \left| \int\limits_{-\infty}^{+\infty} t \, \sfF(t)^m \big(1 - \sfF(t)\big)^m \sfF'(t) \,\rmd t - \int\limits_{-\infty}^{+\infty} t \, \sfG(t)^m \big(1 - \sfG(t)\big)^m \sfG'(t) \,\rmd t \right|
    \\&
    \leq \frac{1}{2 \sqrt{2 \pi e}} \cdot \frac{\Sigma_{jj}}{\theta^2 n} \int\limits_{-\infty}^{+\infty} \sfF(t)^m \big(1 - \sfF(t)\big)^m \,\rmd t
    + \frac{m}{16 \pi e} \cdot \left(\frac{\Sigma_{jj}}{\theta^2 n}\right)^2 \int\limits_{-\infty}^{+\infty} \sfF(t)^{m - 1} \big(1 - \sfF(t)\big)^{m - 1} \,\rmd t
    \\&\quad
    + \left(\frac{m \Sigma_{jj}}{\theta^2 n}\right)^2 \sup\limits_{t \in \R} \left\{ |t| \max\limits_{\theta \in [\sfF(t)\land \sfG(t), \sfF(t)\lor \sfG(t)]} \left( \theta^{m - 2} (1 - \theta)^{m - 2} \right) \right\}.
\end{align*}
Let us remind the reader that the CDF $\sfF(t)$ satisfies the inequalities
\[
    1 - \sfF(t)
    \leq \frac{\Sigma_{jj}}{n t^2} + \frac{\theta^2}{t^2}
    \quad \text{for all $t > 0$ and} \quad
    \sfF(t) \leq \frac{\Sigma_{jj}}{n t^2} + \frac{\theta^2}{t^2}
    \quad \text{for all $t < 0$}.
\]
Due to the Chebyshev inequality, a similar bound holds for $\sfG(t)$:
\[
    1 - \sfG(t)
    \leq \frac{\Sigma_{jj}}{n t^2} + \frac{\theta^2}{t^2}
    \quad \text{for all $t > 0$ and} \quad
    \sfG(t) \leq \frac{\Sigma_{jj}}{n t^2} + \frac{\theta^2}{t^2}
    \quad \text{for all $t < 0$}.
\]
This yields that
\begin{align*}
    &
    \left| \int\limits_{-\infty}^{+\infty} t \, \sfF(t)^m \big(1 - \sfF(t)\big)^m \sfF'(t) \,\rmd t - \int\limits_{-\infty}^{+\infty} t \, \sfG(t)^m \big(1 - \sfG(t)\big)^m \sfG'(t) \,\rmd t \right|
    \\&
    \leq \frac{1}{2 \sqrt{2 \pi e}} \cdot \frac{\Sigma_{jj}}{\theta^2 n} \int\limits_{-\infty}^{+\infty} \left(\frac14 \land \frac{\Sigma_{jj}/n + \theta^2}{t^2} \right)^m \,\rmd t
    + \frac{m}{16 \pi e} \cdot \left(\frac{\Sigma_{jj}}{\theta^2 n}\right)^2 \int\limits_{-\infty}^{+\infty} \left(\frac14 \land \frac{\Sigma_{jj}/n + \theta^2}{t^2} \right)^{m - 1} \,\rmd t
    \\&\quad
    + \left(\frac{m \Sigma_{jj}}{\theta^2 n}\right)^2 \sup\limits_{t \in \R} \left\{ |t| \left(\frac14 \land \frac{\Sigma_{jj}/n + \theta^2}{t^2} \right)^{m - 2} \right\}.
\end{align*}
Applying Proposition \ref{prop:integral}, we obtain that
\begin{align*}
    &
    \left| \int\limits_{-\infty}^{+\infty} t \, \sfF(t)^m \big(1 - \sfF(t)\big)^m \sfF'(t) \,\rmd t - \int\limits_{-\infty}^{+\infty} t \, \sfG(t)^m \big(1 - \sfG(t)\big)^m \sfG'(t) \,\rmd t \right|
    \\&
    \leq \frac{1}{2 \sqrt{2 \pi e}} \cdot \frac{\Sigma_{jj}}{\theta^2 n} \cdot \frac{2m}{(2m - 1) \cdot 4^{m - 1}} \sqrt{\theta^2 + \frac{\Sigma_{jj}}n} 
    \\&\quad
    + \frac{m}{16 \pi e} \cdot \left(\frac{\Sigma_{jj}}{\theta^2 n}\right)^2 \frac{2m - 1}{(2m - 2) \cdot 4^{m - 2}} \sqrt{\theta^2 + \frac{\Sigma_{jj}}n} 
    \\&\quad
    + \left(\frac{m \Sigma_{jj}}{\theta^2 n}\right)^2 \sup\limits_{t \in \R} \left\{ |t| \left(\frac14 \land \frac{\Sigma_{jj}/n + \theta^2}{t^2} \right)^{m - 2} \right\}.
\end{align*}
Similarly to Step 1, we can prove that
\[
    \sup\limits_{t \in \R} \left\{ |t| \left( \frac14 \land \frac{\theta^2 + \Sigma_{jj} / n}{t^2} \right)^{m - 2} \right\}
    = \frac{2 \sqrt{\theta^2 + \Sigma_{jj} / n}}{4^{m - 2}},
\]
and then it holds that
\begin{align*}
    &
    \left| \int\limits_{-\infty}^{+\infty} t \, \sfF(t)^m \big(1 - \sfF(t)\big)^m \sfF'(t) \,\rmd t - \int\limits_{-\infty}^{+\infty} t \, \sfG(t)^m \big(1 - \sfG(t)\big)^m \sfG'(t) \,\rmd t \right|
    \\&
    \leq \left[ \frac{m}{(2m - 1) \cdot 4^{m - 1} \sqrt{2 \pi e}} \cdot \frac{\Sigma_{jj}}{\theta^2 n} + \frac{m(2m - 1)}{(2m - 2) \cdot 4^m \pi e} \cdot \left(\frac{\Sigma_{jj}}{\theta^2 n}\right)^2 + \frac{32}{4^m} \cdot \left(\frac{m \Sigma_{jj}}{\theta^2 n}\right)^2 \right] \sqrt{\theta^2 + \frac{\Sigma_{jj}}n}.
\end{align*}
Taking into account
\[
    \binom{2m}{m} = \frac{(2m)!}{m! \cdot m!} = \prod\limits_{j = 1}^m \frac{2j}{j} \cdot \prod\limits_{j = 1}^m \frac{2j - 1}{j} \leq 4^m,
\]
we immediately obtain that
\begin{align*}
    \Exp \, \SMM_{m, n}(\nu_j, \theta)
    &
    \leq (2m + 1) \cdot \frac{\Sigma_{jj}}{\theta n} \sqrt{1 + \frac{\Sigma_{jj}}{\theta^2 n}} \cdot \left[ \frac{4m}{(2m - 1) \sqrt{2 \pi e}} + \frac{m(2m - 1)}{(2m - 2) \pi e} \cdot \frac{\Sigma_{jj}}{\theta^2 n} + 32\left(\frac{m \Sigma_{jj}}{\theta^2 n}\right)^2 \right]
    \\&
    \leq (2m + 1) \cdot \frac{\Sigma_{jj}}{\theta n} \sqrt{1 + \frac{\|\Sigma\|}{\theta^2 n}} \cdot \left[ \frac{4m}{(2m - 1) \sqrt{2 \pi e}} + \frac{m(2m - 1)}{(2m - 2) \pi e} \cdot \frac{\|\Sigma\|}{\theta^2 n} + 32\left(\frac{m \|\Sigma\|}{\theta^2 n}\right)^2 \right]
    \\&
    \lesssim (2m + 1) \cdot \frac{\Sigma_{jj}}{\theta n} \sqrt{1 + \frac{\|\Sigma\|}{\theta^2 n}} \cdot \left[ 1 + \frac{m \|\Sigma\|}{\theta^2 n} + \left(\frac{m \|\Sigma\|}{\theta^2 n}\right)^2\right].
\end{align*}

\subsection{Proof of Lemma \ref{lem:smom_moments_heavy_tails}}

Let $\nu = (\nu_1, \dots, \nu_d)^\top$ stand for the difference $\nabla f_\xi(x) - \nabla f(x)$ and let us show that, for any $j \in \{1, \dots, d\}$, it holds that
\begin{align*}
    \left| \Exp \, \SMM_{m, n}(\nu_j, \theta) \right|
    &
    \lesssim \frac{m M_j}{\theta^2 n}  \left[1 + \theta + \left( \frac{2^{\beta_j} B_j}{n^{\beta_j - 1}} \right)^{1 / \beta_j} \right]
\end{align*}
and
\[
    \Exp \, \SMM_{m, n}(\nu_j, \theta)^2
    \lesssim m \left[ 1 + \theta^2 + \left( \frac{M_j}{\theta n} \right)^2 + \left( \frac{2^{\beta_j} B_j}{\beta_j n^{\beta_j - 1}} \right)^{2 / \beta_j} + \left( \frac{B_j M_j}{\theta n^{\beta_j}} \right)^{2 / (\beta_j + 1)} \right].
\]
From now on, we fix an arbitrary $j \in \{1, \dots, d\}$.
Similarly to the proof of Lemma \ref{lem:smom_moments_light_tails}, the core idea is to compare the cumulative distribution functions
\[
    \sfF(t)
    = \int\limits_{-\infty}^{+\infty} \Phi_\theta\left(t - \frac un \right) \sfp_j^{*n}(u) \, \rmd u
    \quad \text{and} \quad
    \sfG(t)
    = \int\limits_{-\infty}^{+\infty} \Phi_\theta\left(t - \frac un \right) \sfs_j^{*n}(u) \, \rmd u.
\]
The first one is directly related to the density of $\SMM_{m, n}(\nu_j, \theta)$, which is equal to
\[
    (2m + 1) \binom{2m}{m} \sfF(t)^m \, \big(1 - \sfF(t) \big)^m \, \sfF'(t).
\]
However, the proof of Lemma \ref{lem:smom_moments_heavy_tails} is far more technical. The main obstacle is that we cannot use Chebyshev's inequality to specify the rate of decay of $\sfF(t)\big(1 - \sfF(t)\big)$ and of $\sfG(t)\big(1 - \sfG(t)\big)$ as $t$ approaches infinity. Instead, we prove the following non-trivial result (see Lemma \ref{lem:f-g_non-uniform_bound} below): if Assumption \ref{as:convolution} holds and $2 M_j \leq n \theta^2$, then, for any $t \in \R$ it holds that
\[
    \left| \sfF(t) - \sfG(t) \right|
    \lesssim \frac{M_j}{\theta n t} \left(1 + \frac{\theta}t + \frac{B_j}{n^{\beta_j - 1} |t|^{\beta_j}} \right).
\]
Combining this result with the bound on the second derivative of $\Phi_\theta$ \eqref{eq:phi_second_derivative_bound} and Lemma \ref{lem:expectation_comparison}, we obtain that
\begin{equation}
    \label{eq:f-g_non-uniform_bound}
    \left| \sfF(t) - \sfG(t) \right|
    \lesssim \frac{M_j}{\theta^2 n} \left\{1 \land \left(\frac{\theta}t + \frac{B_j \theta}{n^{\beta_j - 1} |t|^{\beta_j + 1}}\right) \right\}.
\end{equation}
Despite the simple statement, the proof of Lemma \ref{lem:f-g_non-uniform_bound} is quite technical. A reader can find it in Appendix \ref{sec:smoothmed_aux}. 
With the bound \eqref{eq:f-g_non-uniform_bound} at hand, the proof of Lemma \ref{lem:smom_moments_heavy_tails} is relatively simple. For convenience, we divide it into several steps.

\medskip

\noindent\textbf{Step 1: a bound on the tails of $\sfG$.}
\quad
The goal of this step is to specify the rate of decay of $\sfG(t) \big(1 - \sfG(t) \big)$ as $t$ tends to infinity.
First, consider the case $t > 0$. By the definition of $\sfG(t)$, it holds that
\begin{align}
    1 - \sfG(t)
    &
    = \int\limits_{-\infty}^{+\infty} \left(1 - \Phi_\theta\left( t - \frac{u}n \right) \right) \sfs_j^{*n}(u) \, \rmd u \notag
    \\&
    = n \int\limits_{-\infty}^{+\infty} \big(1 - \Phi_\theta(-y) \big) \sfs_j^{*n}(nt + ny) \, \rmd y \notag
    \\&
    = n \int\limits_{-\infty}^{-t/2} \big(1 - \Phi_\theta(-y) \big) \sfs_j^{*n}(nt + ny) \, \rmd y
    + n \int\limits_{-t/2}^{+\infty} \big(1 - \Phi_\theta(-y) \big) \sfs_j^{*n}(nt + ny) \, \rmd y. \label{eq:g_tail}
\end{align}
If $y \leq -t / 2$, then
\[
    1 - \Phi_\theta(-y)
    \leq 1 - \Phi_\theta(t / 2)
    \leq \exp\left\{- \frac{t^2}{8 \theta^2} \right\},
\]
and we have
\begin{equation}
    \label{eq:left_interval}
    n \int\limits_{-\infty}^{-t/2} \big(1 - \Phi_\theta(-y) \big) \sfs_j^{*n}(nt + ny) \, \rmd y
    \leq \exp\left\{- \frac{t^2}{8 \theta^2} \right\} \int\limits_{-\infty}^{-t/2} \sfs_j^{*n}(nt + ny) \, n \, \rmd y
    \leq \exp\left\{- \frac{t^2}{8 \theta^2} \right\}.
\end{equation}
Otherwise, if $y > -t / 2$, then, due to Assumption \ref{as:convolution}, it holds that
\begin{align}
    \label{eq:right_interval}
    n \int\limits_{-t/2}^{+\infty} \big(1 - \Phi_\theta(-y) \big) \sfs_j^{*n}(nt + ny) \, \rmd y
    &\notag
    \leq \int\limits_{-t/2}^{+\infty} \frac{B_j n^2}{n^{(1 + \beta_j) / \beta_j} + n^{1 + \beta_j} (t + y)^{1 + \beta_j}} \, \rmd y
    \\&
    \leq \int\limits_{0}^{+\infty} \frac{B_j}{n^{\beta_j - 1} (t/2 + v)^{1 + \beta_j}} \, \rmd v
    = \frac{2^{\beta_j} B_j}{\beta_j n^{\beta_j - 1} t^{\beta_j}}.
\end{align}
Plugging the inequalities \eqref{eq:left_interval}, \eqref{eq:right_interval} into \eqref{eq:g_tail}, we obtain that
\[
    1 - \sfG(t) \lesssim \frac{2^{\beta_j} B_j}{\beta n^{\beta_j - 1} t^{\beta_j}} + \exp\left\{- \frac{t^2}{8 \theta^2} \right\}
    \quad \text{for all $t > 0$.}
\]
Similarly, we can prove that
\[
    \sfG(t) \lesssim \frac{2^{\beta_j} B_j}{\beta_j n^{\beta_j - 1} |t|^{\beta_j}} + \exp\left\{- \frac{t^2}{8 \theta^2} \right\}
    \quad \text{for all $t < 0$.}
\]
Hence, for any $t \in \R$, we have
\begin{equation}
    \label{eq:g_tail_inequality}
    \sfG(t) \big( 1 - \sfG(t) \big)
    \lesssim \min\left\{ \frac14, \frac{2^{\beta_j} B_j}{\beta_j n^{\beta_j - 1} |t|^{\beta_j}} + \exp\left\{- \frac{t^2}{8 \theta^2} \right\} \right\}.
\end{equation}

\medskip

\noindent\textbf{Step 2: bound on the second moment.}
\quad
The second moment of $\SMM_{m, n}(\nu_j, \theta)$ satisfies the inequality
\begin{align*}
    \Exp \, \SMM_{m, n}(\nu_j, \theta)^2
    &
    = (2m + 1) \binom{2m}{m} \int\limits_{-\infty}^{+\infty} t^2 \sfF(t)^m \, \big(1 - \sfF(t) \big)^m \, \sfF'(t) \, \rmd t
    \\&
    \leq (2m + 1) \binom{2m}{m} \sup\limits_{t \in \R} \left\{ t^2 \, \sfF(t)^m \, \big(1 - \sfF(t) \big)^m \right\}.
\end{align*}
On the other hand, \eqref{eq:f-g_non-uniform_bound} and \eqref{eq:g_tail_inequality} imply that
\begin{equation}
    \label{eq:f_tail_inequality}
    \sfF(t) \big( 1 - \sfF(t) \big)
    \lesssim \min\left\{ \frac14, \frac{2^{\beta_j} B_j}{\beta_j n^{\beta_j - 1} |t|^{\beta_j}} + \exp\left\{- \frac{t^2}{8 \theta^2} \right\} + \frac{M_j}{\theta n t} \left(1 + \frac{B_j}{n^{\beta_j - 1} |t|^{\beta_j}} \right) \right\}
    \quad \text{for all $t \in \R$.}
\end{equation}
This yields
\[
    \sup\limits_{t \in \R} \big\{ t^2 \, \sfF(t)^m \, \big(1 - \sfF(t) \big)^m \big\}
    \lesssim 4^{-m} \left[ 1 + \theta^2 + \left( \frac{M_j}{\theta n} \right)^2 + \left( \frac{2^{\beta_j} B_j}{\beta_j n^{\beta_j - 1}} \right)^{2 / \beta_j} + \left( \frac{B_j M_j}{\theta n^{\beta_j}} \right)^{2 / (\beta_j + 1)} \right]
    \quad \text{if $m\beta_j > 2$}.
\]
Since
\[
    \binom{2m}{m}
    = \frac{(2m)!}{m! \cdot m!}
    = \prod\limits_{j = 1}^m \frac{2j}{j} \cdot \prod\limits_{j = 1}^m \frac{2j - 1}{j}
    \leq 4^m,
\]
we obtain that
\[
    \Exp \, \SMM_{m, n}(\nu_j, \theta)^2
    \lesssim m \left[ 1 + \theta^2 + \left( \frac{M_j}{\theta n} \right)^2 + \left( \frac{2^{\beta_j} B_j}{\beta_j n^{\beta_j - 1}} \right)^{2 / \beta_j} + \left( \frac{B_j M_j}{\theta n^{\beta_j}} \right)^{2 / (\beta_j + 1)} \right].
\]

\medskip

\noindent\textbf{Step 3: bound on the expectation.}
\quad
It remains to bound the expectation of $\SMM_{m, n}(\nu_j, \theta)$. For this purpose, we use Lemma \ref{lem:f-g_comparison}, which yields that
\begin{align*}
    &
    \left( (2m + 1) \binom{2m}{m} \right)^{-1} \left| \Exp \SMM_{m, n}(\nu_j, \theta) \right|
    \\&
    = \left| \int\limits_{-\infty}^{+\infty} t \, \sfF(t)^m \big(1 - \sfF(t)\big)^m \sfF'(t) \,\rmd t - \int\limits_{-\infty}^{+\infty} t \, \sfG(t)^m \big(1 - \sfG(t)\big)^m \sfG'(t) \,\rmd t \right|
    \\&
    \leq \int\limits_{-\infty}^{+\infty} \sfG(t)^m \big(1 - \sfG(t)\big)^m \big| \sfF(t) - \sfG(t)\big| \,\rmd t
    + \frac{m}2 \int\limits_{-\infty}^{+\infty} \sfG(t)^{m - 1} \big(1 - \sfG(t)\big)^{m - 1} \big( \sfF(t) - \sfG(t)\big)^2 \,\rmd t
    \\&\quad
    + m^2 \sup\limits_{t \in \R} \left\{ |t| \big(\sfF(t) - \sfG(t) \big)^2 \cdot \max\limits_{\theta \in [\sfF(t)\land \sfG(t), \sfF(t)\lor \sfG(t)]} \left( \theta^{m - 2} (1 - \theta)^{m - 2} \right) \right\}.
\end{align*}
Note that the requirement 
\[
    |t| \sfG(t)^m \big(1 - \sfG(t) \big)^m \big| \sfF(t) - \sfG(t) \big| \rightarrow 0
    \quad \text{and} \quad
    |t| \sfG(t)^{m - 1} \big(1 - \sfG(t) \big)^{m - 1} \big( \sfF(t) - \sfG(t) \big)^2 \rightarrow 0
    \quad \text{$t \rightarrow \infty$}
\]
is satisfied, because of the inequalities \eqref{eq:f-g_non-uniform_bound}, \eqref{eq:g_tail_inequality} and the conditions of the lemma. These inequalities also imply that
\begin{align*}
    &
    \left| \int\limits_{-\infty}^{+\infty} t \, \sfF(t)^m \big(1 - \sfF(t)\big)^m \sfF'(t) \,\rmd t - \int\limits_{-\infty}^{+\infty} t \, \sfG(t)^m \big(1 - \sfG(t)\big)^m \sfG'(t) \,\rmd t \right|
    \\&
    \lesssim \int\limits_{-\infty}^{+\infty} \left[ \frac14 \land \left( \frac{2^{\beta_j} B}{\beta_j n^{\beta_j - 1} |t|^{\beta_j}} + \exp\left\{- \frac{t^2}{8 \theta^2} \right\} \right) \right]^m \cdot \big|\sfF(t) - \sfG(t) \big| \,\rmd t
    \\&\quad
    + \frac{m}2 \int\limits_{-\infty}^{+\infty} \left[ \frac14 \land \left( \frac{2^{\beta_j} B_j}{\beta_j n^{\beta_j - 1} |t|^{\beta_j}} + \exp\left\{- \frac{t^2}{8 \theta^2} \right\} \right) \right]^{m - 1} \big(\sfF(t) - \sfG(t) \big)^2 \,\rmd t
    \\&\quad
    + m^2 \sup\limits_{t \in \R} \left\{ |t| \big(\sfF(t) - \sfG(t) \big)^2 \left[ \frac14 \land \left( \frac{2^{\beta_j} B_j}{\beta_j n^{\beta_j - 1} |t|^{\beta_j}} + \frac{M_j}{\theta n t} \left(1 + \frac{B_j}{n^{\beta_j - 1} |t|^{\beta_j}} \right) + \exp\left\{- \frac{t^2}{8 \theta^2} \right\} \right) \right]^{m - 2} \right\}.
\end{align*}
Taking into account the bound \eqref{eq:f-g_non-uniform_bound} on the absolute value of the difference $\sfF(t) - \sfG(t)$, we obtain that
\begin{align*}
    &
    \left| \int\limits_{-\infty}^{+\infty} t \, \sfF(t)^m \big(1 - \sfF(t)\big)^m \sfF'(t) \,\rmd t - \int\limits_{-\infty}^{+\infty} t \, \sfG(t)^m \big(1 - \sfG(t)\big)^m \sfG'(t) \,\rmd t \right|
    \\&
    \lesssim \frac{M_j}{\theta^2 n} \int\limits_{-\infty}^{+\infty} \left[ \frac14 \land \left( \frac{2^{\beta_j} B_j}{\beta_j n^{\beta_j - 1} |t|^{\beta_j}} + \exp\left\{- \frac{t^2}{8 \theta^2} \right\} \right) \right]^m \,\rmd t
    \\&\quad
    + \frac{m}2 \left( \frac{M_j}{\theta^2 n} \right)^2 \int\limits_{-\infty}^{+\infty} \left[ \frac14 \land \left( \frac{2^{\beta_j} B_j}{\beta_j n^{\beta_j - 1} |t|^{\beta_j}} + \exp\left\{- \frac{t^2}{8 \theta^2} \right\} \right) \right]^{m - 1} \,\rmd t
    \\&\quad
    + m^2 \left( \frac{M_j}{\theta^2 n} \right)^2 \sup\limits_{t \in \R} \left\{ |t| \left[ \frac14 \land \left( \frac{2^{\beta_j} B_j}{\beta_j n^{\beta_j - 1} |t|^{\beta_j}} + \frac{M_j}{\theta n t} \left(1 + \frac{B_j}{n^{\beta_j - 1} |t|^{\beta_j}} \right) + \exp\left\{- \frac{t^2}{8 \theta^2} \right\} \right) \right]^{m - 2} \right\}.
\end{align*}
The inequality $(a + b)^k \leq 2^{k - 1}(a^k + b^k)$, which holds for any $k \geq 1$ and positive $a$ and $b$,  implies that
\begin{align*}
    &
    \left| \int\limits_{-\infty}^{+\infty} t \, \sfF(t)^m \big(1 - \sfF(t)\big)^m \sfF'(t) \,\rmd t - \int\limits_{-\infty}^{+\infty} t \, \sfG(t)^m \big(1 - \sfG(t)\big)^m \sfG'(t) \,\rmd t \right|
    \\&
    \lesssim \frac{M_j}{\theta^2 n} \int\limits_{-\infty}^{+\infty} \left[ \frac14 \land \frac{2^{\beta_j + 1} B_j}{\beta_j n^{\beta_j - 1} |t|^{\beta_j}} \right]^m \,\rmd t + \frac{M_j}{\theta^2 n} \int\limits_{-\infty}^{+\infty} \left[ \frac14 \land 2\exp\left\{- \frac{t^2}{8 \theta^2} \right\} \right]^m \,\rmd t
    \\&\quad
    + \frac{m}2 \left( \frac{M_j}{\theta^2 n} \right)^2 \int\limits_{-\infty}^{+\infty} \left[ \frac14 \land \frac{2^{\beta_j + 1} B_j}{\beta_j n^{\beta_j - 1} |t|^{\beta_j}} \right]^{m - 1} \,\rmd t + \frac{m}2 \left( \frac{M_j}{\theta^2 n} \right)^2 \int\limits_{-\infty}^{+\infty} \left[ \frac14 \land 2 \exp\left\{- \frac{t^2}{8 \theta^2} \right\} \right]^{m - 1} \,\rmd t
    \\&\quad
    + m^2 \left( \frac{M_j}{\theta^2 n} \right)^2 \sup\limits_{t \in \R} \left\{ |t| \left[ \frac14 \land \frac{2^{\beta_j + 1} B_j}{\beta_j n^{\beta_j - 1} |t|^{\beta_j}} \right]^{m - 2} \right\}
    + m^2 \left( \frac{M_j}{\theta^2 n} \right)^2 \sup\limits_{t \in \R} \left\{ |t| \left[ \frac14 \land \frac{4 M_j}{\theta n |t|} \right]^{m - 2} \right\}
    \\&\quad
    + m^2 \left( \frac{M_j}{\theta^2 n} \right)^2 \sup\limits_{t \in \R} \left\{ |t| \left[ \frac14 \land \frac{8 M_j B_j}{\theta n^{\beta_j} |t|^{\beta_j + 1}} \right]^{m - 2} \right\}
    + m^2 \left( \frac{M_j}{\theta^2 n} \right)^2 \sup\limits_{t \in \R} \left\{ |t| \left[ \frac14 \land 8 \exp\left\{- \frac{t^2}{8 \theta^2} \right\} \right]^{m - 2} \right\}.
\end{align*}
Due to Proposition \ref{prop:integral}, it holds that
\[
    \int\limits_{-\infty}^{+\infty} \left[ \frac14 \land \frac{2^{\beta_j + 1} B_j}{\beta n^{\beta_j - 1} |t|^{\beta_j}} \right]^m \,\rmd t + \int\limits_{-\infty}^{+\infty} \left[ \frac14 \land 2\exp\left\{- \frac{t^2}{8 \theta^2} \right\} \right]^m \,\rmd t \lesssim 4^{-m} \left[1 + \theta + \left( \frac{2^{\beta_j + 1} B_j}{\beta_j n^{\beta_j - 1}} \right)^{1 / \beta_j} \right].
\]
Since
\begin{align*}
    &
    \sup\limits_{t \in \R} \left\{ |t| \left[ \frac14 \land \frac{2^{\beta_j + 1} B_j}{\beta_j n^{\beta_j - 1} |t|^{\beta_j}} \right]^{m - 2} \right\}
    + \sup\limits_{t \in \R} \left\{ |t| \left[ \frac14 \land \frac{4 M_j}{\theta n t} \right]^{m - 2} \right\}
    \\&\quad
    + \sup\limits_{t \in \R} \left\{ |t| \left[ \frac14 \land \frac{8 M_j B_j}{\theta n^{\beta_j} |t|^{\beta_j + 1}} \right]^{m - 2} \right\}
    + \sup\limits_{t \in \R} \left\{ |t| \left[ \frac14 \land 4 \exp\left\{- \frac{t^2}{8 \theta^2} \right\} \right]^{m - 2} \right\}
    \\&
    \lesssim 4^{-m} \left[ 1 + \theta + \frac{M_j}{\theta n} + \left( \frac{M_j B_j}{\theta n^{\beta_j}} \right)^{1 / (\beta_j + 1)} + \left( \frac{2^{\beta_j + 1} B_j}{\beta_j n^{\beta_j - 1}} \right)^{1 / \beta_j} \right],
\end{align*}
we conclude that
\begin{align*}
    &
    \left| \int\limits_{-\infty}^{+\infty} t \, \sfF(t)^m \big(1 - \sfF(t)\big)^m \sfF'(t) \,\rmd t - \int\limits_{-\infty}^{+\infty} t \, \sfG(t)^m \big(1 - \sfG(t)\big)^m \sfG'(t) \,\rmd t \right|
    \\&
    \lesssim 4^{-m} \cdot \frac{M_j}{\theta^2 n}  \left[1 + \theta + \left( \frac{2^{\beta_j + 1} B_j}{\beta_j n^{\beta_j - 1}} \right)^{1 / \beta_j} \right]
    \\&\quad
    + m^2 \cdot 4^{-m} \cdot \left( \frac{M_j}{\theta^2 n} \right)^2 \left[1 + \theta + \frac{M_j}{\theta n} + \left( \frac{M_j B_j}{\theta n^{\beta_j}} \right)^{1 / (\beta_j + 1)} + \left( \frac{2^{\beta_j + 1} B_j}{\beta_j n^{\beta_j - 1}} \right)^{1 / \beta_j} \right].
\end{align*}
Then it holds that
\begin{align*}
    \left| \Exp \SMM_{m, n}(\nu_j, \theta) \right|
    &
    \lesssim (2m + 1) \binom{2m}{m} \cdot 4^{-m} \cdot \frac{M_j}{\theta^2 n}  \left[1 + \theta + \left( \frac{2^{\beta_j + 1} B_j}{\beta_j n^{\beta_j - 1}} \right)^{1 / \beta_j} \right]
    \\&\quad
    + (2m + 1) \binom{2m}{m} \cdot m^2 4^{-m} \cdot \left( \frac{M_j}{\theta^2 n} \right)^2 \left[1 + \theta + \frac{M_j}{\theta n} + \left( \frac{M_j B_j}{\theta n^{\beta_j}} \right)^{1 / (\beta_j + 1)} + \left( \frac{2^{\beta_j + 1} B_j}{\beta_j n^{\beta_j - 1}} \right)^{1 / \beta_j} \right]
    \\&
    \lesssim \frac{m M_j}{\theta^2 n}  \left[1 + \theta + \left( \frac{2^{\beta_j} B_j}{n^{\beta_j - 1}} \right)^{1 / \beta_j} \right].
\end{align*}
The proof is finished.

\subsection{Technical results}
\label{sec:smoothmed_aux}

\begin{lemma}
    \label{lem:expectation_comparison}
    Let $h : \R \rightarrow \R$ be a twice differentiable function with uniformly bounded second derivative:
    \[
        |h''(x)| \leq H \quad \text{for all $x \in \R$.}
    \]
    Let $\sfq(x)$ be a probability density of a random variable and let
    \[
        \rho(x) = \frac{\sfq(x) + \sfq(-x)}2
        \quad \text{and} \quad
        \psi(x) = \frac{\sfq(x) - \sfq(-x)}2
    \]
    stand for its symmetric and antisymmetric parts, respectively.
    Assume that there exists $M > 0$ such that
    the function $\psi(x)$ fulfils
    \[
        \int\limits_{-\infty}^{+\infty} x \psi(x) \rmd x = 0
        \quad \text{and} \quad
        \int\limits_{-\infty}^{+\infty} x^2 |\psi(x)| \rmd x \leq M.
    \]
    Then, for any positive integer $n$, it holds that
    \[
        \left| \int h\left(\frac xn\right) \sfq^{*n}(x) \rmd x
        - \int h\left(\frac xn\right) \rho^{*n}(x) \rmd x \right|
        \leq \frac{H M}{2 n},
    \]
    provided that the integrals in the left-hand side converge.
\end{lemma}

\begin{proof}
    Let $X_1, \dots, X_n$ be i.i.d. random variables with the density $\sfq(x)$. It is known that $(X_1 + \ldots + X_n) \sim \sfq^{*n}(x)$. Then the integral 
    \[
        \int h\left(\frac xn\right) \sfq^{*n}(x) \rmd x
    \]
    admits a representation
    \[
        \int h\left(\frac xn\right) \sfq^{*n}(x) \rmd x
        = \Exp h\left(\frac{X_1 + \ldots + X_n}n \right)
        = \int h\left( \frac{x_1 + \ldots + x_n}n \right) \sfq(x_1) \dots \sfq(x_n) \rmd x_1 \dots \rmd x_n.
    \]
    Similarly, it holds that
    \[
        \int h\left(\frac xn\right) \rho^{*n}(x) \rmd x 
        = \int h\left( \frac{x_1 + \ldots + x_n}n \right) \rho(x_1) \dots \rho(x_n) \rmd x_1 \dots \rmd x_n,
    \]
    and thus,
    \begin{align*}
        \int h\left(\frac xn\right) \sfq^{*n}(x) \rmd x
        - \int h\left(\frac xn\right) \rho^{*n}(x) \rmd x
        &
        = \int h\left( \frac{x_1 + \ldots + x_n}n \right) \sfq(x_1) \dots \sfq(x_n) \rmd x_1 \dots \rmd x_n
        \\&\quad
        - \int h\left( \frac{x_1 + \ldots + x_n}n \right) \rho(x_1) \dots \rho(x_n) \rmd x_1 \dots \rmd x_n.
    \end{align*}
    Let us introduce
    \[
        \pi_k(x_1, \dots, x_n) = \prod\limits_{i = 1}^k \sfq(x_i) \prod\limits_{i = k + 1}^n \rho(x_i), 
        \quad \text{where $k \in \{0, \dots, n\}$}.
    \]
    Then it holds that
    \begin{align*}
        &
        \int h\left( \frac{x_1 + \ldots + x_n}n \right) \sfq(x_1) \dots \sfp(x_n) \rmd x_1 \dots \rmd x_n
        \\&\quad
        - \int h\left( \frac{x_1 + \ldots + x_n}n \right) \rho(x_1) \dots \rho(x_n) \rmd x_1 \dots \rmd x_n
        \\&
        = \int h\left( \frac{x_1 + \ldots + x_n}n \right) \pi_n(x_1, \dots, x_n) \rmd x_1 \dots \rmd x_n
        \\&\quad
        - \int h\left( \frac{x_1 + \ldots + x_n}n \right) \pi_0(x_1, \dots, x_n) \rmd x_1 \dots \rmd x_n
        \\&
        = \sum\limits_{k = 1}^n \int h\left( \frac{x_1 + \ldots + x_n}n \right) \big( \pi_k(x_1, \dots, x_n) - \pi_{k-1} (x_1, \dots, x_n) \big) \rmd x_1 \dots \rmd x_n.
    \end{align*}
    Let us fix any $k \in \{1, \dots, n\}$ and consider
    \[
        \int h\left( \frac{x_1 + \ldots + x_n}n \right) \big( \pi_k(x_1, \dots, x_n) - \pi_{k-1} (x_1, \dots, x_n) \big) \rmd x_1 \dots \rmd x_n.
    \]
    Note that, according to the definition of $\pi_k$, we have
    \[
        \pi_k(x_1, \dots, x_n) - \pi_{k-1} (x_1, \dots, x_n)
        = \left( \prod\limits_{i = 1}^{k-1} \sfq(x_i) \right) \left( \prod\limits_{i = k + 1}^n \rho(x_i) \right) \psi(x_k).
    \]
    Moreover, due to the Taylor's expansion with the Lagrange remainder term, it holds that
    \[
        \left| h\left( \frac{x_1 + \ldots + x_n}n \right)
        - h\left( \frac1n \sum\limits_{i \neq k} x_i \right) - h' \left( \frac1n \sum\limits_{i \neq k} x_i \right) \cdot \frac{x_k}n \right|
        \leq \frac{H x_k^2}{2n^2}.
    \]
    Since, according to the definition of $\psi$ and the conditions of the lemma,
    \[
        \int\limits_{-\infty}^{+\infty} \psi(x_k) \rmd x_k = 0,
        \quad
        \int\limits_{-\infty}^{+\infty} x_k \psi(x_k) \rmd x_k = 0,
        \quad \text{and}
        \int\limits_{-\infty}^{+\infty} x_k^2 |\psi(x_k)| \rmd x_k \leq M^2,
    \]
    we have
    \begin{align*}
        &
        \int h\left( \frac1n \sum\limits_{i \neq k} x_i \right) \big( \pi_k(x_1, \dots, x_n) - \pi_{k-1} (x_1, \dots, x_n) \big) \rmd x_1 \dots \rmd x_n = 0,
        \\&
        \int h' \left( \frac1n \sum\limits_{i \neq k} x_i \right) \cdot \frac{x_k}n \big( \pi_k(x_1, \dots, x_n) - \pi_{k-1} (x_1, \dots, x_n) \big) \rmd x_1 \dots \rmd x_n = 0,
    \end{align*}
    and then
    \begin{align*}
        &
        \left| \int h\left( \frac{x_1 + \ldots + x_n}n \right) \big( \pi_k(x_1, \dots, x_n) - \pi_{k-1} (x_1, \dots, x_n) \big) \rmd x_1 \dots \rmd x_n \right|
        \\&
        \leq \frac{H}{2n^2} \int x_k^2 \left( \prod\limits_{i = 1}^{k-1} \sfq(x_i) \right) \left( \prod\limits_{i = k + 1}^n \rho(x_i) \right) |\psi(x_k)| \rmd x_1 \dots \rmd x_n
        \leq \frac{H M}{2n^2}.
    \end{align*}
    Finally, applying the triangle inequality, we obtain that
    \begin{align*}
        &
        \Bigg| \int h\left( \frac{x_1 + \ldots + x_n}n \right) \sfq(x_1) \dots \sfq(x_n) \rmd x_1 \dots \rmd x_n
        \\&\quad
        - \int h\left( \frac{x_1 + \ldots + x_n}n \right) \rho(x_1) \dots \rho(x_n) \rmd x_1 \dots \rmd x_n \Bigg|
        \\&
        \leq \sum\limits_{k = 1}^n \left| \int h\left( \frac{x_1 + \ldots + x_n}n \right) \big( \pi_k(x_1, \dots, x_n) - \pi_{k-1} (x_1, \dots, x_n) \big) \rmd x_1 \dots \rmd x_n \right|
        \\&
        \leq \sum\limits_{k = 1}^n \frac{H M}{2n^2}
        = \frac{H M}{2n}.
    \end{align*}
\end{proof}

\begin{lemma}
    \label{lem:f-g_comparison}
    Let $\sfF$ and $\sfG$ be any differentiable cumulative distribution functions, such that
    \[
        |t| \sfG(t)^m \big(1 - \sfG(t) \big)^m \big| \sfF(t) - \sfG(t) \big| \rightarrow 0
        \quad \text{and} \quad
        |t| \sfG(t)^{m - 1} \big(1 - \sfG(t) \big)^{m - 1} \big( \sfF(t) - \sfG(t) \big)^2 \rightarrow 0
        \quad \text{as $t \rightarrow \infty$.}
    \]
    Then it holds that
    \begin{align*}
        &
        \left| \int\limits_{-\infty}^{+\infty} t \, \sfF(t)^m \big(1 - \sfF(t)\big)^m \sfF'(t) \,\rmd t - \int\limits_{-\infty}^{+\infty} t \, \sfG(t)^m \big(1 - \sfG(t)\big)^m \sfG'(t) \,\rmd t \right|
        \\&
        \leq \int\limits_{-\infty}^{+\infty} \sfG(t)^m \big(1 - \sfG(t)\big)^m \big| \sfF(t) - \sfG(t)\big| \,\rmd t
        + \frac{m}2 \int\limits_{-\infty}^{+\infty} \sfG(t)^{m - 1} \big(1 - \sfG(t)\big)^{m - 1} \big( \sfF(t) - \sfG(t)\big)^2 \,\rmd t
        \\&\quad
        + m^2 \sup\limits_{t \in \R} \left\{ |t| \big(\sfF(t) - \sfG(t) \big)^2 \cdot \max\limits_{\theta \in [\sfF(t)\land \sfG(t), \sfF(t)\lor \sfG(t)]} \left( \theta^{m - 2} (1 - \theta)^{m - 2} \right) \right\}.
    \end{align*}
\end{lemma}

\begin{proof}
    The proof is based on integration by parts. Let us define a function $\Psi : [0, 1] \rightarrow \R$ as follows:
    \[
        \Psi(x)
        = \frac1{m^2} \left( x^m (1 - x)^m \right)^{''}
        = \frac1m x^{m - 2} (1 - x)^{m - 2} \left( (m - 1)(1 - 2x)^2 - 2 x(1 - x) \right).
    \]
    Note that, for any $x \in [0, 1]$, we have
    \begin{equation}
        \label{eq:psi_bound}
        -\frac1{2m} x^{m - 2} (1 - x)^{m - 2} \leq \Psi(x) \leq \left(1 - \frac1m \right) x^{m - 2} (1 - x)^{m - 2}.
    \end{equation}
    Due to Taylor's expansion with the Lagrange remainder term, for any $t \in \R$, there exists $\theta(t) \in \big(\sfF(t) \land \sfG(t), \sfF(t) \lor \sfG(t) \big)$, such that
    \begin{align*}
        \sfF(t)^m \big(1 - \sfF(t)\big)^m
        - \sfG(t)^m \big(1 - \sfG(t)\big)^m
        &
        = m \sfG(t)^{m - 1} \big(1 - \sfG(t)\big)^{m - 1} \big(1 - 2 \sfG(t) \big) \big(\sfF(t) - \sfG(t) \big)
        \\&\quad
        + \frac{m^2}2 \Psi(\theta(t)) \big(\sfF(t) - \sfG(t) \big)^2.
    \end{align*}
    Then it holds that
    \begin{align}
        &\notag
        \int\limits_{-\infty}^{+\infty} t \, \sfF(t)^m \big(1 - \sfF(t)\big)^m \sfF'(t) \,\rmd t - \int\limits_{-\infty}^{+\infty} t \, \sfG(t)^m \big(1 - \sfG(t)\big)^m \sfG'(t) \,\rmd t
        \\&\notag
        = \int\limits_{-\infty}^{+\infty} t \, \sfG(t)^m \big(1 - \sfG(t)\big)^m \big( \sfF'(t) - \sfG'(t)\big) \,\rmd t
        \\&\quad
        + m \int\limits_{-\infty}^{+\infty} t \, \sfG(t)^{m - 1} \big(1 - \sfG(t)\big)^{m - 1} \big(1 - 2 \sfG(t) \big) \big(\sfF(t) - \sfG(t) \big) \sfF'(t) \,\rmd t  \label{eq:taylors_expansion_1}
        \\&\quad\notag
        + \frac{m^2}2 \int\limits_{-\infty}^{+\infty} t \, \Psi(\theta(t)) \big(\sfF(t) - \sfG(t) \big)^2 \sfF'(t) \,\rmd t.
    \end{align}
    Let us focus on the first term in the right-hand side. Integration by parts yields that
    \begin{align*}
        \int\limits_{-\infty}^{+\infty} t \, \sfG(t)^m \big(1 - \sfG(t)\big)^m \big( \sfF'(t) - \sfG'(t)\big) \,\rmd t
        &
        = - \int\limits_{-\infty}^{+\infty} \sfG(t)^m \big(1 - \sfG(t)\big)^m \big( \sfF(t) - \sfG(t)\big) \,\rmd t
        \\&\quad
        - m \int\limits_{-\infty}^{+\infty} t \, \sfG(t)^{m - 1} \big(1 - \sfG(t)\big)^{m - 1} \big(1 - 2 \sfG(t) \big) \big( \sfF(t) - \sfG(t)\big) \sfG'(t) \,\rmd t.
    \end{align*}
    Substituting this equality into \eqref{eq:taylors_expansion_1}, we obtain that
    \begin{align}
        \label{eq:integration_by_parts_1}
        &\notag
        \int\limits_{-\infty}^{+\infty} t \, \sfF(t)^m \big(1 - \sfF(t)\big)^m \sfF'(t) \,\rmd t - \int\limits_{-\infty}^{+\infty} t \, \sfG(t)^m \big(1 - \sfG(t)\big)^m \sfG'(t) \,\rmd t
        \\&\notag
        = - \int\limits_{-\infty}^{+\infty} \sfG(t)^m \big(1 - \sfG(t)\big)^m \big( \sfF(t) - \sfG(t)\big) \,\rmd t
        \\&\quad
        + m \int\limits_{-\infty}^{+\infty} t \, \sfG(t)^{m - 1} \big(1 - \sfG(t)\big)^{m - 1} \big(1 - 2 \sfG(t) \big) \big( \sfF(t) - \sfG(t)\big) \big(\sfF'(t) - \sfG'(t)\big) \,\rmd t
        \\&\quad\notag
        + \frac{m^2}2 \int\limits_{-\infty}^{+\infty} t \, \Psi(\theta(t)) \big(\sfF(t) - \sfG(t) \big)^2 \sfF'(t) \,\rmd t.
    \end{align}
    Now we can apply integration by parts to the second term in the right-hand side of \eqref{eq:integration_by_parts_1}:
    \begin{align*}
        &
        m \int\limits_{-\infty}^{+\infty} t \, \sfG(t)^{m - 1} \big(1 - \sfG(t)\big)^{m - 1} \big(1 - 2 \sfG(t) \big) \big( \sfF(t) - \sfG(t)\big) \big(\sfF'(t) - \sfG'(t)\big) \,\rmd t
        \\&
        = -\frac{m}2 \int\limits_{-\infty}^{+\infty} \sfG(t)^{m - 1} \big(1 - \sfG(t)\big)^{m - 1} \big(1 - 2 \sfG(t) \big) \big( \sfF(t) - \sfG(t)\big)^2 \,\rmd t
        \\&\quad
        -\frac{m^2}2 \int\limits_{-\infty}^{+\infty} t \, \Psi\big(\sfG(t) \big) \big( \sfF(t) - \sfG(t)\big)^2 \sfG'(t) \,\rmd t.
    \end{align*}
    Hence, we can rewrite the equality \eqref{eq:integration_by_parts_1} in the following form:
    \begin{align*}
        &
        \int\limits_{-\infty}^{+\infty} t \, \sfF(t)^m \big(1 - \sfF(t)\big)^m \sfF'(t) \,\rmd t - \int\limits_{-\infty}^{+\infty} t \, \sfG(t)^m \big(1 - \sfG(t)\big)^m \sfG'(t) \,\rmd t
        \\&
        = - \int\limits_{-\infty}^{+\infty} \sfG(t)^m \big(1 - \sfG(t)\big)^m \big( \sfF(t) - \sfG(t)\big) \,\rmd t
        \\&\quad
        -\frac{m}2 \int\limits_{-\infty}^{+\infty} \sfG(t)^{m - 1} \big(1 - \sfG(t)\big)^{m - 1} \big(1 - 2 \sfG(t) \big) \big( \sfF(t) - \sfG(t)\big)^2 \,\rmd t
        \\&\quad
        + \frac{m^2}2 \int\limits_{-\infty}^{+\infty} t \left[ \Psi\big(\theta(t) \big) \sfF'(t) - \Psi\big(\sfG(t) \big) \sfG'(t) \right] \big( \sfF(t) - \sfG(t)\big)^2 \,\rmd t.
    \end{align*}
    Then, due to \eqref{eq:psi_bound} and the triangle inequality, it holds that
    \begin{align*}
        &
        \left| \int\limits_{-\infty}^{+\infty} t \, \sfF(t)^m \big(1 - \sfF(t)\big)^m \sfF'(t) \,\rmd t - \int\limits_{-\infty}^{+\infty} t \, \sfG(t)^m \big(1 - \sfG(t)\big)^m \sfG'(t) \,\rmd t \right|
        \\&
        \leq \int\limits_{-\infty}^{+\infty} \sfG(t)^m \big(1 - \sfG(t)\big)^m \big( \sfF(t) - \sfG(t)\big) \,\rmd t
        + \frac{m}2 \int\limits_{-\infty}^{+\infty} \sfG(t)^{m - 1} \big(1 - \sfG(t)\big)^{m - 1} \big|1 - 2 \sfG(t) \big| \big( \sfF(t) - \sfG(t)\big)^2 \,\rmd t
        \\&\quad
        + \frac{m^2}2 \int\limits_{-\infty}^{+\infty} |t| \left| \Psi\big(\theta(t) \big) \right|  \big( \sfF(t) - \sfG(t)\big)^2 \sfF'(t) \,\rmd t
        + \frac{m^2}2 \int\limits_{-\infty}^{+\infty} |t| \left| \Psi\big(\sfG(t) \big) \right|  \big( \sfF(t) - \sfG(t)\big)^2 \sfG'(t) \,\rmd t.
    \end{align*}
    Since $\big|1 - 2 \sfG(t) \big| \leq 1$,
    \[
        \int\limits_{-\infty}^{+\infty} |t| \left| \Psi\big(\theta(t) \big) \right|  \big( \sfF(t) - \sfG(t)\big)^2 \sfF'(t) \,\rmd t
        \leq \sup\limits_{t \in \R} \left\{ |t| \left| \Psi\big(\theta(t) \big) \right|  \big( \sfF(t) - \sfG(t)\big)^2 \right\},
    \]
    and, similarly,
    \[
        \int\limits_{-\infty}^{+\infty} |t| \left| \Psi\big(\sfG(t) \big) \right|  \big( \sfF(t) - \sfG(t)\big)^2 \sfG'(t) \,\rmd t
        \leq \sup\limits_{t \in \R} \left\{ |t| \left| \Psi\big(\sfG(t) \big) \right|  \big( \sfF(t) - \sfG(t)\big)^2 \right\},
    \]
    we finally obtain that
    \begin{align*}
        &
        \left| \int\limits_{-\infty}^{+\infty} t \, \sfF(t)^m \big(1 - \sfF(t)\big)^m \sfF'(t) \,\rmd t - \int\limits_{-\infty}^{+\infty} t \, \sfG(t)^m \big(1 - \sfG(t)\big)^m \sfG'(t) \,\rmd t \right|
        \\&
        \leq \int\limits_{-\infty}^{+\infty} \sfG(t)^m \big(1 - \sfG(t)\big)^m \big| \sfF(t) - \sfG(t)\big| \,\rmd t
        + \frac{m}2 \int\limits_{-\infty}^{+\infty} \sfG(t)^{m - 1} \big(1 - \sfG(t)\big)^{m - 1} \big( \sfF(t) - \sfG(t)\big)^2 \,\rmd t
        \\&\quad
        + m^2 \sup\limits_{t \in \R} \left\{ |t| \big(\sfF(t) - \sfG(t) \big)^2 \cdot \max\limits_{\theta \in [\sfF(t)\land \sfG(t), \sfF(t)\lor \sfG(t)]} \left( \theta^{m - 2} (1 - \theta)^{m - 2} \right) \right\}.
    \end{align*}
\end{proof}

\begin{proposition}
    \label{prop:integral}
    For any positive numbers $a, k, \alpha$, and $\beta$, such that $\beta k > \alpha + 1$, it holds that
    \[
        \int\limits_{-\infty}^{+\infty} |t|^{\alpha} \left( \frac14 \land \frac{a^\beta}{t^\beta} \right)^k \rmd t
        = \frac{\beta k}{(\alpha + 1)(\beta k - \alpha - 1)} \cdot \frac{2 (4^{1 / \beta} a)^{\alpha + 1}}{4^k}.
    \]
\end{proposition}

\begin{proof}
    The proof follows from simple calculations:
    \begin{align*}
        \int\limits_{-\infty}^{+\infty} |t|^{\alpha} \left( \frac14 \land \frac{a^\beta}{t^\beta} \right)^k \rmd t
        &
        = 2 \int\limits_0^{+\infty} t^{\alpha} \left( \frac14 \land \frac{a^\beta}{t^\beta} \right)^k \rmd t
        \\&
        = \frac2{4^k} \int\limits_0^{4^{1/\beta} a} t^{\alpha} \, \rmd t
        + 2 a^{\beta k} \int\limits_{4^{1/\beta} a}^{+\infty} t^{\alpha - \beta k} \, \rmd t
        \\&
        = \frac{2 (4^{1 / \beta} a)^{\alpha + 1}}{(\alpha + 1) \cdot 4^k}
        + \frac{2 a^{\beta k} \big(4^{1/\beta} a \big)^{\alpha - \beta k + 1}}{(\beta k - \alpha - 1) }
        \\&
        = \frac{2 (4^{1 / \beta} a)^{\alpha + 1}}{4^k} \left( \frac1{\alpha + 1} + \frac{1}{\beta k - \alpha - 1} \right)
        \\&
        = \frac{\beta k}{(\alpha + 1)(\beta k - \alpha - 1)} \cdot \frac{2 (4^{1 / \beta} a)^{\alpha + 1}}{4^k}.
    \end{align*}
\end{proof}

\begin{lemma}
    \label{lem:f-g_non-uniform_bound}
    Grant Assumption \ref{as:convolution} and let
    \[
        \sfF(t)
        = \int\limits_{-\infty}^{+\infty} \Phi_\theta\left(t - \frac{u}n \right) \sfp_j^{*n}(u) \, \rmd u
        = \int\limits_{\R^n} \Phi_\theta \big(t - \Mean(u_1, \dots, u_n) \big) \sfp_j(u_1) \dots \sfp_j(u_n) \, \rmd u
    \]
    and
    \[
        \sfG(t)
        = \int\limits_{-\infty}^{+\infty} \Phi_\theta\left(t - \frac{u}n \right) \sfs_j^{*n}(u) \, \rmd u
        = \int\limits_{\R^n} \Phi_\theta \big(t - \Mean(u_1, \dots, u_n) \big) \sfs_j(u_1) \dots \sfs_j(u_n) \, \rmd u
    \]
    be two cumulative distribution functions. Assume that $2 M_j \leq n \theta^2$. Then, for any $t \in \R$ it holds that
    \[
        \left| \sfF(t) - \sfG(t) \right| \lesssim \frac{M_j}{\theta n t} \left(1 + \frac{\theta}t + \frac{B_j}{n^{\beta_j - 1} |t|^{\beta_j}} \right).
    \]
\end{lemma}

\begin{proof}
    Since $\sfp_j(u) = \sfs_j(u) + \sfr_j(u)$, it holds that
    \[
        \sfp_j^{*n}
        = \big( \sfs_j(u) + \sfr_j(u) \big)^{*n}
        = \sum\limits_{k = 0}^n \binom{n}{k} \, \sfr_j^{*k} * \sfs_j^{*(n - k)},
    \]
    and then
    \begin{align}
        \label{eq:sum_integrals}
        \sfF(t) - \sfG(t)
        &\notag
        = \int\limits_{-\infty}^{+\infty} \Phi_\theta\left(t - \frac{u}n \right) \sfp_j^{*n}(u) \, \rmd u - \int\limits_{-\infty}^{+\infty} \Phi_\theta\left(t - \frac{u}n \right) \sfs_j^{*n}(u) \, \rmd u
        \\&
        = \sum\limits_{k = 1}^n \binom{n}{k} \int\limits_{-\infty}^{+\infty} \Phi_\theta\left(t - \frac{u}n \right) \sfr_j^{*k} * \sfs_j^{*(n - k)}(u) \, \rmd u
        \\&\notag
        = \sum\limits_{k = 1}^n \binom{n}{k} \int\limits_{\R^n} \Phi_\theta \big(t - \Mean(u_1, \dots, u_n) \big) \left( \prod\limits_{i = 1}^k \sfr_j(u_i) \right) \cdot \left( \prod\limits_{i = k + 1}^n \sfs_j(u_i) \right) \, \rmd u_1 \dots \rmd u_n.
    \end{align}
    In the rest of the proof, we bound the summands in the right-hand side one by one. For readability, we split our derivations into several steps. 

    \medskip

    \noindent
    \textbf{Step 1: Taylor's expansion.}
    \quad
    Let us fix an arbitrary $k \in \{1, \dots, n\}$ and consider
    \[
        \int\limits_{\R^n} \Phi_\theta \big(t - \Mean(u_1, \dots, u_n) \big) \left( \prod\limits_{i = 1}^k \sfr_j(u_i) \right) \cdot \left( \prod\limits_{i = k + 1}^n \sfs_j(u_i) \right) \, \rmd u_1 \dots \rmd u_n.
    \]
    Using Taylor's expansion with the integral remainder term, we rewrite the expression of interest in the form
    \begin{align}
        \label{eq:taylors_expansion}
        &\notag
        \int\limits_{\R^{n}} \Phi_\theta \big(t - \Mean(u_1, \dots, u_n) \big) \left( \prod\limits_{i = 1}^k \sfr_j(u_i) \right) \cdot \left( \prod\limits_{i = k + 1}^n \sfs_j(u_i) \right) \, \rmd u_1 \dots \rmd u_n
        \\&
        = \int\limits_{\R^{n}} \Bigg[ \Phi_\theta \left(t - \Mean(u_1, \dots, u_n) + \frac{u_1}n \right)
        - \Phi'_\theta \left(t - \Mean(u_1, \dots, u_n) + \frac{u_1}n \right) \cdot \frac{u_1}n
        \\&\quad\notag
        + \frac{u_1^2}{n^2} \int\limits_0^1 \Phi''_\theta \left(t - \Mean(u_1, \dots, u_n) + \frac{(1 - v_1) u_1}n \right) v_1 \, \rmd v_1 \Bigg] \left( \prod\limits_{i = 1}^k \sfr_j(u_i) \right) \cdot \left( \prod\limits_{i = k + 1}^n \sfs_j(u_i) \right) \, \rmd u_1 \dots \rmd u_n.
    \end{align}
    Note that
    \[
        \Phi_\theta \left(t - \Mean(u_1, \dots, u_n) + \frac{u_1}n \right)
        \quad \text{and} \quad
        \Phi'_\theta \left(t - \Mean(u_1, \dots, u_n) + \frac{u_1}n \right)
    \]
    do not depend on $u_1$.  Since, according to Assumption \ref{as:convolution}, it holds that 
    \[
        \int\limits_{-\infty}^{+\infty} \sfr_j(u_1) \, \rmd u_1 = 0
        \quad \text{and} \quad 
        \int\limits_{-\infty}^{+\infty} u_1 \sfr_j(u_1) \, \rmd u_1 = 0,
    \]
    the right-hand side of \eqref{eq:taylors_expansion} simplifies to
    \begin{align*}
        &
        \int\limits_{\R^{n}} \Phi_\theta \big(t - \Mean(u_1, \dots, u_n) \big) \left( \prod\limits_{i = 1}^k \sfr_j(u_i) \right) \cdot \left( \prod\limits_{i = k + 1}^n \sfs_j(u_i) \right) \, \rmd u_1 \dots \rmd u_n
        \\&
        = \int\limits_0^1 v_1 \rmd v_1 \int\limits_{\R^{n}} \rmd u_1 \dots \rmd u_n \cdot \Phi''_\theta \left(t - \Mean(u_1, \dots, u_n) + \frac{(1 - v_1) u_1}n \right) \cdot \frac{u_1^2}{n^2} \cdot \left( \prod\limits_{i = 1}^k \sfr_j(u_i) \right) \cdot \left( \prod\limits_{i = k + 1}^n \sfs_j(u_i) \right).
    \end{align*}
    Repeating this trick $(k - 1)$ more times, we obtain that
    \begin{align*}
        &
        \int\limits_{\R^{n}} \Phi_\theta \big(t - \Mean(u_1, \dots, u_n) \big) \left( \prod\limits_{i = 1}^k \sfr_j(u_i) \right) \cdot \left( \prod\limits_{i = k + 1}^n \sfs_j(u_i) \right) \, \rmd u_1 \dots \rmd u_n
        \\&
        = \int\limits_0^1 v_1 \rmd v_1 \ldots \int\limits_0^1 v_k \rmd v_k \int\limits_{\R^{n}} \rmd u_1 \dots \rmd u_n
        \\&\quad
        \cdot \Phi^{(2k)}_\theta \left(t - \Mean(u_1, \dots, u_n) + \sum\limits_{i = 1}^k \frac{(1 - v_i) u_i}n \right) \left( \prod\limits_{i = 1}^k \frac{u_i^2 \, \sfr_j(u_i)}{n^2} \right) \left( \prod\limits_{i = k + 1}^n \sfs_j(u_i) \right),
    \end{align*}
    where $\Phi^{(2k)}_\theta$ stands for the $(2k)$-th derivative of $\Phi_\theta$.
    Due to the properties of convolution, the integral in the right-hand side is equal to
    \[
        \int\limits_0^1 v_1 \rmd v_1 \ldots \int\limits_0^1 v_k \rmd v_k \int\limits_{\R^{n}} \rmd u_1 \dots \rmd u_n \cdot \Phi^{(2k)}_\theta \left(t - \sum\limits_{i = 1}^k \frac{v_i u_i}n - \sum\limits_{i = k + 1}^n \frac{u_i}n \right) \left( \prod\limits_{i = 1}^k \frac{u_i^2 \, \sfr_j(u_i)}{n^2} \right) \cdot \sfs_j^{*(n - k)}\left( \sum\limits_{i = k + 1}^n u_i \right).
    \]
    Substituting $u_{k + 1} + \ldots + u_n$ by $y$, we conclude that
    \begin{align}
        \label{eq:integral_expansion}
        &\notag
        \int\limits_{\R^{n}} \Phi_\theta \big(t - \Mean(u_1, \dots, u_n) \big) \left( \prod\limits_{i = 1}^k \sfr_j(u_i) \right) \cdot \left( \prod\limits_{i = k + 1}^n \sfs_j(u_i) \right) \, \rmd u_1 \dots \rmd u_n
        \\&
        = \int\limits_0^1 v_1 \rmd v_1 \ldots \int\limits_0^1 v_k \rmd v_k \int\limits_{\R^{k}} \rmd u_1 \dots \rmd u_k \int\limits_{-\infty}^{+\infty} \rmd y \cdot \Phi^{(2k)}_\theta \left(t - \sum\limits_{i = 1}^k \frac{v_i u_i}n - \frac{y}n \right) \left( \prod\limits_{i = 1}^k \frac{u_i^2 \, \sfr_j(u_i)}{n^2} \right) \cdot \sfs_j^{*(n - k)}(y).
    \end{align}

    \medskip

    \noindent\textbf{Step 2: bound on the integral \eqref{eq:integral_expansion}.}
    \quad
    We represent the expression \eqref{eq:integral_expansion} as a sum of two terms:
    \begin{align*}
        &
        \int\limits_0^1 v_1 \rmd v_1 \ldots \int\limits_0^1 v_k \rmd v_k \int\limits_{\R^{k}} \rmd u_1 \dots \rmd u_k \int\limits_{-\infty}^{+\infty} \rmd y
        \cdot \Phi^{(2k)}_\theta \left(t - \sum\limits_{i = 1}^k \frac{v_i u_i}n - \frac{y}n \right) \left( \prod\limits_{i = 1}^k \frac{u_i^2 \, \sfr_j(u_i)}{n^2} \right) \cdot \sfs_j^{*(n - k)}(y)
        \\&
        = \int\limits_0^1 v_1 \rmd v_1 \ldots \int\limits_0^1 v_k \rmd v_k \int\limits_{\left[-\frac{nt}{2k}, \frac{nt}{2k}\right]^k} \rmd u_1 \dots \rmd u_k \int\limits_{-\infty}^{+\infty} \rmd y
        \cdot \Phi^{(2k)}_\theta \left(t - \sum\limits_{i = 1}^k \frac{v_i u_i}n - \frac{y}n \right) \left( \prod\limits_{i = 1}^k \frac{u_i^2 \, \sfr_j(u_i)}{n^2} \right) \cdot \sfs_j^{*(n - k)}(y)
        \\&\quad
        + \int\limits_0^1 v_1 \rmd v_1 \ldots \int\limits_0^1 v_k \rmd v_k \int\limits_{\R^k \left\backslash \left[-\frac{nt}{2k}, \frac{nt}{2k}\right]^k \right.} \rmd u_1 \dots \rmd u_k \int\limits_{-\infty}^{+\infty} \rmd y
        \cdot \Phi^{(2k)}_\theta \left(t - \sum\limits_{i = 1}^k \frac{v_i u_i}n - \frac{y}n \right) \left( \prod\limits_{i = 1}^k \frac{u_i^2 \, \sfr_j(u_i)}{n^2} \right) \cdot \sfs_j^{*(n - k)}(y),
    \end{align*}
    and bound the former and the latter summands in the right-hand side separately. According to Lemma \ref{lem:inner_integral} and \ref{lem:outer_integral}, it holds that
    \begin{align*}
        &
        \left| \int\limits_0^1 v_1 \rmd v_1 \ldots \int\limits_0^1 v_k \rmd v_k \int\limits_{\left[-\frac{nt}{2k}, \frac{nt}{2k}\right]^k} \rmd u_1 \dots \rmd u_k \int\limits_{-\infty}^{+\infty} \rmd y \cdot \Phi^{(2k)}_\theta \left(t - \sum\limits_{i = 1}^k \frac{v_i u_i}n - \frac{y}n \right) \left( \prod\limits_{i = 1}^k \frac{u_i^2 \, \sfr_j(u_i)}{n^2} \right) \cdot \sfs_j^{*(n - k)}(y) \right|
        \\&
        \leq \frac{\sqrt{(2k - 1)!}}{\sqrt{2\pi}}
        \cdot \left( \frac{2^{3 + 2\beta_j} B_j \theta \sqrt{\pi}}{n^{\beta_j - 1} |t|^{1 + \beta_j}} + \exp\left\{- \frac{t^2}{64 \theta^2}\right\} \right) \cdot \left( \frac{M_j}{2 \theta^2 n^2} \right)^k.
    \end{align*}
    and
    \begin{align*}
        &
        \left| \int\limits_0^1 v_1 \rmd v_1 \ldots \int\limits_0^1 v_k \rmd v_k \int\limits_{\R^k \left\backslash \left[-\frac{nt}{2k}, \frac{nt}{2k}\right]^k \right.} \rmd u_1 \dots \rmd u_k \int\limits_{-\infty}^{+\infty} \rmd y \cdot \Phi^{(2k)}_\theta \left(t - \sum\limits_{i = 1}^k \frac{v_i u_i}n - \frac{y}n \right) \left( \prod\limits_{i = 1}^k \frac{u_i^2 \, \sfr_j(u_i)}{n^2} \right) \cdot \sfs_j^{*(n - k)}(y) \right|
        \\&
        \leq \sqrt{(2k - 2)!} \left( \frac{M_j}{2 \theta^2 n^2} \right)^{k} \left( \frac{4 k^2 \theta^2}{t^2} + \frac{1}{\sqrt{2\pi}} \cdot \frac{2k \theta}{t} \right).
    \end{align*}
    Hence, we obtain that
    \begin{align}
        \label{eq:integral_derivative_bound}
        & \notag
        \left| \int\limits_0^1 v_1 \rmd v_1 \ldots \int\limits_0^1 v_k \rmd v_k \int\limits_{\R^{k}} \rmd u_1 \dots \rmd u_k \int\limits_{-\infty}^{+\infty} \rmd y \cdot \Phi^{(2k)}_\theta \left(t - \sum\limits_{i = 1}^k \frac{v_i u_i}n - \frac{y}n \right) \left( \prod\limits_{i = 1}^k \frac{u_i^2 \, \sfr_j(u_i)}{n^2} \right) \cdot \sfs_j^{*(n - k)}(y) \right|
        \\&
        \leq \frac{\sqrt{(2k - 1)!}}{\sqrt{2\pi}}
        \cdot \left( \frac{2^{3 + 2\beta_j} B_j \theta \sqrt{\pi}}{n^{\beta_j - 1} |t|^{1 + \beta_j}} + \exp\left\{- \frac{t^2}{64 \theta^2}\right\} \right) \cdot \left( \frac{M_j}{2 \theta^2 n^2} \right)^k
        \\&\quad \notag
        + \sqrt{(2k - 2)!} \left( \frac{M_j}{2 \theta^2 n^2} \right)^{k} \left( \frac{4 k^2 \theta^2}{t^2} + \frac{1}{\sqrt{2\pi}} \cdot \frac{2k \theta}{t} \right).
    \end{align}

    \medskip

    \noindent\textbf{Step 3: final bound.}\quad Summing up the equalities \eqref{eq:sum_integrals}, \eqref{eq:integral_expansion}, and the equality \eqref{eq:integral_derivative_bound}, we obtain that
    \begin{align*}
        \big| \sfF(t) - \sfG(t) \big|
        &
        \leq \sum\limits_{k = 1}^n \binom{n}{k} \frac{\sqrt{(2k - 1)!}}{\sqrt{2\pi}}
        \cdot \left( \frac{2^{3 + 2\beta_j} B_j \theta \sqrt{\pi}}{n^{\beta_j - 1} |t|^{1 + \beta_j}} + \exp\left\{- \frac{t^2}{64 \theta^2}\right\} \right) \cdot \left( \frac{M_j}{2 \theta^2 n^2} \right)^k
        \\&\quad
        + \sum\limits_{k = 1}^n \binom{n}{k} \sqrt{(2k - 2)!}  \left( \frac{4 k^2 \theta^2}{t^2} + \frac{1}{\sqrt{2\pi}} \cdot \frac{2k \theta}{t} \right) \cdot \left( \frac{M_j}{2 \theta^2 n^2} \right)^{k}
        \\&
        \lesssim \sum\limits_{k = 1}^n \binom{n}{k} \sqrt{(2k)!} \left( \frac{B_j \theta}{n^{\beta_j - 1} |t|^{1 + \beta_j}} + \exp\left\{- \frac{t^2}{64 \theta^2}\right\} + \frac{k \theta^2}{t^2} + \frac{\theta}{t}\right) \cdot \left( \frac{M_j}{2 \theta^2 n^2} \right)^{k}.
    \end{align*}
    Finally, Lemma \ref{lem:sums} implies that
    \begin{align*}
        \big| \sfF(t) - \sfG(t) \big|
        &
        \lesssim \sum\limits_{k = 1}^n \binom{n}{k} \sqrt{(2k)!} \left( \frac{B_j \theta}{n^{\beta_j - 1} |t|^{1 + \beta_j}} + \exp\left\{- \frac{t^2}{64 \theta^2}\right\} + \frac{k \theta^2}{t^2} + \frac{\theta}{t}\right) \cdot \left( \frac{M_j}{2 \theta^2 n^2} \right)^{k}
        \\&
        \leq \left( \frac{B_j \theta}{n^{\beta_j - 1} |t|^{1 + \beta_j}} + \exp\left\{- \frac{t^2}{64 \theta^2}\right\} + \frac{2 \theta^2}{t^2} + \frac{\theta}{t}\right) \cdot \frac{2 M_j}{\theta^2 n}
        \\&
        \lesssim \frac{M_j}{\theta n t} \left(1 + \frac{\theta}t + \frac{B_j}{n^{\beta_j - 1} |t|^{\beta_j}} \right),
    \end{align*}
    whenever $2 M_j \leq n \theta^2$. The proof is finished.
    
\end{proof}

\begin{lemma}
    \label{lem:inner_integral}
    Under Assumption \ref{as:convolution}, it holds that
    \begin{align*}
        &
        \left| \int\limits_0^1 v_1 \rmd v_1 \ldots \int\limits_0^1 v_k \rmd v_k \int\limits_{\left[-\frac{nt}{2k}, \frac{nt}{2k}\right]^k} \rmd u_1 \dots \rmd u_k \int\limits_{-\infty}^{+\infty} \rmd y \cdot \Phi^{(2k)}_\theta \left(t - \sum\limits_{i = 1}^k \frac{v_i u_i}n - \frac{y}n \right) \left( \prod\limits_{i = 1}^k \frac{u_i^2 \, \sfr_j(u_i)}{n^2} \right) \cdot \sfs_j^{*(n - k)}(y) \right|
        \\&
        \leq \frac{\sqrt{(2k - 1)!}}{\sqrt{2\pi}}
        \cdot \left( \frac{2^{3 + 2\beta_j} B_j \theta \sqrt{\pi}}{n^{\beta_j - 1} |t|^{1 + \beta_j}} + \exp\left\{- \frac{t^2}{64 \theta^2}\right\} \right) \cdot \left( \frac{M_j}{2 \theta^2 n^2} \right)^k.
    \end{align*}
\end{lemma}

\begin{proof}
    To prove Lemma \ref{lem:inner_integral}, it is enough to show that
    \begin{align*}
        &
        \int\limits_0^1 v_1 \rmd v_1 \ldots \int\limits_0^1 v_k \rmd v_k \int\limits_{\left[-\frac{nt}{2k}, \frac{nt}{2k}\right]^k} \rmd u_1 \dots \rmd u_k \int\limits_{-\infty}^{+\infty} \rmd y \cdot \left| \Phi^{(2k)}_\theta \left(t - \sum\limits_{i = 1}^k \frac{v_i u_i}n - \frac{y}n \right)\right| \, \left( \prod\limits_{i = 1}^k \frac{u_i^2 \, \big|\sfr_j(u_i)\big|}{n^2} \right) \cdot \sfs_j^{*(n - k)}(y)
        \\&
        \leq \frac{\sqrt{(2k - 1)!}}{\sqrt{2\pi}}
        \cdot \left( \frac{2^{3 + 2\beta_j} B_j \theta \sqrt{\pi}}{n^{\beta_j - 1} |t|^{1 + \beta_j}} + \exp\left\{- \frac{t^2}{64 \theta^2}\right\} \right) \cdot \left( \frac{M_j}{2 \theta^2 n^2} \right)^k.
    \end{align*}
    Our argument is quite technical, so we divide the proof into several parts.

    \medskip
    
    \noindent\textbf{Step 1: a bound on the $2k$-th derivative}\quad
    First, let us consider the $2k$-th derivative of $\Phi_\theta$.
    Note that
    \[
        \Phi^{(2k)}_\theta(w)
        = \frac1{\theta^{2k}} \cdot \Phi^{(2k)}_1 \left( \frac{w}\theta \right)
        = \frac1{\theta^{2k} \sqrt{2\pi}} \cdot \cH_{2k - 1}(w / \theta) \cdot \exp\left\{- \frac{w^2}{2 \theta^2} \right\},
    \]
    where $\cH_{2k - 1}$ is the $(2k - 1)$-th ``probabilist's'' Hermite polynomial. We provide a brief information about Hermite polynomials in Appendix \ref{sec:hermite}. In particular, we refer to the result of \citet{indritz1961inequality}, which implies that
    \[
        \max\limits_{w \in \R} \left| \cH_{2k - 1} \left( w / \theta \right) \cdot \exp\left\{- \frac{w^2}{4 \theta^2} \right\} \right|
        \leq \sqrt{(2k - 1)!}
        \quad \text{for all $k \in \mathbb N$.}
    \]
    Thus, it holds that
    \begin{equation}
        \label{eq:phi_derivative_bound}
        \left|\Phi^{(2k)}_\theta(w)\right| \leq \frac{\sqrt{(2k - 1)!}}{\theta^{2k} \, \sqrt{2\pi}} \; \exp\left\{- \frac{w^2}{4 \theta^2} \right\}
        \quad \text{for all $w \in \R$ and all $k$,}
    \end{equation}
    and we obtain the inequality
    \begin{align}
        \label{eq:derivative_integral_bound}
        &\notag
        \int\limits_0^1 v_1 \rmd v_1 \ldots \int\limits_0^1 v_k \rmd v_k \int\limits_{\left[-\frac{nt}{2k}, \frac{nt}{2k}\right]^k} \rmd u_1 \dots \rmd u_k \int\limits_{-\infty}^{+\infty} \rmd y \cdot \left| \Phi^{(2k)}_\theta \left(t - \sum\limits_{i = 1}^k \frac{v_i u_i}n - \frac{y}n \right)\right| \, \left( \prod\limits_{i = 1}^k \frac{u_i^2 \, \big|\sfr_j(u_i)\big|}{n^2} \right) \cdot \sfs_j^{*(n - k)}(y)
        \\&
        \leq \frac{\sqrt{(2k - 1)!}}{\theta^{2k} \, \sqrt{2\pi}}
        \int\limits_0^1 v_1 \rmd v_1 \ldots \int\limits_0^1 v_k \rmd v_k \int\limits_{\left[-\frac{nt}{2k}, \frac{nt}{2k}\right]^k} \rmd u_1 \dots \rmd u_k \int\limits_{-\infty}^{+\infty} \rmd y \cdot \exp\left\{- \frac1{4 \theta^2} \left( t - \sum\limits_{i = 1}^k \frac{v_i u_i}n - \frac{y}n \right)^2 \right\}
        \\& \notag \hspace{4in}
        \cdot \left( \prod\limits_{i = 1}^k \frac{u_i^2 \, \big|\sfr_j(u_i)\big|}{n^2} \right) \cdot \sfs_j^{*(n - k)}(y).
    \end{align}

    \medskip
    
    \noindent\textbf{Step 2: a bound on the convolution.}
    \quad
    Our next goal is to bound the convolution
    \[
        \int\limits_{-\infty}^{+\infty} \exp\left\{- \frac1{4 \theta^2} \left( t - \sum\limits_{i = 1}^k \frac{v_i u_i}n - \frac{y}n \right)^2 \right\} \sfs_j^{*(n - k)}(y) \rmd y
    \]
    using the properties of $\sfs_j^{*(n - k)}$ from Assumption \ref{as:convolution}. Let us fix an arbitrary $w \in \R$ and consider
    \[
        \int\limits_{-\infty}^{+\infty} \exp\left\{- \frac{(w - y / n)^2}{4 \theta^2} \right\} \sfs_j^{*(n - k)}(y) \rmd y.
    \]
    Since $\exp\left\{- (n w - y)^2 / (4 n^2 \theta^2) \right\} \leq 1$, it holds that
    \begin{equation}
        \label{eq:convolution_trivial_bound}
        \int\limits_{-\infty}^{+\infty} \exp\left\{- \frac{(w - y / n)^2}{4 \theta^2} \right\} \sfs_j^{*(n - k)}(y) \rmd y
        \leq 1.
    \end{equation}
    On the other hand, we have
    \begin{align*}
        \int\limits_{-\infty}^{+\infty} \exp\left\{- \frac{(w - y / n)^2}{4 \theta^2} \right\} \sfs_j^{*(n - k)}(y) \rmd y
        &
        = n \int\limits_{-\infty}^{+\infty} \exp\left\{- \frac{y^2}{4 \theta^2} \right\} \sfs_j^{*(n - k)}(ny + nw) \, \rmd y
        \\&
        \leq \int\limits_{-w / 2}^{w / 2} \exp\left\{- \frac{y^2}{4 \theta^2} \right\} \frac{B_j n (n - k)}{(n - k)^{(\beta_j + 1) / \beta_j} + n^{1 + \beta_j} |w + y|^{1 + \beta_j}} \rmd y
        \\&\quad
        + n \int\limits_{|y| > w/2} \exp\left\{- \frac{y^2}{4 \theta^2} \right\} \sfs_j^{*(n - k)}(ny + nw) \, \rmd y.
    \end{align*}
    If $|y| \leq w / 2$, then
    \[
        \frac{B_j n (n - k)}{(n - k)^{(\beta_j + 1) / \beta_j} + n^{1 + \beta_j} |w + y|^{1 + \beta_j}}
        \leq \frac{B_j n (n - k)}{(n - k)^{(\beta_j + 1) / \beta_j} + n^{1 + \beta_j} |w/2|^{1 + \beta_j}}
    \]
    and it holds that
    \begin{align}
        \label{eq:integral_segment}
        &\notag
        \int\limits_{-w / 2}^{w / 2} \exp\left\{- \frac{y^2}{4 \theta^2} \right\} \frac{B_j n (n - k)}{(n - k)^{(\beta_j + 1) / \beta_j} + n^{1 + \beta_j} |w + y|^{1 + \beta_j}} \rmd y
        \\&
        \leq \frac{B_j n (n - k)}{(n - k)^{(\beta_j + 1) / \beta_j} + |n w / 2|^{1 + \beta_j}} \int\limits_{-w / 2}^{w / 2} \exp\left\{- \frac{y^2}{4 \theta^2} \right\} \rmd y
        \\&\notag
        \leq \frac{B_j n (n - k)}{(n - k)^{(\beta_j + 1) / \beta_j} + |n w / 2|^{1 + \beta_j}} \cdot \sqrt{4 \pi} \theta
        \leq \frac{2^{2 + \beta_j} B_j \theta \sqrt{\pi}}{n^{\beta_j - 1} |w|^{1 + \beta_j}}.
    \end{align}
    Otherwise,
    \begin{align}
        \label{eq:integral_complement}
        n \int\limits_{|y| > w/2} \exp\left\{- \frac{y^2}{4 \theta^2} \right\} \sfs_j^{*(n - k)}(ny + nw) \, \rmd y
        &\notag
        \leq \exp\left\{- \frac{w^2}{16 \theta^2} \right\} \int\limits_{-\infty}^{+\infty} \sfs_j^{*(n - k)}(ny + nw) \, n \rmd y
        \\&
        = \exp\left\{- \frac{w^2}{16 \theta^2} \right\}.
    \end{align}
    Taking into account \eqref{eq:convolution_trivial_bound}, \eqref{eq:integral_segment}, and \eqref{eq:integral_complement}, we obtain that
    \[
        \int\limits_{-\infty}^{+\infty} \exp\left\{- \frac{(w - y / n)^2}{4 \theta^2} \right\} \sfs_j^{*(n - k)}(y) \rmd y
        \leq 1 \land \left( \frac{2^{2 + \beta_j} B_j \theta \sqrt{\pi}}{n^{\beta_j - 1} |w|^{1 + \beta_j}} + \exp\left\{- \frac{w^2}{16 \theta^2} \right\} \right)
    \]
    and, hence,
    \begin{align}
        \label{eq:convolution_bound}
        &\notag
        \int\limits_{-\infty}^{+\infty} \exp\left\{- \frac1{4 \theta^2} \left( t - \sum\limits_{i = 1}^k \frac{v_i u_i}n - \frac{y}n \right)^2 \right\} \sfs_j^{*(n - k)}(y) \rmd y
        \\&
        \leq \min\left\{1, \left( \frac{2^{2 + \beta_j} B_j \theta \sqrt{\pi}}{n^{\beta_j - 1}} \left| t - \sum\limits_{i = 1}^k \frac{v_i u_i}n \right|^{-1 - \beta_j} + \exp\left\{- \frac{1}{16 \theta^2} \left( t - \sum\limits_{i = 1}^k \frac{v_i u_i}n \right)^2 \right\} \right) \right\}.
    \end{align}

    \medskip
    
    \noindent\textbf{Step 3: final bound.}
    \quad The inequalities \eqref{eq:derivative_integral_bound} and \eqref{eq:convolution_bound} yield that
    \begin{align*}
        &
        \int\limits_0^1 v_1 \rmd v_1 \ldots \int\limits_0^1 v_k \rmd v_k \int\limits_{\left[-\frac{nt}{2k}, \frac{nt}{2k}\right]^k} \rmd u_1 \dots \rmd u_k \int\limits_{-\infty}^{+\infty} \rmd y \cdot \left| \Phi^{(2k)}_\theta \left(t - \sum\limits_{i = 1}^k \frac{v_i u_i}n - \frac{y}n \right)\right| \, \left( \prod\limits_{i = 1}^k \frac{u_i^2 \, \big|\sfr_j(u_i)\big|}{n^2} \right) \cdot \sfs_j^{*(n - k)}(y)
        \\&
        \leq \frac{\sqrt{(2k - 1)!}}{\theta^{2k} \, \sqrt{2\pi}}
        \int\limits_0^1 v_1 \rmd v_1 \ldots \int\limits_0^1 v_k \rmd v_k \int\limits_{\left[-\frac{nt}{2k}, \frac{nt}{2k}\right]^k} \rmd u_1 \dots \rmd u_k \cdot \left( \prod\limits_{i = 1}^k \frac{u_i^2 \, \big|\sfr_j(u_i)\big|}{n^2} \right)
        \\&\quad
        \cdot \min\left\{1, \left( \frac{2^{2 + \beta_j} B_j \theta \sqrt{\pi}}{n^{\beta_j - 1}} \left| t - \sum\limits_{i = 1}^k \frac{v_i u_i}n \right|^{-1 - \beta_j} + \exp\left\{- \frac{1}{16 \theta^2} \left( t - \sum\limits_{i = 1}^k \frac{v_i u_i}n \right)^2 \right\} \right) \right\}.
    \end{align*}
    On the set $[-nt / (2k), nt / (2k)]^k$, we have
    \[
        \left| t - \sum\limits_{i = 1}^k \frac{v_i u_i}n \right| \geq \frac{t}2,
    \]
    and, hence,
    \begin{align*}
        &
        \int\limits_0^1 v_1 \rmd v_1 \ldots \int\limits_0^1 v_k \rmd v_k \int\limits_{\left[-\frac{nt}{2k}, \frac{nt}{2k}\right]^k} \rmd u_1 \dots \rmd u_k \int\limits_{-\infty}^{+\infty} \rmd y \cdot \left| \Phi^{(2k)}_\theta \left(t - \sum\limits_{i = 1}^k \frac{v_i u_i}n - \frac{y}n \right)\right| \, \left( \prod\limits_{i = 1}^k \frac{u_i^2 \, \big|\sfr_j(u_i)\big|}{n^2} \right) \cdot \sfs_j^{*(n - k)}(y)
        \\&
        \leq \frac{\sqrt{(2k - 1)!}}{\theta^{2k} \, \sqrt{2\pi}}
        \cdot \left( \frac{2^{3 + 2\beta_j} B_j \theta \sqrt{\pi}}{n^{\beta_j - 1} |t|^{1 + \beta_j}} + \exp\left\{- \frac{t^2}{64 \theta^2}\right\} \right) \int\limits_0^1 v_1 \rmd v_1 \ldots \int\limits_0^1 v_k \rmd v_k \int\limits_{\left[-\frac{nt}{2k}, \frac{nt}{2k}\right]^k} \rmd u_1 \dots \rmd u_k \cdot \left( \prod\limits_{i = 1}^k \frac{u_i^2 \, \big|\sfr_j(u_i)\big|}{n^2} \right).
    \end{align*}
    Due to Assumption \ref{as:convolution}, the right-hand side does not exceed
    \begin{align*}
        &
        \frac{\sqrt{(2k - 1)!}}{\theta^{2k} \, \sqrt{2\pi}}
        \cdot \left( \frac{2^{3 + 2\beta_j} B_j \theta \sqrt{\pi}}{n^{\beta_j - 1} |t|^{1 + \beta_j}} + \exp\left\{- \frac{t^2}{64 \theta^2}\right\} \right) \cdot \left( \frac{M_j}{n^2} \right)^k \int\limits_0^1 v_1 \rmd v_1 \ldots \int\limits_0^1 v_k \rmd v_k
        \\&
        = \frac{\sqrt{(2k - 1)!}}{\sqrt{2\pi}}
        \cdot \left( \frac{2^{3 + 2\beta_j} B_j \theta \sqrt{\pi}}{n^{\beta_j - 1} |t|^{1 + \beta_j}} + \exp\left\{- \frac{t^2}{64 \theta^2}\right\} \right) \cdot \left( \frac{M_j}{2 \theta^2 n^2} \right)^k.
    \end{align*}
    The proof is finished.
    
\end{proof}

\begin{lemma}
    \label{lem:outer_integral}
    Let Assumption \ref{as:convolution} be fulfilled. Then, for any $k \geq 2$, it holds that
    \begin{align*}
        &
        \left| \int\limits_0^1 v_1 \rmd v_1 \ldots \int\limits_0^1 v_k \rmd v_k \int\limits_{\R^k \left\backslash \left[-\frac{nt}{2k}, \frac{nt}{2k}\right]^k \right.} \rmd u \int\limits_{-\infty}^{+\infty} \rmd y \cdot \Phi^{(2k)}_\theta \left(t - \sum\limits_{i = 1}^k \frac{v_i u_i}n - \frac{y}n \right) \left( \prod\limits_{i = 1}^k \frac{u_i^2 \, \sfr_j(u_i)}{n^2} \right) \cdot \sfs_j^{*(n - k)}(y) \right|
        \\&
        \leq \sqrt{(2k - 2)!} \left( \frac{M_j}{2 \theta^2 n^2} \right)^{k} \left( \frac{4 k^2 \theta^2}{t^2} + \frac{1}{\sqrt{2\pi}} \cdot \frac{2k \theta}{t} \right).
    \end{align*}
\end{lemma}

\begin{proof}
    Let $\{A_i : 1 \leq i \leq k\}$ be a collection of sets in $\R^k$, such that
    \begin{itemize}
        \item $A_1, \dots, A_k$ form a partition of $\R^k \left\backslash \left[-\frac{nt}{2k}, \frac{nt}{2k}\right]^k \right.$, that is,
        \[
            \R^k \bigg \backslash \left[-\frac{nt}{2k}, \frac{nt}{2k}\right]^k = \bigcup\limits_{i = 1}^k A_i
            \quad \text{and} \quad
            A_i \cap A_\ell = \emptyset \quad \text{for all $i \neq \ell$;}
        \]
        \item for all $i \in \{1, \dots, k\}$, it holds that
        \[
            A_i \subseteq \left\{ u : |u_i| \geq \frac{nt}{2k} \right\}.
        \]
    \end{itemize}
    Let us fix any $\ell \in \{1, \dots, k\}$ and consider the integral
    \[
        \int\limits_0^1 v_1 \rmd v_1 \ldots \int\limits_0^1 v_k \rmd v_k \int\limits_{A_\ell} \rmd u_1 \dots \rmd u_k \int\limits_{-\infty}^{+\infty} \rmd y \cdot \Phi^{(2k)}_\theta \left(t - \sum\limits_{i = 1}^k \frac{v_i u_i}n - \frac{y}n \right) \, \left( \prod\limits_{i = 1}^k \frac{u_i^2 \, \sfr_j(u_i)}{n^2} \right) \cdot \sfs_j^{*(n - k)}(y).
    \]
    Applying the Newton-Leibnitz formula, we obtain that
    \begin{align*}
        \frac{u_\ell^2}{n^2} \int\limits_0^1 \Phi^{(2k)}_\theta \left(t - \sum\limits_{i = 1}^k \frac{v_i u_i}n - \frac{y}n \right) v_\ell \, \rmd v_\ell 
        &
        = \Phi^{(2k - 2)}_\theta \left(t - \sum\limits_{\substack{1 \leq i \leq k,\\i \neq \ell}} \frac{v_i u_i}n - \frac{y}n \right) 
        \\&\quad
        + \Phi^{(2k - 1)}_\theta \left(t - \sum\limits_{\substack{1 \leq i \leq k,\\i \neq \ell}} \frac{v_i u_i}n - \frac{y}n \right) \cdot \frac{u_\ell}{n}.
    \end{align*}
    This implies that
    \begin{align*}
        &
        \left| \int\limits_0^1 v_1 \rmd v_1 \ldots \int\limits_0^1 v_k \rmd v_k \int\limits_{A_\ell} \rmd u_1 \dots \rmd u_k \int\limits_{-\infty}^{+\infty} \rmd y \cdot \Phi^{(2k)}_\theta \left(t - \sum\limits_{i = 1}^k \frac{v_i u_i}n - \frac{y}n \right) \, \left( \prod\limits_{i = 1}^k \frac{u_i^2 \, \sfr_j(u_i)}{n^2} \right) \cdot \sfs_j^{*(n - k)}(y) \right|
        \\&
        \leq \int\limits_0^1 v_1 \rmd v_1 \ldots \int\limits_0^1 v_k \rmd v_k \int\limits_{A_\ell} \rmd u_1 \dots \rmd u_k \int\limits_{-\infty}^{+\infty} \rmd y \cdot \left| \Phi^{(2k - 2)}_\theta \left(t - \sum\limits_{\substack{1 \leq i \leq k,\\i \neq \ell}} \frac{v_i u_i}n - \frac{y}n \right) \right|
        \\&\quad
        \cdot \left( \prod\limits_{\substack{1 \leq i \leq k\\i \neq \ell}} \frac{u_i^2 \, \big| \sfr_j(u_i) \big|}{n^2} \right) \cdot \big|\sfr_j(u_\ell) \big| \cdot \sfs_j^{*(n - k)}(y)
        \\& \quad
        + \int\limits_0^1 v_1 \rmd v_1 \ldots \int\limits_0^1 v_k \rmd v_k \int\limits_{A_\ell} \rmd u_1 \dots \rmd u_k \int\limits_{-\infty}^{+\infty} \rmd y \cdot \left| \Phi^{(2k - 1)}_\theta \left(t - \sum\limits_{\substack{1 \leq i \leq k,\\i \neq \ell}} \frac{v_i u_i}n - \frac{y}n \right) \right|
        \\&\quad
        \cdot \left( \prod\limits_{\substack{1 \leq i \leq k\\i \neq \ell}} \frac{u_i^2 \, \big| \sfr_j(u_i) \big|}{n^2} \right) \cdot \frac{u_\ell \big|\sfr_j(u_\ell) \big|}{n} \cdot \sfs_j^{*(n - k)}(y)
    \end{align*}
    Let us apply the inequality \eqref{eq:phi_derivative_bound} we derived in the proof of Lemma \ref{lem:inner_integral} to $\Phi^{(2k - 2)}_\theta$ and $\Phi^{(2k - 1)}_\theta$: for all $w \in \R$, it holds that
    \[
        0 \leq \Phi_\theta(w) \leq 1,
        \quad
        \left|\Phi^{(2k - 2)}_\theta(w)\right|
        \leq \frac{\sqrt{(2k - 3)!}}{\theta^{2k - 2} \, \sqrt{2\pi}}
        \; \exp\left\{- \frac{w^2}{4 \theta^2} \right\}
        \leq \frac{\sqrt{(2k - 2)!}}{\theta^{2k - 2}},
        \quad \text{for all $k \geq 2$.}
    \]
    and
    \[
        \left|\Phi^{(2k - 1)}_\theta(w)\right|
        \leq \frac{\sqrt{(2k - 2)!}}{\theta^{2k - 1} \, \sqrt{2\pi}} \; \exp\left\{- \frac{w^2}{4 \theta^2} \right\}
        \leq \frac{\sqrt{(2k - 2)!}}{\theta^{2k - 1} \, \sqrt{2\pi}}.
    \]
    Then, for any $k \in \mathbb N$,
    \begin{align*}
        &
        \int\limits_0^1 v_1 \rmd v_1 \ldots \int\limits_0^1 v_k \rmd v_k \int\limits_{A_\ell} \rmd u_1 \dots \rmd u_k \int\limits_{-\infty}^{+\infty} \rmd y \cdot \left| \Phi^{(2k - 2)}_\theta \left(t - \sum\limits_{\substack{1 \leq i \leq k,\\i \neq \ell}} \frac{v_i u_i}n - \frac{y}n \right) \right|
        \\&\quad
        \cdot \left( \prod\limits_{\substack{1 \leq i \leq k\\i \neq \ell}} \frac{u_i^2 \, \big| \sfr_j(u_i) \big|}{n^2} \right) \cdot \big|\sfr_j(u_\ell) \big| \cdot \sfs_j^{*(n - k)}(y)
        \\&
        \leq \frac{\sqrt{(2k - 2)!}}{\theta^{2k - 2}} \int\limits_0^1 v_1 \rmd v_1 \ldots \int\limits_0^1 v_k \rmd v_k \int\limits_{A_\ell} \left( \prod\limits_{\substack{1 \leq i \leq k\\i \neq \ell}} \frac{u_i^2 \, \big| \sfr_j(u_i) \big|}{n^2} \right) \cdot \big|\sfr_j(u_\ell) \big| \, \rmd u_1 \dots \rmd u_k.
    \end{align*}
    Due to Assumption \ref{as:convolution},
    \begin{align*}
        &
        \frac{\sqrt{(2k - 2)!}}{\theta^{2k - 2}} \int\limits_0^1 v_1 \rmd v_1 \ldots \int\limits_0^1 v_k \rmd v_k \int\limits_{A_\ell} \left( \prod\limits_{\substack{1 \leq i \leq k\\i \neq \ell}} \frac{u_i^2 \, \big| \sfr_j(u_i) \big|}{n^2} \right) \cdot \big|\sfr_j(u_\ell) \big| \, \rmd u_1 \dots \rmd u_k
        \\&
        \leq \sqrt{(2k - 2)!} \left( \frac{M_j}{2 \theta^2 n^2} \right)^{k - 1} \int\limits_{|u_\ell| > nt / (2k)} \big|\sfr_j(u_\ell) \big| \, \rmd u_\ell
        \\&
        \leq \sqrt{(2k - 2)!} \left( \frac{M_j}{2 \theta^2 n^2} \right)^{k - 1} \cdot \frac{2 k^2}{n^2 t^2} \int\limits_{|u_\ell| > nt / (2k)} u_\ell^2 \big|\sfr_j(u_\ell) \big| \, \rmd u_\ell
        \\&
        \leq \sqrt{(2k - 2)!} \left( \frac{M_j}{2 \theta^2 n^2} \right)^{k - 1} \cdot \frac{2 k^2 M_j}{n^2 t^2}.
    \end{align*}
    Similarly, it holds that
    \begin{align*}
        &
        \int\limits_0^1 v_1 \rmd v_1 \ldots \int\limits_0^1 v_k \rmd v_k \int\limits_{A_\ell} \rmd u_1 \dots \rmd u_k \int\limits_{-\infty}^{+\infty} \rmd y \cdot \left| \Phi^{(2k - 1)}_\theta \left(t - \sum\limits_{\substack{1 \leq i \leq k,\\i \neq \ell}} \frac{v_i u_i}n - \frac{y}n \right) \right|
        \\&\quad
        \cdot \left( \prod\limits_{\substack{1 \leq i \leq k\\i \neq \ell}} \frac{u_i^2 \, \big| \sfr_j(u_i) \big|}{n^2} \right) \cdot \frac{u_\ell \big|\sfr_j(u_\ell) \big|}{n} \cdot \sfs_j^{*(n - k)}(y)
        \\&
        \leq \frac{\sqrt{(2k - 2)!}}{\theta^{2k - 1} \, \sqrt{2\pi}} \int\limits_0^1 v_1 \rmd v_1 \ldots \int\limits_0^1 v_k \rmd v_k \int\limits_{A_\ell} \left( \prod\limits_{\substack{1 \leq i \leq k\\i \neq \ell}} \frac{u_i^2 \, \big| \sfr_j(u_i) \big|}{n^2} \right) \cdot \frac{u_\ell \big|\sfr_j(u_\ell) \big|}{n} \, \rmd u_1 \dots \rmd u_k
        \\&
        \leq \frac{\sqrt{(2k - 2)!}}{\theta \, \sqrt{2\pi}} \cdot \left( \frac{M_j}{2 \theta^2 n^2} \right)^{k - 1} \int\limits_{|u_\ell| > nt / (2k)}\frac{u_\ell \big|\sfr_j(u_\ell) \big|}{n} \, \rmd u_\ell
        \\&
        \leq \frac{\sqrt{(2k - 2)!}}{\theta \, \sqrt{2\pi}} \cdot \left( \frac{M_j}{2 \theta^2 n^2} \right)^{k - 1} \cdot \left( \frac{k}{n^2 t} \right) \int\limits_{|u_\ell| > nt / (2k)} u_\ell^2 \big|\sfr_j(u_\ell) \big| \, \rmd u_\ell
        \\&
        \leq \frac{\sqrt{(2k - 2)!}}{\sqrt{2\pi}} \cdot \left( \frac{M_j}{2 \theta^2 n^2} \right)^{k - 1} \cdot \frac{k M_j}{\theta n^2 t}.
    \end{align*}
    Thus, we obtain that
    \begin{align*}
        &
        \left| \int\limits_0^1 v_1 \rmd v_1 \ldots \int\limits_0^1 v_k \rmd v_k \int\limits_{A_\ell} \rmd u_1 \dots \rmd u_k \int\limits_{-\infty}^{+\infty} \rmd y \cdot \Phi^{(2k)}_\theta \left(t - \sum\limits_{i = 1}^k \frac{v_i u_i}n - \frac{y}n \right) \, \left( \prod\limits_{i = 1}^k \frac{u_i^2 \, \sfr_j(u_i)}{n^2} \right) \cdot \sfs_j^{*(n - k)}(y) \right|
        \\&
        \leq \sqrt{(2k - 2)!} \left( \frac{M_j}{2 \theta^2 n^2} \right)^{k} \cdot \frac{4 k^2 \theta^2}{t^2} + \frac{\sqrt{(2k - 2)!}}{\sqrt{2\pi}} \cdot \left( \frac{M_j}{2 \theta^2 n^2} \right)^{k} \cdot \frac{2k \theta}{t}.
    \end{align*}
    Hence, due to the triangle inequality,
    \begin{align*}
        &
        \left| \int\limits_0^1 v_1 \rmd v_1 \ldots \int\limits_0^1 v_k \rmd v_k \int\limits_{\R^k \left\backslash \left[-\frac{nt}{2k}, \frac{nt}{2k}\right]^k \right.} \rmd u_1 \dots \rmd u_k \int\limits_{-\infty}^{+\infty} \rmd y \cdot \Phi^{(2k)}_\theta \left(t - \sum\limits_{i = 1}^k \frac{v_i u_i}n - \frac{y}n \right) \left( \prod\limits_{i = 1}^k \frac{u_i^2 \, \sfr_j(u_i)}{n^2} \right) \cdot \sfs_j^{*(n - k)}(y) \right|
        \\&
        \leq \sum\limits_{\ell = 1}^k \left| \int\limits_0^1 v_1 \rmd v_1 \ldots \int\limits_0^1 v_k \rmd v_k \int\limits_{A_\ell} \rmd u_1 \dots \rmd u_k \int\limits_{-\infty}^{+\infty} \rmd y \cdot \Phi^{(2k)}_\theta \left(t - \sum\limits_{i = 1}^k \frac{v_i u_i}n - \frac{y}n \right) \left( \prod\limits_{i = 1}^k \frac{u_i^2 \, \sfr_j(u_i)}{n^2} \right) \cdot \sfs_j^{*(n - k)}(y) \right|
        \\&
        \leq \sqrt{(2k - 2)!} \left( \frac{M_j}{2 \theta^2 n^2} \right)^{k} \left( \frac{4 k^2 \theta^2}{t^2} + \frac{1}{\sqrt{2\pi}} \cdot \frac{2k \theta}{t} \right).
    \end{align*}
\end{proof}

\begin{lemma}
    \label{lem:sums}
    Let $n \in \mathbb N$ and $a \geq 0$ be such that $na \leq 1/2$. Then it holds that 
    \[
        \sum\limits_{k = 1}^n \binom{n}{k} \sqrt{(2k)!} \left(\frac{a}2 \right)^k
        \leq 2na \quad \text{and} \quad
        \sum\limits_{k = 1}^n \binom{n}{k} \sqrt{(2k)!} \cdot k \left(\frac{a}2 \right)^k
        \leq 4na.
    \]
\end{lemma}

\begin{proof}
    First, note that, for any positive integer $k$, we have
    \[
        (2k)! = \prod\limits_{j = 1}^k (2j) \cdot \prod\limits_{j = 1}^k (2j - 1)
        \leq \prod\limits_{j = 1}^k (2j) \cdot \prod\limits_{j = 1}^k (2j)
        \leq 4^k \cdot k! \cdot k!. 
    \]
    This implies that 
    \[
        \sum\limits_{k = 1}^n \binom{n}{k} \sqrt{(2k)!} \left(\frac{a}2 \right)^k
        \leq \sum\limits_{k = 1}^n \frac{n!}{(n - k)!} \cdot a^k
        \leq \sum\limits_{k = 1}^n (na)^k
        \leq \sum\limits_{k = 1}^\infty (na)^k
        = \frac{na}{1 - na}
        \leq 2 na.
    \]
    Similarly, it holds that
    \[
        \sum\limits_{k = 1}^n \binom{n}{k} \sqrt{(2k)!} \cdot k \left(\frac{a}2 \right)^k
        \leq \sum\limits_{k = 1}^n \frac{n!}{(n - k)!} \cdot k a^k
        \leq \sum\limits_{k = 1}^n k (na)^k 
        \leq na \sum\limits_{k = 0}^\infty k (na)^{k - 1} 
        = \frac{na}{(1 - na)^2}
        \leq 4na.
    \]
\end{proof}
\newpage

\section{Proofs for \algname{clipped-SGD}}

\subsection{Auxiliary Results}

\paragraph{Bernstein inequality.} The following lemma (known as {\it Bernstein inequality for martingale differences} \citep{bennett1962probability,dzhaparidze2001bernstein,freedman1975tail}) is essential for deriving high-probability upper bounds in our analysis.
\begin{lemma}\label{lem:Bernstein_ineq}
    Let the sequence of random variables $\{X_i\}_{i\ge 1}$ form a martingale difference sequence, i.e.\ $\EE\left[X_i\mid X_{i-1},\ldots, X_1\right] = 0$ for all $i \ge 1$. Assume that conditional variances $\sigma_i^2\eqdef\EE\left[X_i^2\mid X_{i-1},\ldots, X_1\right]$ exist and are bounded and assume also that there exists deterministic constant $c>0$ such that $|X_i| \le c$ almost surely for all $i\ge 1$. Then for all $b > 0$, $G > 0$ and $n\ge 1$
    \begin{equation}
        \Prob\left\{\Big|\sum\limits_{i=1}^nX_i\Big| > b \text{ and } \sum\limits_{i=1}^n\sigma_i^2 \le G\right\} \le 2\exp\left(-\frac{b^2}{2G + \nicefrac{2cb}{3}}\right).
    \end{equation}
\end{lemma}

\paragraph{Bias and variance of the clipped stochastic vector.} We also rely on the following result from \citep{gorbunov2020stochastic}.

\begin{lemma}[Simplified version of Lemma F.5 from \citep{gorbunov2020stochastic}]\label{lem:bias_and_variance_clip}
    Let $X$ be a random vector in $\R^d$ and $\tX = \clip(X,\lambda)$. Then,
    \begin{equation}
        \left\|\tX - \EE[\tX]\right\| \leq 2\lambda. \label{eq:bound_X}
    \end{equation} 
    Moreover, if for some $\sigma \geq 0$
    \begin{equation}
        \EE[X] = x\in\R^d,\quad \EE[\|X - x\|^2] \leq \sigma^2 \label{eq:UBV_X}
    \end{equation}
    and $x \leq \nicefrac{\lambda}{2}$, then
    \begin{eqnarray}
        \left\|\EE[\tX] - x\right\| &\leq& \frac{4\sigma^2}{\lambda}, \label{eq:bias_X}\\
        \EE\left[\left\|\tX - \EE[\tX]\right\|^2\right] &\leq& 18\sigma^2. \label{eq:variance_X}
    \end{eqnarray}
\end{lemma}

\subsection{Quasi-Convex Case}

The analysis of \algname{clipped-SGD} in the quasi-convex case relies on the following lemma from \citep{sadiev2023high}.

\begin{lemma}\label{lem:main_opt_lemma_clipped_SGD_convex}
    Let Assumptions~\ref{as:L_smoothness} and \ref{as:QSC}\footnote{Although \citet{sadiev2023high} claim that they use Assumption~\ref{as:str_cvx} with $\mu = 0$, their proof relies on Assumption~\ref{as:QSC} with $\mu = 0$ instead, which is strictly weaker.} with $\mu = 0$ hold on $Q = B_{2R}(x^*)$, where $R \geq \|x^0 - x^*\|$, and let stepsize  $\gamma_k \equiv \gamma$ satisfy $\gamma \leq \frac{1}{L}$. If $x^{k} \in Q$ for all $k = 0,1,\ldots,K+1$, $K \ge 0$, then after $K$ iterations of \algname{clipped-SGD} we have
    \begin{eqnarray}
       \gamma\left(f(\overline{x}^K) -f(x^*)\right) &\leq& \frac{\|x^0 - x^*\|^2 - \|x^{K+1} - x^*\|^2}{K+1}\notag\\
       &&\quad - \frac{2\gamma}{K+1}\sum\limits_{k=0}^K \langle x^k - x^*  - \gamma \nabla f(x^k), \theta_k \rangle  + \frac{\gamma^2}{K+1}\sum\limits_{k=0}^K \|\theta_k\|^2 ,\label{eq:main_opt_lemma_clipped_SGD_star_convex}\\
       \overline{x}^K &=& \frac{1}{K+1}\sum\limits_{k=0}^K x^k, \label{eq:x_avg_clipped_SGD}\\
       \theta_{k} &\eqdef& \clip(\nabla f_{\Xi^k}(x^k), \lambda_k) - \nabla f(x^{k}).\label{eq:theta_k_def_clipped_SGD_non_convex}
    \end{eqnarray}
\end{lemma}

\begin{theorem}\label{thm:clipped_SGD_cvx_appendix}
    Let Assumptions~\ref{as:L_smoothness} and \ref{as:QSC} with $\mu = 0$ hold on $Q = B_{2R}(x^*)$, where $R \geq \|x^0 - x^*\|$. Assume that $\nabla f_{\Xi^k}(x^k)$ satisfies Assumption~\ref{as:bounded_bias_and_variance} with parameters $b_k, \sigma_k$ for $k = 0,1,\ldots,K$, $K > 0$ and
    \begin{gather}
        \gamma_k \equiv \gamma \leq \min\left\{\frac{1}{160L \ln \frac{4(K+1)}{\delta}}, \frac{R}{208 \sigma_k  \sqrt{K\ln \frac{4(K+1)}{\delta}}}, \frac{R}{160b_k \ln \frac{4(K+1)}{\delta}}, \frac{R}{1600b_k(K+1)}\right\},\label{eq:clipped_SGD_step_size_cvx}\\
        \lambda_{k}  = \frac{R}{40 \gamma\ln\frac{4(K+1)}{\delta}}, \label{eq:clipped_SGD_clipping_level_cvx}
    \end{gather}
    for some $\delta \in (0,1]$. Then, after $K$ iterations of \algname{clipped-SGD} the iterates with probability at least $1 - \delta$ satisfy
    \begin{equation}
        f(\overline{x}^K) - f(x^*) \leq \frac{2R^2}{\gamma(K+1)} \quad \text{and} \quad  \{x^k\}_{k=0}^{K} \subseteq B_{\sqrt{2}R}(x^*).\label{eq:clipped_SGD_convex_case_appendix}
    \end{equation}
    In particular, when 
    \begin{gather}
        \gamma = \min\left\{\frac{1}{160L \ln \frac{4(K+1)}{\delta}}, \frac{R}{208 \sigma  \sqrt{K\ln \frac{4(K+1)}{\delta}}}, \frac{R}{160b\ln \frac{4(K+1)}{\delta}}, \frac{R}{1600b(K+1)}\right\}, \label{eq:clipped_SGD_step_size_cvx_exact}\\
        \text{where}\quad \sigma = \min\limits_{k=0,1,\ldots,K}\sigma_k,\quad b = \min\limits_{k=0,1,\ldots,K}b_k,
    \end{gather}
    then the iterates produced by \algname{clipped-SGD} after $K$ iterations with probability at least $1-\delta$ satisfy
    \begin{equation}
        f(\overline{x}^K) - f(x^*) = \cO\left(\max\left\{\frac{LR^2 \ln \frac{K}{\delta}}{K}, \frac{\sigma R\sqrt{\ln\frac{K}{\delta}}}{\sqrt{K}}, \frac{bR\ln \frac{K}{\delta}}{K}, bR\right\}\right). \label{eq:clipped_SGD_convex_case_2_appendix}
    \end{equation}
\end{theorem}
\begin{proof}
    Our proof follows similar steps to the one given by \citet{sadiev2023high}. The main difference comes due to the presence of the bias in $\nabla f_{\Xi^k}(x^k)$. Therefore, for completeness, we provide the full proof here.

    Let $R_k = \|x^k - x^*\|$ for all $k\geq 0$. Our next objective is to establish, by induction, that $R_{l} \leq 2R$ with a high probability. This will enable us to apply the result from Lemma~\ref{lem:main_opt_lemma_clipped_SGD_convex} and subsequently utilize Bernstein's inequality to estimate the stochastic component of the upper bound. To be more precise, for each $k = 0,\ldots, K+1$, we consider the probability event $E_k$, defined as follows: inequalities
    \begin{eqnarray}
        - 2\gamma\sum\limits_{l=0}^{t-1} \langle x^l - x^* - \gamma \nabla f(x^l), \theta_l \rangle  + \gamma^2\sum\limits_{l=0}^{t-1} \|\theta_l\|^2 &\leq&  R^2, \label{eq:clipped_SGD_convex_induction_inequality_1}\\
        R_t &\leq& \sqrt{2}R \label{eq:clipped_SGD_convex_induction_inequality_2}
    \end{eqnarray}
    hold for all $t = 0,1,\ldots, k$ simultaneously. We aim to demonstrate through induction that $\Prob\{E_k\} \geq 1 - \nicefrac{k\delta}{(K+1)}$ for all $k = 0,1,\ldots, K+1$. The base case, $k = 0$, is trivial. Assuming that the statement holds for some $k = T - 1 \leq K$, specifically, $\Prob\{E_{T-1}\} \geq 1 - \nicefrac{(T-1)\delta}{(K+1)}$, we need to establish that $\Prob\{E_{T}\} \geq 1 - \nicefrac{T\delta}{(K+1)}$.

    To begin, we observe that the probability event $E_{T-1}$ implies that $x_t \in B_{\sqrt{2}R}(x^*)$ for all $t = 0,1,\ldots, T-1$. Furthermore, $E_{T-1}$ implies that
    \begin{eqnarray*}
        \|x^T - x^*\| = \|x^{T-1} - x^* - \gamma \tnabla f_{\Xi^{T-1}}(x^{T-1})\| \leq \|x^{T-1} - x^*\| + \gamma\|\tnabla f_{\Xi^{T-1}}(x^{T-1})\| \leq \sqrt{2}R + \gamma \lambda \overset{\eqref{eq:clipped_SGD_clipping_level_cvx}}{\leq} 2R,
    \end{eqnarray*}
    i.e., $x^0, x^1, \ldots, x^T \in B_{2R}(x^*)$. Hence, with $E_{T-1}$ implying $\{x^k\}_{k=0}^{T} \subseteq Q$, we confirm that the conditions of Lemma~\ref{lem:main_opt_lemma_clipped_SGD_convex} are met, resulting in
    \begin{eqnarray}
        \gamma \left(f(\overline{x}^{t-1}) -f(x^*)\right) &\leq& \frac{\|x^0 - x^*\|^2 - \|x^{t} - x^*\|^2}{t}\notag\\
       &&\quad - \frac{2\gamma}{t}\sum\limits_{l=0}^{t-1} \langle x^l - x^*  - \gamma \nabla f(x^l), \theta_l \rangle  + \frac{\gamma^2}{t}\sum\limits_{l=0}^{t-1} \|\theta_l\|^2 \label{eq:clipped_SGD_convex_technical_1}
    \end{eqnarray}
    for all $t=1,\ldots, T$ simultaneously and for all $t = 1, \ldots, T-1$ this probability event also implies that
    \begin{eqnarray}
        f(\overline{x}^{t-1}) -f(x^*) &\leq& \frac{1}{\gamma t}\left(R^2 - 2\gamma\sum\limits_{l=0}^{t-1} \langle x^l - x^* - \gamma \nabla f(x^l), \theta_l \rangle  + \gamma^2\sum\limits_{l=0}^{t-1} \|\theta_l\|^2\right)  \overset{\eqref{eq:clipped_SGD_convex_induction_inequality_1}}{\leq} \frac{2R^2}{\gamma t}. \label{eq:clipped_SGD_convex_technical_1_1}
    \end{eqnarray}
    Considering that $f(\overline{x}^{T-1}) -f(x^*) \geq 0$, we can further deduce from \eqref{eq:clipped_SGD_convex_technical_1} that when $E_{T-1}$ holds, the following holds as well:
    \begin{eqnarray}
        R_T^2 \leq R^2 - 2\gamma\sum\limits_{l=0}^{t-1} \langle x^l - x^* - \gamma \nabla f(x^l), \theta_l \rangle  + \gamma^2\sum\limits_{l=0}^{t-1} \|\theta_l\|^2. \label{eq:clipped_SGD_convex_technical_2}
    \end{eqnarray}
    Next, we define random vectors
    \begin{equation}
        \eta_t = \begin{cases} x^t - x^* - \gamma \nabla f(x^t),& \text{if } \|x^t - x^* - \gamma \nabla f(x^t)\| \leq 2R,\\ 0,&\text{otherwise}, \end{cases} \notag
    \end{equation}
    for all $t = 0,1,\ldots, T-1$. As per their definition, these random vectors are bounded with probability 1
    \begin{equation}
        \|\eta_t\| \leq 2R. \label{eq:clipped_SGD_convex_technical_6}
    \end{equation}
    Moreover, for $t = 0,\ldots, T-1$ event $E_{T-1}$ implies
    \begin{eqnarray}
        \|\nabla f(x^{t})\| &\overset{\eqref{eq:L_smoothness}}{\leq}& L\|x^t - x^*\|  \overset{\eqref{eq:clipped_SGD_convex_induction_inequality_2}}{\leq}  \sqrt{2}LR \overset{\eqref{eq:clipped_SGD_step_size_cvx},\eqref{eq:clipped_SGD_clipping_level_cvx}}{\leq} \frac{\lambda}{2},\label{eq:clipped_SGD_convex_technical_4}\\
        \left\|\EE_{\Xi^t}[\nabla f(x^{t})]\right\| &\leq& \left\|\EE_{\Xi^t}[\nabla f(x^{t})] - \nabla f(x^{t})\right\| + \|\nabla f(x^{t})\| \overset{\eqref{eq:bias_bound}, \eqref{eq:clipped_SGD_convex_technical_4}}{\leq} b_t + \sqrt{2}LR \overset{\eqref{eq:clipped_SGD_step_size_cvx},\eqref{eq:clipped_SGD_clipping_level_cvx}}{\leq} \frac{\lambda}{2}, \label{eq:clipped_SGD_convex_technical_4_bias}\\
        \|x^t - x^* - \gamma \nabla f(x^t)\| &\leq& \|x^t - x^*\| + \gamma \|\nabla f(x^t)\| \overset{\eqref{eq:clipped_SGD_convex_technical_4}}{\leq} \sqrt{2}R (1+L\gamma) \overset{\eqref{eq:clipped_SGD_step_size_cvx}}{\leq} 2R.\notag
    \end{eqnarray}
    The latter inequality means that $E_{T-1}$ implies $\eta_t = x^t - x^* - \gamma \nabla f(x^t)$ for $t = 0,\ldots, T-1$.
    Next, we define the unbiased part and the bias of $\theta_{t}$ as $\theta_{t}^u$ and $\theta_{t}^b$, respectively:
    \begin{equation}
        \theta_{t}^u = \clip(\nabla f_{\Xi^t}(x^{t}), \lambda_t) - \EE_{\Xi^t}\left[\clip(\nabla f_{\Xi^t}(x^{t}), \lambda_t)\right],\quad \theta_{t}^b = \EE_{\Xi^t}\left[\clip(\nabla f_{\Xi^t}(x^{t}), \lambda_t)\right] - \nabla f(x^{t}). \label{eq:clipped_SGD_convex_theta_u_b}
    \end{equation}
    We notice that $\theta_{t} = \theta_{t}^u + \theta_{t}^b$. Using new notation, we get that $E_{T-1}$ implies
    \begin{eqnarray}
        R_T^2 
        &\leq& R^2 \underbrace{-2\gamma\sum\limits_{t=0}^{T-1}\langle \theta_t^u, \eta_t\rangle}_{\circledOne}  \underbrace{-2\gamma\sum\limits_{t=0}^{T-1}\langle \theta_t^b, \eta_t\rangle}_{\circledTwo} + \underbrace{2\gamma^2\sum\limits_{t=0}^{T-1}\left(\left\|\theta_{t}^u\right\|^2 - \EE_{\Xi^t}\left[\left\|\theta_{t}^u\right\|^2\right]\right)}_{\circledThree}\notag\\
        &&\quad + \underbrace{2\gamma^2\sum\limits_{t=0}^{T-1}\EE_{\Xi^t}\left[\left\|\theta_{t}^u\right\|^2\right]}_{\circledFour} + \underbrace{2\gamma^2\sum\limits_{t=0}^{T-1}\left\|\theta_{t}^b\right\|^2}_{\circledFive}. \label{eq:clipped_SGD_convex_technical_7}
    \end{eqnarray}

    To conclude our inductive proof successfully, we must obtain sufficiently strong upper bounds with high probability for the terms $\circledOne, \circledTwo, \circledThree, \circledFour, \circledFive$. In other words, we need to demonstrate that $\circledOne + \circledTwo + \circledThree + \circledFour + \circledFive \leq R^2$ with a high probability. In the subsequent stages of the proof, we will rely on the bounds for the norms and second moments of $\theta_{t}^u$ and $\theta_{t}^b$. First, as per the definition of the clipping operator, we can assert with probability $1$ that
     \begin{equation}
        \|\theta_{t}^u\| \leq 2\lambda. \label{eq:clipped_SGD_convex_norm_theta_u_bound}
    \end{equation}
    Furthermore, given that $E_{T-1}$ implies $\|\EE_{\Xi^t}[\nabla f_{\Xi^t}(x^{t})]\| \leq \nicefrac{\lambda}{2}$ for $t = 0,1,\ldots,T-1$ (as per \eqref{eq:clipped_SGD_convex_technical_4_bias}), then, according to Lemma~\ref{lem:bias_and_variance_clip}, we can deduce that $E_{T-1}$ implies
    \begin{eqnarray}
        \|\theta_{t}^b\| &\leq& \left\|\EE_{\Xi^t}\left[\clip(\nabla f_{\Xi^t}(x^{t}), \lambda_t)\right] - \EE_{\Xi^t}[\nabla f_{\Xi^t}(x^{t})]\right\| +  \left\|\EE_{\Xi^t}[\nabla f_{\Xi^t}(x^{t})] - \nabla f(x^{t})\right\|\notag\\
        &\overset{\eqref{eq:bias_X}, \eqref{eq:bias_bound}}{\leq}& \frac{4\sigma^2}{\lambda} + b_t, \label{eq:clipped_SGD_convex_norm_theta_b_bound} \\
        \EE_{\Xi^t}\left[\|\theta_{t}^u\|^2\right] &\overset{\eqref{eq:variance_X}}{\leq}& 18 \sigma^2. \label{eq:clipped_SGD_convex_second_moment_theta_u_bound}
    \end{eqnarray}

    \textbf{Upper bound for $\circledOne$.} By definition of $\theta_{t}^u$, we readily observe that $\EE_{\Xi^t}[\theta_{t}^u] = 0$ and
    \begin{equation}
        \EE_{\Xi^t}\left[-2\gamma\langle\theta_t^u, \eta_t\rangle\right] = 0. \notag
    \end{equation}
    Next, the sum $\circledOne$ contains only the terms that are bounded with probability $1$:
    \begin{equation}
        |2\gamma\left\langle \theta_{t}^u, \eta_t\right\rangle| \leq 2\gamma \|\theta_{t}^u\| \cdot \|\eta_t\| \overset{\eqref{eq:clipped_SGD_convex_technical_6},\eqref{eq:clipped_SGD_convex_norm_theta_u_bound}}{\leq} 8\gamma \lambda R\overset{\eqref{eq:clipped_SGD_clipping_level_cvx}}{=} \frac{R^2}{5\ln\frac{4(K+1)}{\delta}} \eqdef c. \label{eq:clipped_SGD_convex_technical_8} 
    \end{equation}
    The conditional variances $\sigma_t^2 \eqdef \EE_{\Xi^t}[4\gamma^2\langle\theta_t^u, \eta_t\rangle^2]$ of the summands are bounded:
    \begin{equation}
        \sigma_t^2 \leq \EE_{\Xi^t}\left[4\gamma^2\|\theta_{t}^u\|^2\cdot \|\eta_t\|^2\right] \overset{\eqref{eq:clipped_SGD_convex_technical_6}}{\leq} 16\gamma^2 R^2 \EE_{\Xi^t}\left[\|\theta_{t}^u\|^2\right]. \label{eq:clipped_SGD_convex_technical_9}
    \end{equation}
    To summarize, we have demonstrated that $\{-2\gamma\left\langle \theta_{t}^u, \eta_t\right\rangle\}_{t=0}^{T-1}$ is a bounded martingale difference sequence with bounded conditional variances $\{\sigma_t^2\}_{t=0}^{T-1}$. Therefore, one can apply Bernstein's inequality (Lemma~\ref{lem:Bernstein_ineq}) with $X_t = -2\gamma \left\langle \theta_{t}^u, \eta_t\right\rangle$, parameter $c$ as in \eqref{eq:clipped_SGD_convex_technical_8}, $b = \frac{R^2}{5}$, $G = \frac{R^4}{150\ln\frac{4(K+1)}{\delta}}$ and get
    \begin{equation*}
        \Prob\left\{|\circledOne| > \frac{R^2}{5}\quad \text{and}\quad \sum\limits_{t=0}^{T-1} \sigma_{t}^2 \leq \frac{R^4}{150\ln\frac{4(K+1)}{\delta}}\right\} \leq 2\exp\left(- \frac{b^2}{2G + \nicefrac{2cb}{3}}\right) = \frac{\delta}{2(K+1)},
    \end{equation*}
    which is equivalent to
    \begin{equation}
        \Prob\left\{ E_{\circledOne} \right\} \geq 1 - \frac{\delta}{2(K+1)},\quad \text{for}\quad E_{\circledOne} = \left\{ \text{either} \quad  \sum\limits_{t=0}^{T-1} \sigma_{t}^2 > \frac{R^4}{150\ln\frac{4(K+1)}{\delta}} \quad \text{or}\quad |\circledOne| \leq \frac{R^2}{5}\right\}. \label{eq:clipped_SGD_convex_sum_1_upper_bound}
    \end{equation}
    Additionally, event $E_{T-1}$ implies that
    \begin{eqnarray}
        \sum\limits_{t=0}^{T-1} \sigma_{t}^2 &\overset{\eqref{eq:clipped_SGD_convex_technical_9}}{\leq}& 16\gamma^2 R^2 \sum\limits_{t=0}^{T-1}  \EE_{\Xi^t}\left[\|\theta_{t}^u\|^2\right] \overset{\eqref{eq:clipped_SGD_convex_second_moment_theta_u_bound}}{\leq} 288\gamma^2 R^2 \sigma^{2}T\notag\\
        &\overset{\eqref{eq:clipped_SGD_step_size_cvx}}{\leq}&  \frac{R^4}{150 \ln\frac{4(K+1)}{\delta}}. \label{eq:clipped_SGD_convex_sum_1_variance_bound}
    \end{eqnarray}

    \textbf{Upper bound for $\circledTwo$.} From $E_{T-1}$ it follows that
\begin{eqnarray}
        \circledTwo &=& -2\gamma\sum\limits_{t=0}^{T-1}\langle \theta_t^b, \eta_t \rangle \leq 2\gamma\sum\limits_{t=0}^{T-1}\|\theta_{t}^b\|\cdot \|\eta_t\| \overset{\eqref{eq:clipped_SGD_convex_technical_6},\eqref{eq:clipped_SGD_convex_norm_theta_b_bound}}{\leq}  \frac{16 \gamma \sigma^2 T R}{\lambda} + 4\gamma R b T \notag\\ &\overset{\eqref{eq:clipped_SGD_clipping_level_cvx}}{=}& \frac{2\gamma^2\sigma^2 T \ln \frac{4(K+1)}{\delta}}{5} +  4\gamma R b T \overset{\eqref{eq:clipped_SGD_step_size_cvx}}{\leq} \frac{R^2}{5}. \label{eq:clipped_SGD_convex_sum_2_upper_bound}
\end{eqnarray}

\textbf{Upper bound for $\circledThree$.} By construction, we have
    \begin{equation}
        \EE_{\Xi^t}\left[2\gamma^2\left(\left\|\theta_{t}^u\right\|^2 - \EE_{\Xi^t}\left[\left\|\theta_{t}^u\right\|^2\right]\right)\right] = 0. \notag
    \end{equation}
    Next, the sum $\circledThree$ contains only the terms that are bounded with probability $1$:
    \begin{eqnarray}
        \left|2\gamma^2\left(\left\|\theta_{t}^u\right\|^2 - \EE_{\Xi^t}\left[\left\|\theta_{t}^u\right\|^2\right]\right)\right| &\leq& 2\gamma^2\left( \|\theta_{t}^u\|^2 +   \EE_{\Xi^t}\left[\left\|\theta_{t}^u\right\|^2\right]\right)\notag\\
        &\overset{\eqref{eq:clipped_SGD_convex_norm_theta_u_bound}}{\leq}& 16\gamma^2\lambda^2\overset{\eqref{eq:clipped_SGD_clipping_level_cvx}}{=} \frac{R^2}{100\ln^2\frac{4(K+1)}{\delta}} \le \frac{R^2}{5 \ln\frac{4(K+1)}{\delta}}\eqdef c. \label{eq:clipped_SGD_convex_technical_10}
    \end{eqnarray}
    The conditional variances $\widetilde\sigma_t^2 \eqdef \EE_{\Xi^t}\left[4\gamma^4\left(\left\|\theta_{t}^u\right\|^2 - \EE_{\Xi^t}\left[\left\|\theta_{t}^u\right\|^2\right]\right)^2\right]$ of the summands are bounded:
    \begin{eqnarray}
        \widetilde\sigma_t^2 &\overset{\eqref{eq:clipped_SGD_convex_technical_10}}{\leq}& \frac{R^2}{5 \ln\frac{4(K+1)}{\delta}} \EE_{\Xi^t}\left[2\gamma^2\left|\left\|\theta_{t}^u\right\|^2 - \EE_{\Xi^t}\left[\left\|\theta_{t}^u\right\|^2\right]\right|\right] \leq \frac{4\gamma^2 R^2}{5\ln\frac{4(K+1)}{\delta}} \EE_{\Xi^t}\left[\|\theta_{t}^u\|^2\right]. \label{eq:clipped_SGD_convex_technical_11}
    \end{eqnarray}
    To summarize, we have demonstrated that $\left\{2\gamma^2\left(\left\|\theta_{t}^u\right\|^2 - \EE_{\Xi^t}\left[\left\|\theta_{t}^u\right\|^2\right]\right)\right\}_{t=0}^{T-1}$ is a bounded martingale difference sequence with bounded conditional variances $\{\widetilde\sigma_t^2\}_{t=0}^{T-1}$. Therefore, one can apply Bernstein's inequality (Lemma~\ref{lem:Bernstein_ineq}) with $X_t = 2\gamma^2\left(\left\|\theta_{t}^u\right\|^2 - \EE_{\Xi^t}\left[\left\|\theta_{t}^u\right\|^2\right]\right)$, parameter $c$ as in \eqref{eq:clipped_SGD_convex_technical_10}, $b = \frac{R^2}{5}$, $G = \frac{R^4}{150\ln\frac{4(K+1)}{\delta}}$ and get
    \begin{equation*}
        \Prob\left\{|\circledThree| > \frac{R^2}{5}\quad \text{and}\quad \sum\limits_{t=0}^{T-1} \widetilde\sigma_{t}^2 \leq \frac{R^4}{150\ln\frac{4(K+1)}{\delta}}\right\} \leq 2\exp\left(- \frac{b^2}{2G + \nicefrac{2cb}{3}}\right) = \frac{\delta}{2(K+1)},
    \end{equation*}
    which is equivalent to
    \begin{equation}
        \Prob\left\{ E_{\circledThree} \right\} \geq 1 - \frac{\delta}{2(K+1)},\quad \text{for}\quad E_{\circledThree} = \left\{ \text{either} \quad  \sum\limits_{t=0}^{T-1} \widetilde\sigma_{t}^2 > \frac{R^4}{150\ln\frac{4(K+1)}{\delta}} \quad \text{or}\quad |\circledThree| \leq \frac{R^2}{5}\right\}. \label{eq:clipped_SGD_convex_sum_3_upper_bound}
    \end{equation}
    Additionally, event $E_{T-1}$ implies that
    \begin{eqnarray}
        \sum\limits_{t=0}^{T-1} \widetilde\sigma_{t}^2 &\overset{\eqref{eq:clipped_SGD_convex_technical_11}}{\leq}& \frac{4\gamma^2R^2}{5\ln\frac{4(K+1)}{\delta}} \sum\limits_{t=0}^{T-1}  \EE_{\Xi^t}\left[\|\theta_{t}^u\|^2\right] \overset{\eqref{eq:clipped_SGD_convex_second_moment_theta_u_bound}}{\leq}  
        \frac{72\gamma^2R^2\sigma^{2}T}{5\ln\frac{4(K+1)}{\delta}}\notag\\
        &\overset{\eqref{eq:clipped_SGD_step_size_cvx}}{\leq}& \frac{R^4}{150 \ln\frac{4(K+1)}{\delta}}. \label{eq:clipped_SGD_convex_sum_3_variance_bound}
    \end{eqnarray}

    \textbf{Upper bound for $\circledFour$.} From $E_{T-1}$ it follows that
\begin{eqnarray}
        \circledFour &=& 2\gamma^2\sum\limits_{t=0}^{T-1}\EE_{\Xi^t}\left[\left\|\theta_{t}^u\right\|^2\right] \overset{\eqref{eq:clipped_SGD_convex_second_moment_theta_u_bound}}{\leq} 36\gamma^2\sigma^{2}T \overset{\eqref{eq:clipped_SGD_step_size_cvx}}{\leq} \frac{R^2}{5}.\label{eq:clipped_SGD_convex_sum_4_upper_bound}
\end{eqnarray}

\textbf{Upper bound for $\circledFive$.} From $E_{T-1}$ it follows that
\begin{eqnarray}
        \circledFive &=& 2\gamma^2\sum\limits_{t=0}^{T-1}\left\|\theta_{t}^b\right\|^2 \overset{\eqref{eq:clipped_SGD_convex_norm_theta_b_bound}}{\leq} \frac{32\sigma^{4}T\gamma^2}{\lambda^{2}} + 2\gamma^2 b^2T \overset{\eqref{eq:clipped_SGD_clipping_level_cvx}}{=}   51200\cdot \frac{\sigma^{4}T\gamma^{4}\ln\frac{4(K+1)}{\delta}}{R^2} + 2\gamma^2 b^2T\overset{\eqref{eq:clipped_SGD_step_size_cvx}}{\leq} \frac{R^2}{5}.\label{eq:clipped_SGD_convex_sum_5_upper_bound}
\end{eqnarray}

That is, we derived the upper bounds for  $\circledOne, \circledTwo, \circledThree, \circledFour, \circledFive$. More specifically, the probability event $E_{T-1}$ implies:
\begin{gather*}
        R_T^2 \overset{\eqref{eq:clipped_SGD_convex_technical_7}}{\leq} R^2 + \circledOne + \circledTwo + \circledThree + \circledFour + \circledFive,\\
        \circledTwo \overset{\eqref{eq:clipped_SGD_convex_sum_2_upper_bound}}{\leq} \frac{R^2}{5},\quad \circledFour \overset{\eqref{eq:clipped_SGD_convex_sum_4_upper_bound}}{\leq} \frac{R^2}{5}, \quad \circledFive \overset{\eqref{eq:clipped_SGD_convex_sum_5_upper_bound}}{\leq} \frac{R^2}{5},\\
        \sum\limits_{t=0}^{T-1} \sigma_t^2 \overset{\eqref{eq:clipped_SGD_convex_sum_1_variance_bound}}{\leq} \frac{R^4}{150 \ln\frac{4(K+1)}{\delta}},\quad \sum\limits_{t=0}^{T-1} \widetilde\sigma_t^2 \overset{\eqref{eq:clipped_SGD_convex_sum_3_variance_bound}}{\leq} \frac{R^4}{150 \ln\frac{4(K+1)}{\delta}}.
\end{gather*}
    In addition, we also have (see \eqref{eq:clipped_SGD_convex_sum_1_upper_bound}, \eqref{eq:clipped_SGD_convex_sum_3_upper_bound} and our induction assumption)
\begin{equation*}
        \Prob\{E_{T-1}\} \geq 1 - \frac{(T-1)\delta}{K+1},\quad \Prob\{E_{\circledOne}\} \geq 1 - \frac{\delta}{2(K+1)},\quad \Prob\{E_{\circledThree}\} \geq 1 - \frac{\delta}{2(K+1)},
\end{equation*}
    where 
\begin{eqnarray*}
        E_{\circledOne} &=& \left\{ \text{either} \quad  \sum\limits_{t=0}^{T-1} \sigma_{t}^2 > \frac{R^4}{150\ln\frac{4(K+1)}{\delta}} \quad \text{or}\quad |\circledOne| \leq \frac{R^2}{5}\right\},\\
        E_{\circledThree} &=& \left\{ \text{either} \quad  \sum\limits_{t=0}^{T-1} \widetilde\sigma_{t}^2 > \frac{R^4}{150\ln\frac{4(K+1)}{\delta}} \quad \text{or}\quad |\circledThree| \leq \frac{R^2}{5}\right\}.
\end{eqnarray*}
    Therefore, probability event $E_{T-1} \cap E_{\circledOne} \cap E_{\circledThree}$ implies
\begin{eqnarray}
        R_T^2 &\leq& R^2 + \frac{R^2}{5} + \frac{R^2}{5} + \frac{R^2}{5} + \frac{R^2}{5} + \frac{R^2}{5} = 2R^2, \notag
\end{eqnarray}
    which is equivalent to \eqref{eq:clipped_SGD_convex_induction_inequality_1} and \eqref{eq:clipped_SGD_convex_induction_inequality_2} for $t = T$, and 
\begin{equation*}
        \Prob\{E_T\} \geq \Prob\left\{E_{T-1} \cap E_{\circledOne} \cap E_{\circledThree}\right\} = 1 - \Prob\left\{\overline{E}_{T-1} \cup \overline{E}_{\circledOne} \cup \overline{E}_{\circledThree}\right\} \geq 1 - \Prob\{\overline{E}_{T-1}\} - \Prob\{\overline{E}_{\circledOne}\} - \Prob\{\overline{E}_{\circledThree}\} \geq 1 - \frac{T\delta}{K+1}.
\end{equation*}
    We have now completed the inductive part of our proof. That is, for all $k = 0,1,\ldots,K+1$, we have $\Prob\{E_k\} \geq 1 - \nicefrac{k\delta}{(K+1)}$. Notably, when $k = K+1$, we can conclude that with a probability of at least $1 - \delta$:
\begin{equation*}
        f(\overline{x}^K) - f(x^*) \overset{\eqref{eq:clipped_SGD_convex_technical_1_1}}{\leq}\frac{2R^2}{\gamma(K+1)}
\end{equation*}
    and $\{x^k\}_{k=0}^{K} \subseteq Q$, which follows from \eqref{eq:clipped_SGD_convex_induction_inequality_2}.
    
Finally, if
\begin{equation*}
        \gamma = \min\left\{\frac{1}{160L \ln \frac{4(K+1)}{\delta}}, \frac{R}{208 \sigma  \sqrt{K\ln \frac{4(K+1)}{\delta}}}, \frac{R}{160b\ln \frac{4(K+1)}{\delta}}, \frac{R}{1600b(K+1)}\right\},
\end{equation*}
then with probability at least $1-\delta$
\begin{eqnarray*}
        f(\overline{x}^K) - f(x^*) &\leq& \frac{2R^2}{\gamma(K+1)}\\
        &=& \max\left\{\frac{320 LR^2 \ln \frac{4(K+1)}{\delta}}{K+1}, \frac{416\sigma R\sqrt{K\ln\frac{4(K+1)}{\delta}}}{K+1}, \frac{320 bR\ln\frac{4(K+1)}{\delta}}{K+1}, 3200bR\right\}\notag\\
        &=& \cO\left(\max\left\{\frac{L R^2 \ln \frac{K}{\delta}}{K}, \frac{ \sigma R\sqrt{\ln\frac{K}{\delta}}}{\sqrt{K}}, \frac{bR\ln\frac{K}{\delta}}{K}, bR\right\}\right).
\end{eqnarray*}
This concludes the proof.
\end{proof}

\subsection{Quasi-Strongly Convex Case}

\begin{lemma}\label{lem:main_opt_lemma_clipped_SGD_str_convex}
    Let Assumptions~\ref{as:L_smoothness} and \ref{as:QSC} with $\mu = 0$ hold on $Q = B_{2R}(x^*)$, where $R \geq \|x^0 - x^*\|$, and let stepsize  $\gamma_k \equiv \gamma$ satisfy $\gamma \leq \frac{1}{L}$. If $x^{k} \in Q$ for all $k = 0,1,\ldots,K$, $K \ge 0$, then after $K$ iterations of \algname{clipped-SGD} we have
    \begin{eqnarray}
        \|x^{K+1} - x^*\|^2 &\leq& \exp(-\gamma\mu(K+1))\|x^0 - x^*\|^2  - 2\gamma \sum\limits_{k=0}^K \exp(-\gamma\mu(K-k)) \langle x^k - x^* - \gamma \nabla f(x^k), \theta_k \rangle\notag\\
        &&\quad + \gamma^2 \sum\limits_{k=0}^K \exp(-\gamma\mu(K-k))\|\theta_k\|^2, \label{eq:main_opt_lemma_clipped_SGD_QSC}
    \end{eqnarray}
    where $\theta_k$ is defined in \eqref{eq:theta_k_def_clipped_SGD_non_convex}.
\end{lemma}
\begin{proof}
    Using the update rule of \algname{clipped-SGD}, we obtain
     \begin{eqnarray*}
        \|x^{k+1} - x^*\|^2 &=& \|x^k - x^*\|^2 - 2\gamma \langle x^k - x^*,  \clip(\nabla f_{\Xi^k}(x^k), \lambda_k)\rangle + \gamma^2\|\clip(\nabla f_{\Xi^k}(x^k), \lambda_k)\|^2\\
        &=& \|x^k - x^*\|^2 -2\gamma \langle x^k - x^*, \nabla f(x^k) \rangle - 2\gamma \langle x^k - x^*, \theta_k \rangle\\
        &&\quad + \gamma^2\|\nabla f(x^k)\|^2 + 2\gamma^2 \langle \nabla f(x^k), \theta_k \rangle + \gamma^2\|\theta_k\|^2\\
        &=& \|x^k - x^*\|^2 - 2\gamma \langle x^k - x^* - \gamma \nabla f(x^k), \theta_k \rangle\\
        &&\quad  - 2\gamma \langle x^k - x^*, \nabla f(x^k) \rangle + \gamma^2\|\nabla f(x^k)\|^2 + \gamma^2\|\theta_k\|^2\\
        &\overset{\eqref{eq:QSC}, \eqref{eq:L_smoothness_cor_2}}{\leq}&  (1-\gamma\mu)\|x^k - x^*\|^2 - 2\gamma \langle x^k - x^* - \gamma \nabla f(x^k), \theta_k \rangle\\
        &&\quad  - 2\gamma (f(x^k) - f(x^*)) + 2L\gamma^2 (f(x^k) - f(x^*)) + \gamma^2\|\theta_k\|^2\\
        &\overset{\gamma \leq \nicefrac{1}{L}}{\leq}& (1-\gamma\mu)\|x^k - x^*\|^2 - 2\gamma \langle x^k - x^* - \gamma \nabla f(x^k), \theta_k \rangle + \gamma^2\|\theta_k\|^2\\
        &&\quad  - 2\gamma (f(x^k) - f(x^*)) + 2L\gamma^2 (f(x^k) - f(x^*)) + \gamma^2\|\theta_k\|^2\\
        &\overset{\gamma \leq \nicefrac{1}{L}}{\leq}& (1-\gamma\mu)\|x^k - x^*\|^2 - 2\gamma \langle x^k - x^* - \gamma \nabla f(x^k), \theta_k \rangle + \gamma^2\|\theta_k\|^2\\
        &\leq& \exp(-\gamma\mu)\|x^k - x^*\|^2 - 2\gamma \langle x^k - x^* - \gamma \nabla f(x^k), \theta_k \rangle + \gamma^2\|\theta_k\|^2.
    \end{eqnarray*}
    Unrolling the recurrence, we obtain \eqref{eq:main_opt_lemma_clipped_SGD_QSC}.
\end{proof}

\begin{theorem}\label{thm:clipped_SGD_str_cvx_appendix}
    Let Assumptions~\ref{as:L_smoothness} and \ref{as:QSC} with $\mu > 0$ hold on $Q = B_{2R}(x^*)$, where $R \geq \|x^0 - x^*\|$. Assume that $\nabla f_{\Xi^k}(x^k)$ satisfies Assumption~\ref{as:bounded_bias_and_variance} with parameters $b_k, \sigma_k$ for $k = 0,1,\ldots,K$, $K > 0$ and
    \begin{eqnarray}
        0< \gamma &\leq& \min\left\{\frac{1}{400 L\ln \tfrac{4(K+1)}{\delta}}, \frac{\ln(B_K)}{\mu(K+1)}, \frac{\ln(C_K)}{\mu(1+ \nicefrac{K}{2})}, \frac{2\ln(D)}{\mu (K+1)}\right\}, \label{eq:gamma_SGDA_str_mon}\\
        B_K &=& \max\left\{2, \frac{(K+1)\mu^2R^2}{5400\sigma^2\ln\left(\frac{4(K+1)}{\delta}\right)\ln^2(B_K)} \right\} = \cO\!\left(\!\max\!\left\{2, \frac{K\mu^2R^2}{\sigma^2\ln\left(\!\frac{K}{\delta}\!\right)\ln^2\left(\!\max\!\left\{2, \frac{K\mu^2R^2}{\sigma^2\ln\left(\!\frac{K}{\delta}\!\right)} \right\}\right)} \right\}\!\right), \label{eq:B_K_SGD_str_cvx_2} \\
        C_K &=& \max\left\{2, \frac{(\frac{K}{2}+1)\mu R}{480 b\ln\left(\frac{4(K+1)}{\delta}\right)\ln(C_K)} \right\} = \cO\left(\max\left\{2, \frac{K\mu R}{b\ln\left(\frac{K}{\delta}\right)\ln\left(\max\left\{2, \frac{K\mu R}{b\ln\left(\frac{K}{\delta}\right)} \right\}\right)} \right\}\right), \label{eq:C_K_SGD_str_cvx_2} \\
        D &=& \max\left\{2, \frac{\mu R}{80 b\ln(D)} \right\} = \cO\left(\max\left\{2, \frac{\mu R}{b\ln\left(\max\left\{2, \frac{\mu R}{b} \right\}\right)} \right\}\right), \label{eq:D_K_SGD_str_cvx_2} \\
        \lambda_k &=& \frac{\exp(-\gamma\mu(1 + \nicefrac{k}{2}))R}{120\gamma \ln \tfrac{4(K+1)}{\delta}}, \label{eq:lambda_SGDA_str_mon}
    \end{eqnarray}
    for some $\delta \in (0,1]$ and $b = \max_{k=0,1,\ldots,K} b_k$, $\sigma = \max_{k=0,1,\ldots,K} \sigma_k$. Then, after $K$ iterations the iterates produced by \algname{clipped-SGD} with probability at least $1 - \delta$ satisfy 
    \begin{equation}
        \|x^{K+1} - x^*\|^2 \leq 2\exp(-\gamma\mu(K+1))R^2. \label{eq:main_result_str_cvx_SGD_appendix}
    \end{equation}
    In particular, when $\gamma$ equals the minimum from \eqref{eq:gamma_SGDA_str_mon}, then the iterates produced by \algname{clipped-SGD} after $K$ iterations with probability at least $1-\delta$ satisfy
    \begin{equation}
       \|x^{K} - x^*\|^2 = \cO\left(\max\left\{R^2\exp\left(- \frac{\mu K}{L \ln \tfrac{K}{\delta}}\right), \frac{\sigma^2\ln\left(\frac{K}{\delta}\right)\ln^2\left(B_K\right)}{K\mu^2}, \frac{bR\ln\left(\frac{K}{\delta}\right)\ln\left(C_K\right)}{K\mu}, \frac{bR\ln(D)}{\mu}\right\}\right). \label{eq:clipped_SGD_str_cvx_appendix}
    \end{equation}
\end{theorem}
\begin{proof}
    Our proof follows similar steps to the one given by \citet{sadiev2023high}. The main difference comes due to the presence of the bias in $\nabla f_{\Xi^k}(x^k)$. Therefore, for completeness, we provide the full proof here.

    Let $R_k = \|x^k - x^*\|$ for all $k\geq 0$. As in the previous results, the main part of the proof is inductive. More precisely, for each $k = 0,1,\ldots,K+1$ we consider probability event $E_k$ as follows: inequalities
    \begin{equation}
        R_t^2 \leq 2 \exp(-\gamma\mu t) R^2 \label{eq:induction_inequality_str_mon_SGDA}
    \end{equation}
    hold for $t = 0,1,\ldots,k$ simultaneously. We aim to demonstrate through induction that $\Prob\{E_k\} \geq 1 - \nicefrac{k\delta}{(K+1)}$ for all $k = 0,1,\ldots, K+1$. The base case, $k = 0$, is trivial. Assuming that the statement holds for some $k = T - 1 \leq K$, specifically, $\Prob\{E_{T-1}\} \geq 1 - \nicefrac{(T-1)\delta}{(K+1)}$, we need to establish that $\Prob\{E_{T}\} \geq 1 - \nicefrac{T\delta}{(K+1)}$. Since $R_t^2 \leq 2\exp(-\gamma\mu t) R^2 \leq 2R^2$, we have $x^t \in B_{2R}(x^*)$, where function $f$ is $L$-smooth. Thus, $E_{T-1}$ implies
    \begin{eqnarray}
        \|\nabla f(x^t)\| &\leq& L\|x^t - x^*\| \overset{\eqref{eq:induction_inequality_str_mon_SGDA}}{\leq} \sqrt{2}L\exp(- \nicefrac{\gamma\mu t}{2})R, \label{eq:operator_bound_x_t_SGDA_str_mon}\\
        \left\|\EE_{\Xi^t}[\nabla f_{\Xi^t}(x^t)]\right\| &\leq& \left\|\EE_{\Xi^t}[\nabla f_{\Xi^t}(x^t)] - \nabla f(x^t)\right\| + \|\nabla f(x^t)\| \overset{\eqref{eq:bias_bound},\eqref{eq:operator_bound_x_t_SGDA_str_mon}}{\leq} b + \sqrt{2}L\exp(- \nicefrac{\gamma\mu t}{2})R\notag\\
        &\overset{\eqref{eq:gamma_SGDA_str_mon},\eqref{eq:C_K_SGD_str_cvx_2},\eqref{eq:lambda_SGDA_str_mon}}{\leq}& \frac{\lambda_t}{2} \label{eq:operator_bound_x_t_SGDA_str_mon_bias} 
    \end{eqnarray}
    and
    \begin{eqnarray}
        \|\theta_t\|^2 &\leq& 2\|\tnabla f_{\Xi}(x^t)\|^2 + 2\|\nabla f(x^t)\|^2 \overset{\eqref{eq:operator_bound_x_t_SGDA_str_mon}}{\leq} \frac{5}{2}\lambda_t^2 \overset{\eqref{eq:lambda_SGDA_str_mon}}{\leq} \frac{\exp(-\gamma\mu t)R^2}{4\gamma^2} \label{eq:omega_bound_x_t_SGDA_str_mon}
    \end{eqnarray}
    for all $t = 0, 1, \ldots, T-1$, where we use that $\|a+b\|^2 \leq 2\|a\|^2 + 2\|b\|^2$ holding for all $a,b \in \R^d$. 

    Using Lemma~\ref{lem:main_opt_lemma_clipped_SGD_str_convex}, we obtain that $E_{T-1}$ implies
    \begin{eqnarray}
        R_T^2 &\leq& \exp(-\gamma\mu T)R^2 - 2\gamma \sum\limits_{t=0}^{T-1} \exp(-\gamma\mu(T-1-t)) \langle x^t - x^* - \gamma \nabla f(x^t), \theta_t \rangle\notag\\
        &&\quad + \gamma^2 \sum\limits_{t=0}^{T-1} \exp(-\gamma\mu(T-1-t))\|\theta_t\|^2. \notag
    \end{eqnarray}
    Next, we define random vectors
    \begin{gather}
        \eta_t = \begin{cases} x^t - x^* - \gamma \nabla f(x^t),& \text{if } \|x^t - x^* - \gamma \nabla f(x^t)\| \leq \sqrt{2}(1 + \gamma L) \exp(- \nicefrac{\gamma\mu t}{2})R,\\ 0,& \text{otherwise}, \end{cases} \label{eq:eta_t_SGDA_str_mon}
    \end{gather}
    for $t = 0, 1, \ldots, T-1$. As per their definition, these random vectors are bounded with probability 1
    \begin{equation}
         \|\eta_t\| \leq \sqrt{2}(1 + \gamma L)\exp(-\nicefrac{\gamma\mu t}{2})R \label{eq:eta_t_bound_SGDA_str_mon} 
    \end{equation}
    for all $t = 0, 1, \ldots, T-1$. Moreover, for $t = 0,\ldots, T-1$ event $E_{T-1}$ implies $\|\nabla f(x^t)\| \leq \sqrt{2}L\exp(-\nicefrac{\gamma\mu t}{2})R$ (due to \eqref{eq:operator_bound_x_t_SGDA_str_mon}) and
    \begin{eqnarray*}
        \|x^t - x^* - \gamma \nabla f(x^t)\| &\leq& \|x^t - x^*\| + \gamma \|\nabla f(x^t)\|\\
        &\overset{\eqref{eq:operator_bound_x_t_SGDA_str_mon}}{\leq}& \sqrt{2}(1 + \gamma L)\exp(-\nicefrac{\gamma\mu t}{2})R
    \end{eqnarray*}
    for $t = 0, 1, \ldots, T-1$. The latter inequality means that $E_{T-1}$ implies  $\eta_t = x^t - x^* - \gamma \nabla f(x^t)$ for all $t = 0,1,\ldots,T-1$, meaning that from $E_{T-1}$ it follows that
    \begin{eqnarray}
        R_T^2 &\leq& \exp(-\gamma\mu T)R^2 - 2\gamma \sum\limits_{t=0}^{T-1} \exp(-\gamma\mu(T-1-t)) \langle \eta_t, \theta_t \rangle + \gamma^2 \sum\limits_{t=0}^{T-1} \exp(-\gamma\mu(T-1-t)) \|\theta_t\|^2. \notag
    \end{eqnarray}
    Next, we define the unbiased part and the bias of $\theta_{t}$ as $\theta_{t}^u$ and $\theta_{t}^b$, respectively:
    \begin{gather}
        \theta_t^u \eqdef \clip(\nabla f_{\Xi^t}(x^t), \lambda_t) -  \EE_{\Xi^t}\left[\clip(\nabla f_{\Xi^t}(x^t), \lambda_t)\right],\quad \theta_t^b \eqdef \EE_{\Xi^t}\left[\clip(\nabla f_{\Xi^t}(x^t), \lambda_t)\right] - \nabla f(x^t), \label{eq:omega_unbias_bias_SGDA_str_mon}
    \end{gather}
    for all $t = 0,\ldots, T-1$. We notice that $\theta_{t} = \theta_{t}^u + \theta_{t}^b$. Using new notation, we get that $E_{T-1}$ implies
    \begin{eqnarray}
        R_T^2 &\leq& \exp(-\gamma\mu T) R^2  \underbrace{-2\gamma \sum\limits_{t=0}^{T-1} \exp(-\gamma\mu(T-1-t)) \langle \eta_t, \theta_t^u \rangle}_{\circledOne} \notag\\ &&\quad  \underbrace{-2\gamma \sum\limits_{t=0}^{T-1} \exp(-\gamma\mu(T-1-t)) \langle \eta_t, \theta_t^b \rangle}_{\circledTwo} 
        + \underbrace{2\gamma^2 \sum\limits_{t=0}^{T-1} \exp(-\gamma\mu(T-1-t)) \EE_{\Xi}\left[\|\theta^u_t\|^2\right]}_{\circledThree}\notag\\
        &&\quad + \underbrace{2\gamma^2 \sum\limits_{t=0}^{T-1} \exp(-\gamma\mu(T-1-t))\left( \|\theta^u_t\|^2 -  \EE_{\Xi}\left[\|\theta^u_t\|^2\right]\right)}_{\circledFour} + \underbrace{2\gamma^2 \sum\limits_{t=0}^{T-1} \exp(-\gamma\mu(T-1-t)) \|\theta^b_t\|^2}_{\circledFive}. \label{eq:SGDA_str_mon_12345_bound}
    \end{eqnarray}
    where we also apply inequality $\|a+b\|^2 \leq 2\|a\|^2 + 2\|b\|^2$ holding for all $a,b \in \R^d$ to upper bound $\|\theta_t\|^2$. To conclude our inductive proof successfully, we must obtain sufficiently strong upper bounds with high probability for the terms $\circledOne, \circledTwo, \circledThree, \circledFour, \circledFive$.  In other words, we need to demonstrate that $\circledOne + \circledTwo + \circledThree + \circledFour + \circledFive  \leq \exp(-\gamma\mu T) R^2$ with high probability. In the subsequent stages of the proof, we will rely on the bounds for the norms and second moments of $\theta_{t}^u$ and $\theta_{t}^b$. First, as per the definition of the clipping operator, we can assert with probability $1$ that
    \begin{equation}
        \|\theta_t^u\| \leq 2\lambda_t.\label{eq:omega_magnitude_str_mon}
    \end{equation}
    Furthermore, given that $E_{T-1}$ implies $\|\EE_{\Xi^t}[\nabla f_{\Xi^t}(x^{t})]\| \leq \nicefrac{\lambda}{2}$ for $t = 0,1,\ldots,T-1$ (as per \eqref{eq:operator_bound_x_t_SGDA_str_mon_bias}), then, according to Lemma~\ref{lem:bias_and_variance_clip}, we can deduce that $E_{T-1}$ implies
    \begin{eqnarray}
        \left\|\theta_t^b\right\| &\leq& \left\|\EE_{\Xi^t}\left[\clip(\nabla f_{\Xi^t}(x^t), \lambda_t)\right] - \EE_{\Xi^t}\left[\nabla f_{\Xi^t}(x^t)\right]\right\| + \left\|\EE_{\Xi^t}\left[\nabla f_{\Xi^t}(x^t)\right] - \nabla f(x^t)\right\| \notag\\
        &\leq& \frac{4\sigma^2}{\lambda_t} + b, \label{eq:bias_omega_str_mon}\\
        \EE_{\Xi^t}\left[\left\|\theta_t^u\right\|^2\right] &\leq& 18 \sigma^2, \label{eq:variance_omega_str_mon}
    \end{eqnarray}
    for all $t = 0,1, \ldots, T-1$.

\paragraph{Upper bound for $\circledOne$.} By definition of $\theta_{t}^u$, we readily observe that $\EE_{\Xi^t}[\theta_{t}^u] = 0$ and
    \begin{equation*}
        \EE_{\Xi^t}\left[-2\gamma \exp(-\gamma\mu(T-1-t)) \langle \eta_t, \theta_t^u \rangle\right] = 0.
    \end{equation*}
    Next, the sum $\circledOne$ contains only the terms that are bounded with probability $1$:
    \begin{eqnarray}
        |-2\gamma \exp(-\gamma\mu(T-1-t)) \langle \eta_t, \theta_t^u \rangle | &\leq& 2\gamma\exp(-\gamma\mu (T - 1 - t)) \|\eta_t\|\cdot \|\theta_t^u\|\notag\\
        &\overset{\eqref{eq:eta_t_bound_SGDA_str_mon},\eqref{eq:omega_magnitude_str_mon}}{\leq}& 4\sqrt{2}\gamma (1 + \gamma L) \exp(-\gamma\mu (T - 1 - \nicefrac{t}{2})) R \lambda_t\notag\\
        &\overset{\eqref{eq:gamma_SGDA_str_mon},\eqref{eq:lambda_SGDA_str_mon}}{\leq}& \frac{\exp(-\gamma\mu T)R^2}{5\ln\tfrac{4(K+1)}{\delta}} \eqdef c. \label{eq:SGDA_str_mon_technical_1_1}
    \end{eqnarray}
    The conditional variances $\sigma_t^2 \eqdef \EE_{\Xi^t}\left[4\gamma^2 \exp(-2\gamma\mu(T-1-t)) \langle \eta_t, \theta_t^u \rangle^2\right]$ of the summands are bounded:
    \begin{eqnarray}
        \sigma_t^2 &\leq& \EE_{\Xi^t}\left[4\gamma^2\exp(-\gamma\mu (2T - 2 - 2t)) \|\eta_t\|^2\cdot \|\theta_t^u\|^2\right]\notag\\
        &\overset{\eqref{eq:eta_t_bound_SGDA_str_mon}}{\leq}& 8\gamma^2 (1 + \gamma L)^2 \exp(-\gamma\mu (2T - 2 - t)) R^2 \EE_{\Xi^t}\left[\|\theta_t^u\|^2\right]\notag\\
        &\overset{\eqref{eq:gamma_SGDA_str_mon}}{\leq}& 10\gamma^2\exp(-\gamma\mu (2T - t))R^2 \EE_{\Xi^t}\left[\|\theta_t^u\|^2\right]. \label{eq:SGDA_str_mon_technical_1_2}
    \end{eqnarray}

    To summarize, we have demonstrated that $\{2\gamma (1-\gamma\mu)^{T-1-t} \langle \eta_t, \theta_t^u \rangle\}_{t = 0}^{T-1}$ is a bounded martingale difference sequence with bounded conditional variances $\{\sigma_t^2\}_{t = 0}^{T-1}$. Therefore, one can apply Bernstein's inequality (Lemma~\ref{lem:Bernstein_ineq}) with $X_t = 2\gamma (1-\gamma\mu)^{T-1-t} \langle \eta_t, \theta_t^u \rangle$, parameter $c$ as in \eqref{eq:SGDA_str_mon_technical_1_1}, $b = \tfrac{1}{5}\exp(-\gamma\mu T) R^2$, $G = \tfrac{\exp(-2 \gamma\mu T) R^4}{150\ln\frac{4(K+1)}{\delta}}$ and get
    \begin{equation*}
        \Prob\left\{|\circledOne| > \frac{1}{5}\exp(-\gamma\mu T) R^2 \text{ and } \sum\limits_{t=0}^{T-1}\sigma_t^2 \leq \frac{\exp(- 2\gamma\mu T) R^4}{150\ln\tfrac{4(K+1)}{\delta}}\right\} \leq 2\exp\left(- \frac{b^2}{2F + \nicefrac{2cb}{3}}\right) = \frac{\delta}{2(K+1)},
    \end{equation*}
    which is equivalent to
    \begin{equation}
        \Prob\{E_{\circledOne}\} \geq 1 - \frac{\delta}{2(K+1)},\quad \text{for}\quad E_{\circledOne} = \left\{\text{either} \quad \sum\limits_{t=0}^{T-1}\sigma_t^2 > \frac{\exp(- 2\gamma\mu T) R^4}{150\ln\tfrac{4(K+1)}{\delta}}\quad \text{or}\quad |\circledOne| \leq \frac{1}{5}\exp(-\gamma\mu T) R^2\right\}. \label{eq:bound_1_SGDA_str_mon}
    \end{equation}
    Additionally, event $E_{T-1}$ implies that
    \begin{eqnarray}
        \sum\limits_{t=0}^{T-1}\sigma_t^2 &\overset{\eqref{eq:SGDA_str_mon_technical_1_2}}{\leq}& 10\gamma^2\exp(- 2\gamma\mu T)R^2\sum\limits_{t=0}^{T-1} \frac{\EE_{\Xi^t}\left[\|\theta_t^u\|^2\right]}{\exp(-\gamma\mu t)}\notag\\ 
        &\overset{\eqref{eq:variance_omega_str_mon}, T \leq K+1}{\leq}& 180\gamma^2\exp(-2\gamma\mu T) R^2 \sigma^2 \sum\limits_{t=0}^{K} \frac{1}{\exp(-\gamma\mu t)}\notag\\
        &\overset{\eqref{eq:lambda_SGDA_str_mon}}{\leq}& 180\gamma^2\exp(-2\gamma\mu T) R^{2} \sigma^2 (K+1)\exp(\gamma\mu K)\notag\\
        &\overset{\eqref{eq:gamma_SGDA_str_mon}, \eqref{eq:B_K_SGD_str_cvx_2}}{\leq}& \frac{\exp(-2\gamma\mu T)R^4}{150\ln\tfrac{4(K+1)}{\delta}}. \label{eq:bound_1_variances_SGDA_str_mon}
    \end{eqnarray}

    \paragraph{Upper bound for $\circledTwo$.} From $E_{T-1}$ it follows that
    \begin{eqnarray}
        \circledTwo &\leq& 2\gamma \exp(-\gamma\mu (T-1)) \sum\limits_{t=0}^{T-1} \frac{\|\eta_t\|\cdot \|\theta_t^b\|}{\exp(-\gamma\mu t)}\notag\\
        &\overset{\eqref{eq:eta_t_bound_SGDA_str_mon}, \eqref{eq:bias_omega_str_mon}}{\leq}& \sqrt{2} \gamma (1+\gamma L) \exp(-\gamma\mu (T-1)) R  \sum\limits_{t=0}^{T-1} \exp(\nicefrac{\gamma\mu t}{2})\left(\frac{4\sigma^2}{\lambda_t} + b\right)\notag\\
        &\overset{\eqref{eq:gamma_SGDA_str_mon},\eqref{eq:lambda_SGDA_str_mon}}{\leq}& 3840 \gamma^2\exp(-\gamma\mu T) \sigma^{2} (K+1) \exp\left(\gamma\mu T\right) \ln\tfrac{4(K+1)}{\delta} \notag \\
        &&\quad +2\gamma \exp(-\gamma\mu T)R (K+1) \exp(\nicefrac{\gamma\mu T}{2})b\notag\\
        &\overset{\eqref{eq:gamma_SGDA_str_mon},\eqref{eq:B_K_SGD_str_cvx_2},\eqref{eq:D_K_SGD_str_cvx_2}}{\leq}& \frac{1}{5}\exp(-\gamma\mu T) R^2. \label{eq:bound_2_SGDA_str_mon}
    \end{eqnarray}

    \paragraph{Upper bound for $\circledThree$.} From $E_{T-1}$ it follows that
    \begin{eqnarray}
        \circledThree &=& 2\gamma^2 \exp(-\gamma\mu (T-1)) \sum\limits_{t=0}^{T-1} \frac{\EE_{\Xi^t}\left[\|\theta_t^u\|^2\right]}{\exp(-\gamma\mu t)} \notag\\
        &\overset{\eqref{eq:variance_omega_str_mon}}{\leq}& 144\gamma^2\exp(-\gamma\mu (T-1)) \sigma^2\sum\limits_{t=0}^{T-1} \frac{1}{\exp(-\gamma\mu t)} \notag\\
        &\overset{\eqref{eq:lambda_SGDA_str_mon}}{\leq}&  144\gamma^2 \exp(-\gamma\mu (T-1)) \sigma^2 (K+1)\exp(\gamma\mu K) \notag\\
        &\overset{\eqref{eq:gamma_SGDA_str_mon}}{\leq}& \frac{1}{5} \exp(-\gamma\mu T) R^2. \label{eq:bound_3_SGDA_str_mon}
    \end{eqnarray}

    \paragraph{Upper bound for $\circledFour$.} By construction, we have
    \begin{equation*}
        2\gamma^2 \exp(-\gamma\mu (T-1-t))\EE_{\Xi^t}\left[\|\theta_t^u\|^2  -\EE_{\Xi^t}\left[\|\theta_t^u\|^2\right] \right] = 0.
    \end{equation*}
    Next, the sum $\circledFour$ contains only the terms that are bounded with probability $1$:
    \begin{eqnarray}
        2\gamma^2 \exp(-\gamma\mu (T-1-t))\left| \|\theta_t^u\|^2  -\EE_{\Xi^t}\left[\|\theta_t^u\|^2\right] \right| 
        &\overset{\eqref{eq:omega_magnitude_str_mon}}{\leq}& \frac{16\gamma^2 \exp(-\gamma\mu T) \lambda_l^2}{\exp(-\gamma\mu (1+t))}\notag\\
        &\overset{\eqref{eq:lambda_SGDA_str_mon}}{\leq}& \frac{\exp(-\gamma\mu T)R^2}{5\ln\tfrac{4(K+1)}{\delta}}\notag\\
        &\eqdef& c. \label{eq:SGDA_str_mon_technical_4_1}
    \end{eqnarray}
    The conditional variances
    \begin{equation*}
        \widetilde\sigma_t^2 \eqdef \EE_{\Xi^t}\left[4\gamma^4 \exp(-2\gamma\mu (T-1-t)) \left|\|\theta_t^u\|^2  -\EE_{\Xi^t}\left[\|\theta_t^u\|^2\right] \right|^2 \right]
    \end{equation*}
    of the summands are bounded:
    \begin{eqnarray}
        \widetilde\sigma_t^2 &\overset{\eqref{eq:SGDA_str_mon_technical_4_1}}{\leq}& \frac{2\gamma^2\exp(-2\gamma\mu T)R^2}{5\exp(-\gamma\mu (1+t))\ln\tfrac{4(K+1)}{\delta}} \EE_{\Xi^t}\left[ \left|\|\theta_t^u\|^2  -\EE_{\Xi^t}\left[\|\theta_t^u\|^2\right] \right|\right]\notag\\
        &\leq& \frac{4\gamma^2\exp(-2\gamma\mu T)R^2}{5\exp(-\gamma\mu (1+t))\ln\tfrac{4(K+1)}{\delta}} \EE_{\Xi^t}\left[\|\theta_t^u\|^2\right]. \label{eq:SGDA_str_mon_technical_4_2}
    \end{eqnarray}
    To summarize, we have demonstrated that $\left\{2\gamma^2 (1-\gamma\mu)^{T-1-t}\left( \|\theta_t^u\|^2 -\EE_{\Xi^t}\left[\|\theta_t^u\|^2\right]\right)\right\}_{t = 0}^{T-1}$ is a bounded martingale difference sequence with bounded conditional variances $\{\widetilde\sigma_t^2\}_{t = 0}^{T-1}$. Therefore, one can apply Bernstein's inequality (Lemma~\ref{lem:Bernstein_ineq}) with $X_t = 2\gamma^2 (1-\gamma\mu)^{T-1-t}\left( \|\theta_t^u\|^2 -\EE_{\Xi^t}\left[\|\theta_t^u\|^2\right]\right)$, parameter $c$ as in \eqref{eq:SGDA_str_mon_technical_4_1}, $b = \tfrac{1}{5}\exp(-\gamma\mu T) R^2$, $G = \tfrac{\exp(-2 \gamma\mu T) R^4}{150\ln\frac{4(K+1)}{\delta}}$ and get
    \begin{equation*}
        \Prob\left\{|\circledFour| > \frac{1}{5}\exp(-\gamma\mu T) R^2 \text{ and } \sum\limits_{l=0}^{T-1}\widetilde\sigma_t^2 \leq \frac{\exp(-2\gamma\mu T) R^4}{150\ln\frac{4(K+1)}{\delta}}\right\} \leq 2\exp\left(- \frac{b^2}{2G + \nicefrac{2cb}{3}}\right) = \frac{\delta}{2(K+1)},
    \end{equation*}
    which is equivalent to
    \begin{equation}
        \Prob\{E_{\circledFour}\} \geq 1 - \frac{\delta}{2(K+1)},\quad \text{for}\quad E_{\circledFour} = \left\{\text{either} \quad \sum\limits_{t=0}^{T-1}\widetilde\sigma_t^2 > \frac{\exp(-2\gamma\mu T) R^4}{150\ln\tfrac{4(K+1)}{\delta}}\quad \text{or}\quad |\circledFour| \leq \frac{1}{5}\exp(-\gamma\mu T) R^2\right\}. \label{eq:bound_4_SGDA_str_mon}
    \end{equation}
    Additionally, event $E_{T-1}$ implies that
    \begin{eqnarray}
        \sum\limits_{l=0}^{T-1}\widetilde\sigma_t^2 &\overset{\eqref{eq:SGDA_str_mon_technical_4_2}}{\leq}& \frac{4\gamma^2\exp(-\gamma\mu (2T-1))R^2}{5\ln\tfrac{4(K+1)}{\delta}} \sum\limits_{t=0}^{T-1} \frac{\EE_{\Xi^t}\left[\|\theta_l^u\|^2\right]}{\exp(-\gamma\mu t)}\notag\\ &\overset{\eqref{eq:variance_omega_str_mon}, T \leq K+1}{\leq}& \frac{72\gamma^2\exp(-\gamma\mu (2T-1)) R^2 \sigma^2}{5\ln\tfrac{4(K+1)}{\delta}} \sum\limits_{t=0}^{K} \frac{1}{\exp(-\gamma\mu t)}\notag\\
        &\overset{\eqref{eq:lambda_SGDA_str_mon}}{\leq}& \frac{72\gamma^2\exp(-\gamma\mu (2T-1)) R^{2} \sigma^2 (K+1)\exp(\gamma\mu K)}{5\ln\tfrac{4(K+1)}{\delta}} \notag\\
        &\overset{\eqref{eq:gamma_SGDA_str_mon}}{\leq}& \frac{\exp(-2\gamma\mu T)R^4}{150\ln\tfrac{4(K+1)}{\delta}}. \label{eq:bound_4_variances_SGDA_str_mon}
    \end{eqnarray}

    \paragraph{Upper bound for $\circledFive$.} From $E_{T-1}$ it follows that
    \begin{eqnarray}
        \circledFive &=&  2\gamma^2 \sum\limits_{l=0}^{T-1} \exp(-\gamma\mu (T-1-t)) \|\theta_t^b\|^2\notag\\
        &\overset{\eqref{eq:bias_omega_str_mon}}{\leq}& \gamma^2 \exp(-\gamma\mu (T-1)) \sum\limits_{t=0}^{T-1} \exp(\gamma\mu t) \left(\frac{64\sigma^{4}}{\lambda_t^{2}} + 2b^2\right)\notag\\
        &\overset{\eqref{eq:lambda_SGDA_str_mon}, T \leq K+1}{\leq}& \frac{921 600\gamma^{4} \exp(-\gamma\mu (T-1)) \sigma^{4} \ln^{2}\tfrac{4(K+1)}{\delta}}{R^{2}} \sum\limits_{t=0}^{K} \exp\left(2\gamma\mu\left(1 + \frac{t}{2}\right)\right)\exp(\gamma\mu t)\notag\\
        && \quad + 2\gamma^2\exp(-\gamma\mu (T-1))b^2 \sum\limits_{t=0}^{K} \exp(\gamma\mu t)\notag\\ 
        &\leq& \frac{921 600\gamma^{4} \exp(-\gamma\mu (T-3)) \sigma^{4} \ln^{2}\tfrac{4(K+1)}{\delta} (K+1) \exp(2\gamma\mu K)}{R^{2}}\notag\\
        &&\quad + 2\gamma^2\exp(-\gamma\mu (T-1))b^2 \exp(\gamma\mu K) (K+1) \notag\\
        &\overset{\eqref{eq:gamma_SGDA_str_mon}, \eqref{eq:B_K_SGD_str_cvx_2}, \eqref{eq:D_K_SGD_str_cvx_2}}{\leq}& \frac{1}{5}\exp(-\gamma\mu T) R^2. \label{eq:bound_5_SGDA_str_mon}
    \end{eqnarray}

    That is, we derived the upper bounds for  $\circledOne, \circledTwo, \circledThree, \circledFour, \circledFive$. More specifically, the probability event $E_{T-1}$ implies:
    \begin{gather*}
        R_T^2 \overset{\eqref{eq:SGDA_str_mon_12345_bound}}{\leq} \exp(-\gamma\mu T) R^2 + \circledOne + \circledTwo + \circledThree + \circledFour + \circledFive ,\\
        \circledTwo \overset{\eqref{eq:bound_2_SGDA_str_mon}}{\leq} \frac{1}{5}\exp(-\gamma\mu T)R^2,\quad \circledThree \overset{\eqref{eq:bound_3_SGDA_str_mon}}{\leq} \frac{1}{5}\exp(-\gamma\mu T)R^2,\\ \circledFive \overset{\eqref{eq:bound_5_SGDA_str_mon}}{\leq} \frac{1}{5}\exp(-\gamma\mu T)R^2\\
        \sum\limits_{t=0}^{T-1}\sigma_t^2 \overset{\eqref{eq:bound_1_variances_SGDA_str_mon}}{\leq}  \frac{\exp(-2\gamma\mu T)R^4}{150\ln\tfrac{4(K+1)}{\delta}},\quad \sum\limits_{t=0}^{T-1}\widetilde\sigma_t^2 \overset{\eqref{eq:bound_4_variances_SGDA_str_mon}}{\leq} \frac{\exp(-2\gamma\mu T)R^4}{150\ln\tfrac{4(K+1)}{\delta}}.
    \end{gather*}
     Moreover, we also have (see \eqref{eq:bound_1_SGDA_str_mon}, \eqref{eq:bound_4_SGDA_str_mon} and our induction assumption)
     \begin{gather*}
        \Prob\{E_{T-1}\} \geq 1 - \frac{(T-1)\delta}{K+1},\\
        \Prob\{E_{\circledOne}\} \geq 1 - \frac{\delta}{2(K+1)}, \quad \Prob\{E_{\circledFour}\} \geq 1 - \frac{\delta}{2(K+1)}.
    \end{gather*}
    where
    \begin{eqnarray}
        E_{\circledOne}&=&  \left\{\text{either} \quad \sum\limits_{t=0}^{T-1}\sigma_t^2 > \frac{\exp(- 2\gamma\mu T) R^4}{150\ln\tfrac{4(K+1)}{\delta}}\quad \text{or}\quad |\circledOne| \leq \frac{1}{5}\exp(-\gamma\mu T) R^2\right\},\notag\\
        E_{\circledFour}&=& \left\{\text{either} \quad \sum\limits_{t=0}^{T-1}\widetilde\sigma_t^2 > \frac{\exp(-2\gamma\mu T) R^4}{150\ln\tfrac{4(K+1)}{\delta}}\quad \text{or}\quad |\circledFour| \leq \frac{1}{5}\exp(-\gamma\mu T) R^2\right\}.\notag
    \end{eqnarray}
     Therefore, probability event $E_{T-1} \cap E_{\circledOne} \cap E_{\circledFour} $ implies
    \begin{eqnarray*}
        R_T^2 &\overset{\eqref{eq:SGDA_str_mon_12345_bound}}{\leq}& \exp(-\gamma\mu T) R^2 + \circledOne + \circledTwo + \circledThree + \circledFour + \circledFive\\
        &\leq& 2\exp(-\gamma\mu T) R^2,
    \end{eqnarray*}
    which is equivalent to \eqref{eq:induction_inequality_str_mon_SGDA} for $t = T$, and
    \begin{equation}
        \Prob\{E_T\} \geq \Prob\{E_{T-1} \cap E_{\circledOne} \cap E_{\circledFour} \} = 1 - \Prob\{\overline{E}_{T-1} \cup \overline{E}_{\circledOne} \cup \overline{E}_{\circledFour} \} \geq 1 - \frac{T\delta}{K+1}. \notag
    \end{equation}
    We have now completed the inductive part of our proof. That is, for all $k = 0,1,\ldots,K+1$, we have $\Prob\{E_k\} \geq 1 - \nicefrac{k\delta}{(K+1)}$. Notably, when $k = K+1$, we can conclude that with a probability of at least $1 - \delta$:
    \begin{equation}
        \|x^{K+1} - x^*\|^2 \leq 2\exp(-\gamma\mu (K+1))R^2. \notag
    \end{equation}
    Finally, if 
    \begin{eqnarray*}
        \gamma &=& \min\left\{\frac{1}{400 L\ln \tfrac{4(K+1)}{\delta}}, \frac{\ln(B_K)}{\mu(K+1)}, \frac{\ln(C_K)}{\mu(1+ \nicefrac{K}{2})}, \frac{2\ln(D)}{\mu (K+1)}\right\}, \notag\\
        B_K &=& \max\left\{2, \frac{(K+1)\mu^2R^2}{5400\sigma^2\ln\left(\frac{4(K+1)}{\delta}\right)\ln^2(B_K)} \right\} = \cO\!\left(\!\max\!\left\{2, \frac{K\mu^2R^2}{\sigma^2\ln\left(\!\frac{K}{\delta}\!\right)\ln^2\left(\!\max\!\left\{2, \frac{K\mu^2R^2}{\sigma^2\ln\left(\!\frac{K}{\delta}\!\right)} \right\}\right)} \right\}\!\right),  \\
        C_K &=& \max\left\{2, \frac{(\frac{K}{2}+1)\mu R}{480 b\ln\left(\frac{4(K+1)}{\delta}\right)\ln(C_K)} \right\} = \cO\left(\max\left\{2, \frac{K\mu R}{b\ln\left(\frac{K}{\delta}\right)\ln\left(\max\left\{2, \frac{K\mu R}{b\ln\left(\frac{K}{\delta}\right)} \right\}\right)} \right\}\right), \\
        D &=& \max\left\{2, \frac{\mu R}{80 b\ln(D)} \right\} = \cO\left(\max\left\{2, \frac{\mu R}{b\ln\left(\max\left\{2, \frac{\mu R}{b} \right\}\right)} \right\}\right),
    \end{eqnarray*}
    then with probability at least $1-\delta$
    \begin{eqnarray*}
        \|x^{K+1} - x^*\|^2 &\leq& 2\exp(-\gamma\mu (K+1))R^2\\
        &\leq& 2R^2\max\left\{\exp\left(-\frac{\mu(K+1)}{400 L \ln \tfrac{4(K+1)}{\delta}}\right), \frac{1}{B_K}, \frac{1}{C_K}, \frac{1}{D} \right\}\\
        &=& \cO\left(\max\left\{R^2\exp\left(- \frac{\mu K}{L \ln \tfrac{K}{\delta}}\right), \frac{\sigma^2\ln\left(\frac{K}{\delta}\right)\ln^2\left(B_K\right)}{K\mu^2}, \frac{bR\ln\left(\frac{K}{\delta}\right)\ln\left(C_K\right)}{K\mu}, \frac{bR\ln(D)}{\mu}\right\}\right).
    \end{eqnarray*}
    This concludes the proof.
\end{proof}
\newpage

\section{Proofs for \algname{clipped-SSTM}}

\subsection{Convex Case}

The analysis of \algname{clipped-SSTM} in the convex case relies on the following lemma from \citep{sadiev2023high}.

\begin{lemma}[Lemma F.1 from \citep{sadiev2023high}]\label{lem:main_opt_lemma_clipped_SSTM}
    Let Assumptions~\ref{as:L_smoothness} and \ref{as:str_cvx} with $\mu = 0$ hold on $Q = B_{3R}(x^*)$, where $R \ge \|x^0 - x^*\|$, and let stepsize parameter $a$ satisfy $a\ge 1$. If $x^{k}, y^k, z^k \in B_{3R}(x^*)$ for all $k = 0,1,\ldots,N$, $N \ge 0$, then after $N$ iterations of \algname{clipped-SSTM} for all $z\in B_{3R}(x^*)$ we have
    \begin{eqnarray}
        A_N\left(f(y^N) - f(z)\right) &\le& \frac{1}{2}\|z^0 - z\|^2 - \frac{1}{2}\|z^{N} - z\|^2 + \sum\limits_{k=0}^{N-1}\alpha_{k+1}\left\langle \theta_{k+1}, z - z^{k} + \alpha_{k+1} \nabla f(x^{k+1})\right\rangle\notag\\
        &&\quad + \sum\limits_{k=0}^{N-1}\alpha_{k+1}^2\left\|\theta_{k+1}\right\|^2,\label{eq:main_opt_lemma_clipped_SSTM}\\
        \theta_{k+1} &\eqdef& \clip(\nabla f_{\Xi^k}(x^{k+1}), \lambda_k) - \nabla f(x^{k+1}).\label{eq:theta_k+1_def_clipped_SSTM}
    \end{eqnarray}
\end{lemma}

Next, we also use the following technical result from \citep{gorbunov2020stochastic}.

\begin{lemma}[Lemma~E.1 from \citep{gorbunov2020stochastic}]\label{lem:alpha_k_A_K_lemma}
    Let sequences $\{\alpha_k\}_{k\ge0}$ and $\{A_k\}_{k\ge 0}$ satisfy
    \begin{equation}
        \alpha_{0} = A_0 = 0,\quad A_{k+1} = A_k + \alpha_{k+1},\quad \alpha_{k+1} = \frac{k+2}{2aL}\quad \forall k \geq 0, \label{eq:alpha_k_A_k_def}
    \end{equation}
    where $a > 0$, $L > 0$. Then for all $k\ge 0$
    \begin{eqnarray}
        A_{k+1} &=& \frac{(k+1)(k+4)}{4aL}, \label{eq:A_k+1_explicit}\\
        A_{k+1} &\geq& a L \alpha_{k+1}^2. \label{eq:A_k+1_lower_bound}
    \end{eqnarray}
\end{lemma}

\begin{theorem}\label{thm:clipped_SSTM_cvx_appendix}
    Let Assumptions~\ref{as:L_smoothness} and \ref{as:str_cvx} with $\mu = 0$ hold on $Q = B_{3R}(x^*)$, where $R \geq \|x^0 - x^*\|$. Assume that $\nabla f_{\Xi^k}(x^{k+1})$ satisfies Assumption~\ref{as:bounded_bias_and_variance} with parameters $b_k, \sigma_k$ for $k = 0,1,\ldots,K$, $K > 0$ and
    \begin{gather}
        a \geq \max\left\{97200\ln^2\frac{4K}{\delta}, \frac{1800\sigma(K+1)\sqrt{K}\sqrt{\ln\frac{4K}{\delta}}}{LR}, \frac{4b(K+2)^2}{15LR}, \frac{60b(K+2)\ln\frac{4K}{\delta}}{LR}\right\},\label{eq:clipped_SSTM_parameter_a}\\
        \lambda_{k} = \frac{R}{30\alpha_{k+1}\ln\frac{4K}{\delta}}, \label{eq:clipped_SSTM_clipping_level}
    \end{gather}
    for some $\delta \in (0,1]$ and $b = \max_{k=0,1,\ldots,K} b_k$, $\sigma = \max_{k=0,1,\ldots,K} \sigma_k$. Then, after $K$ iterations of \algname{clipped-SSTM} the iterates with probability at least $1 - \delta$ satisfy
    \begin{equation}
        f(y^K) - f(x^*) \leq \frac{6aL R^2}{K(K+3)} \quad \text{and} \quad  \{x^k\}_{k=0}^{K+1}, \{z^k\}_{k=0}^{K}, \{y^k\}_{k=0}^K \subseteq B_{2R}(x^*). \label{eq:clipped_SSTM_convex_case_appendix}
    \end{equation}
    In particular, when parameter $a$ equals the maximum from \eqref{eq:clipped_SSTM_parameter_a}, then the iterates produced by \algname{clipped-SSTM} after $K$ iterations with probability at least $1-\delta$ satisfy
    \begin{equation}
        f(y^K) - f(x^*) = \cO\left(\max\left\{\frac{LR^2\ln^2\frac{K}{\delta}}{K^2}, \frac{\sigma R \sqrt{\ln\frac{K}{\delta}}}{\sqrt{K}}, \frac{bR\ln\frac{K}{\delta}}{K}, bR\right\}\right). \label{eq:clipped_SSTM_convex_case_2_appendix}
    \end{equation}
\end{theorem}
\begin{proof}
    Our proof follows similar steps to the one given by \citet{sadiev2023high}. The main difference comes due to the presence of the bias in $\nabla f_{\Xi^k}(x^k)$. Therefore, for completeness, we provide the full proof here.

    Let $R_k = \|z^k - x^*\|$, $\widetilde{R}_0 = R_0$, and $\widetilde{R}_{k+1} = \max\{\widetilde{R}_k, R_{k+1}\}$ for all $k\geq 0$. We will initially demonstrate through induction that for all $k\geq 0$, the iterates $x^{k+1}, z^k, y^k$ belong to $B_{\widetilde{R}_k}(x^*)$. The base of the induction is straightforward because $y^0 = z^0$, $\widetilde{R}_0 = R_0$, and $x^1 = \tfrac{A_0 y^0 + \alpha_1 z^0}{A_1} = z^0$. Now, assume that for some $l\ge 1$, $x^{l}, z^{l-1}, y^{l-1} \in B_{\widetilde{R}_{l-1}}(x^*)$. By the definitions of $R_l$ and $\widetilde{R}_l$, we have $z^l \in B_{R_{l}}(x^*)\subseteq B_{\widetilde{R}_{l}}(x^*)$. As $y^l$ is a convex combination of $y^{l-1}\in B_{\widetilde{R}_{l-1}}(x^*)\subseteq B_{\widetilde{R}_{l}}(x^*)$, it follows that $z^l\in B_{\widetilde{R}_{l}}(x^*)$ and, given the convex nature of $B_{\widetilde{R}_{l}}(x^*)$, we can conclude that $y^l \in B_{\widetilde{R}_{l}}(x^*)$. Finally, as $x^{l+1}$ is a convex combination of $y^l$ and $z^l$, it is evident that $x^{l+1}$ also lies in $B_{\widetilde{R}_{l}}(x^*)$.
    
    Our next objective is to establish, by induction, that $\widetilde{R}_{l} \leq 3R$ with high probability. This will enable us to apply the result from Lemma~\ref{lem:main_opt_lemma_clipped_SGD_convex} and subsequently utilize Bernstein's inequality to estimate the stochastic component of the upper bound. To be more precise, for each $k = 0,\ldots, K+1$, we consider the probability event $E_k$, defined as follows: inequalities
    \begin{gather}
        \sum\limits_{l=0}^{t-1}\alpha_{l+1}\left\langle\theta_{l+1}, x^* - z^{l} + \alpha_{l+1}\nabla f_{\Xi^l}(x^{l+1})\right\rangle + \sum\limits_{l=0}^{t-1}\alpha_{l+1}^2\left\|\theta_{l+1}\right\|^2 \leq R^2, \label{eq:clipped_SSTM_induction_inequality_1}\\
        R_t \leq 2R \label{eq:clipped_SSTM_induction_inequality_2}
    \end{gather}
    hold for all $t = 0,1,\ldots, k$ simultaneously. We aim to demonstrate through induction that $\Prob\{E_k\} \geq 1 - \nicefrac{k\delta}{(K+1)}$ for all $k = 0,1,\ldots, K+1$. The base case, $k = 0$, is trivial: the left-hand side of \eqref{eq:clipped_SSTM_induction_inequality_1} equals zero and $R \geq R_0$ by definition. Assuming that the statement holds for some $k = T - 1 \leq K$, specifically, $\Prob\{E_{T-1}\} \geq 1 - \nicefrac{(T-1)\delta}{(K+1)}$, we need to establish that $\Prob\{E_{T}\} \geq 1 - \nicefrac{T\delta}{(K+1)}$. 
    
   To begin, we observe that the probability event $E_{T-1}$ implies that $\widetilde{R}_t \leq 2R$ for all $t = 0,1,\ldots, T-1$. Moreover, it implies that
    \begin{eqnarray}
        \|z^{T} - x^*\| \overset{\eqref{eq:z_SSTM}}{\leq} \|z^{T} - x^*\| + \alpha_{T}\|\tnabla f_{\Xi^{T-1}}(x^{T})\| \leq 2R + \alpha_{T}\lambda_{T-1} \overset{\eqref{eq:clipped_SSTM_clipping_level}}{\leq} 3R.\notag
    \end{eqnarray}
    Hence, with $E_{T-1}$ implying $\{x^k\}_{k=0}^{T} \subseteq Q$, we confirm that the conditions of Lemma~\ref{lem:main_opt_lemma_clipped_SSTM} are met, resulting in
    \begin{eqnarray}
        A_t\left(f(y^t) - f(x^*)\right) &\le& \frac{1}{2}R_0^2 - \frac{1}{2}R_t^2 + \sum\limits_{l=0}^{t-1}\alpha_{l+1}\left\langle\theta_{l+1}, x^* - z^{l} + \alpha_{l+1} \nabla f(x^{l+1})\right\rangle + \sum\limits_{l=0}^{t-1}\alpha_{l+1}^2\left\|\theta_{l+1}\right\|^2 \label{eq:clipped_SSTM_technical_1}
    \end{eqnarray}
    for all $t = 0,1,\ldots, T$ simultaneously and for all $t = 1,\ldots, T-1$ this probability event also implies that
    \begin{eqnarray}
        f(y^t) - f(x^*) \overset{\eqref{eq:clipped_SSTM_induction_inequality_1}, \eqref{eq:clipped_SSTM_technical_1}}{\leq} \frac{\frac{1}{2}R_0^2 - \frac{1}{2}R_t^2 + R^2}{A_t} \leq \frac{3R^2}{2A_t} = \frac{6aLR^2}{t(t+3)}. \label{eq:clipped_SSTM_technical_1_1}
    \end{eqnarray}
    Considering that $f(y^{T}) -f(x^*) \geq 0$, we can further deduce from \eqref{eq:clipped_SSTM_technical_1} that when $E_{T-1}$ holds, the following holds as well:
    \begin{eqnarray}
        R_T^2 &\leq& R_0^2 + \underbrace{2\sum\limits_{t=0}^{T-1}\alpha_{t+1}\left\langle\theta_{t+1}, x^* - z^{t} + \alpha_{t+1} \nabla f(x^{t+1})\right\rangle + 2\sum\limits_{t=0}^{T-1}\alpha_{t+1}^2\left\|\theta_{t+1}\right\|^2}_{2B_T}\notag\\
        &\leq& R^2 + 2B_T. \label{eq:clipped_SSTM_technical_2}
    \end{eqnarray}
    
    Prior to our estimation of $B_T$, we need to establish several helpful inequalities. We start with showing that $E_{T-1}$ implies $\|\nabla f(x^{t+1})\| \leq \nicefrac{\lambda_t}{4}$ for all $t = 0,1,\ldots, T-1$. For $t = 0$ we have $x^1 = x^0$ and
    \begin{eqnarray}
        \|\nabla f(x^{1})\| = \|\nabla f(x^0)\| \overset{\eqref{eq:L_smoothness}}{\leq} L\|x^0 - x^*\| \leq \frac{R}{a\alpha_1} = \frac{\lambda_0}{2}\cdot \frac{60\ln\frac{4K}{\delta}}{a} \overset{\eqref{eq:clipped_SSTM_parameter_a}}{\leq} \frac{\lambda_0}{4}. \label{eq:clipped_SSTM_technical_3}
    \end{eqnarray}
    Next, for $t = 1,\ldots, T-1$ we have $\alpha_{t+1}(x^{t+1} - z^t) = A_t(y^t - x^{t+1})$ and event $E_{T-1}$ implies
    \begin{eqnarray}
        \|\nabla f(x^{t+1})\| &\leq& \|\nabla f(x^{t+1}) - \nabla f(y^t)\| + \|\nabla f(y^t)\| \notag\\
        &\overset{\eqref{eq:L_smoothness}, \eqref{eq:L_smoothness_cor_2}}{\leq}& L\|x^{t+1} - y^t\| + \sqrt{2L\left(f(y^t) - f(x^*)\right)} \notag\\
        &\overset{\eqref{eq:clipped_SSTM_technical_1_1}}{\leq}& \frac{L\alpha_{t+1}}{A_t}\|x^{t+1} - z^t\| + \sqrt{\frac{12aL^2R^2}{t(t+3)}}\notag\\
        &\leq& \frac{4LR\alpha_{t+1}}{A_t} + \sqrt{\frac{12aL^2R^2}{t(t+3)}} \notag\\
        &=&  \frac{R}{60\alpha_{t+1}\ln\frac{4K}{\delta}}\left(\frac{240L\alpha_{t+1}^2\ln\frac{4K}{\delta}}{A_t} + 60\sqrt{\frac{12aL^2\alpha_{t+1}^2\ln^2\frac{4K}{\delta}}{t(t+3)}}\right)\notag\\
        &\overset{\eqref{eq:A_k+1_explicit},\eqref{eq:clipped_SSTM_clipping_level}}{\leq}& \frac{\lambda_t}{2}\left(\frac{240L\left(\frac{t+2}{2aL}\right)^2\ln\frac{4K}{\delta}}{\frac{t(t+3)}{4aL}} + 60\sqrt{\frac{12aL^2\left(\frac{t+2}{2aL}\right)^2\ln^2\frac{4K}{\delta}}{t(t+3)}}\right) \notag\\
        &=& \frac{\lambda_t}{2} \left(\frac{240(t+2)^2\ln\frac{4K}{\delta}}{t(t+3)a} + 60\sqrt\frac{3(t+2)^2\ln\frac{4K}{\delta}}{t(t+3)a}\right) \notag\\
        &\leq& \frac{\lambda_t}{2} \left(\frac{540\ln\frac{4K}{\delta}}{a} + \frac{90\sqrt{3}\ln\frac{4K}{\delta}}{\sqrt{a}}\right) \overset{\eqref{eq:clipped_SSTM_parameter_a}}{\leq} \frac{\lambda_t}{4}, \label{eq:clipped_SSTM_technical_4}
    \end{eqnarray}
    where in the last row we use $\frac{(t+2)^2}{t(t+3)} \leq \frac{9}{4}$ for all $t \geq 1$. Therefore, probability event $E_{T-1}$ implies that
    \begin{eqnarray}
        \left\|\EE_{\Xi^t}[\nabla f_{\Xi^t}(x^{t+1})]\right\| \leq \left\|\EE_{\Xi^t}[\nabla f_{\Xi^t}(x^{t+1})] - \nabla f(x^{t+1})\right\| + \|\nabla f(x^{t+1})\| \leq b + \frac{\lambda_t}{4} \leq \frac{\lambda_t}{2} \label{eq:clipped_SSTM_technical_4_bias}
    \end{eqnarray}
    and
    \begin{eqnarray}
        \|x^* - z^{t} + \alpha_{t+1} \nabla f(x^{t+1})\| \leq \|x^* - z^{t}\| + \alpha_{t+1}\|\nabla f(x^{t+1})\| \overset{\eqref{eq:clipped_SSTM_induction_inequality_2}, \eqref{eq:clipped_SSTM_technical_3}, \eqref{eq:clipped_SSTM_technical_4}}{\leq} 2R + \frac{R}{60\ln\frac{4K}{\delta}} \leq 3R \label{eq:clipped_SSTM_technical_5}
    \end{eqnarray}
    for all $t = 0,1,\ldots, T-1$. Next, we define random vectors
    \begin{equation}
        \eta_t = \begin{cases} x^* - z^{t} + \alpha_{t+1} \nabla f(x^{t+1}),& \text{if } \|x^* - z^{t} + \alpha_{t+1} \nabla f(x^{t+1})\| \leq 3R,\\ 0,&\text{otherwise}, \end{cases} \notag
    \end{equation}
    for all $t = 0,1,\ldots, T-1$. As per their definition, these random vectors are bounded with probability 1
    \begin{equation}
        \|\eta_t\| \leq 3R. \label{eq:clipped_SSTM_technical_6}
    \end{equation}
    This means that $E_{T-1}$ implies $\eta_t = x^* - z^{t} + \alpha_{t+1} \nabla f(x^{t+1})$ for all $t = 0,1,\ldots, T-1$. Then, form $E_{T-1}$ it follows that
    \begin{eqnarray}
        B_T &=& \sum\limits_{t=0}^{T-1}\alpha_{t+1}\left\langle\theta_{t+1}, \eta_t\right\rangle + \sum\limits_{t=0}^{T-1}\alpha_{t+1}^2\left\|\theta_{t+1}\right\|^2. \notag
    \end{eqnarray}
    Next, we define the unbiased part and the bias of $\theta_{t}$ as $\theta_{t}^u$ and $\theta_{t}^b$, respectively:
    \begin{equation}
        \theta_{t}^u = \clip(\nabla f_{\Xi^t}(x^{t+1}), \lambda_t) - \EE_{\Xi^t}\left[\clip(\nabla f_{\Xi^t}(x^{t+1}), \lambda_t)\right],\quad \theta_{t}^b = \EE_{\Xi^t}\left[\clip(\nabla f_{\Xi^t}(x^{t+1}), \lambda_t)\right] - \nabla f(x^{t+1}). \label{eq:clipped_SSTM_convex_theta_u_b}
    \end{equation}
    We notice that $\theta_{t} = \theta_{t}^u + \theta_{t}^b$. Using new notation, we get that $E_{T-1}$ implies
    \begin{eqnarray}
        B_T &=& \sum\limits_{t=0}^{T-1}\alpha_{t+1}\left\langle\theta_{t+1}^u + \theta_{t+1}^b, \eta_t\right\rangle + \sum\limits_{t=0}^{T-1}\alpha_{t+1}^2\left\|\theta_{t+1}^u + \theta_{t+1}^b\right\|^2 \notag\\
        &\leq& \underbrace{\sum\limits_{t=0}^{T-1}\alpha_{t+1}\left\langle\theta_{t+1}^u, \eta_t\right\rangle}_{\circledOne} + \underbrace{\sum\limits_{t=0}^{T-1}\alpha_{t+1}\left\langle\theta_{t+1}^b, \eta_t\right\rangle}_{\circledTwo} + \underbrace{2\sum\limits_{t=0}^{T-1}\alpha_{t+1}^2\left(\left\|\theta_{t+1}^u\right\|^2 - \EE_{\Xi^t}\left[\left\|\theta_{t+1}^u\right\|^2\right]\right)}_{\circledThree}\notag\\
        &&\quad + \underbrace{2\sum\limits_{t=0}^{T-1}\alpha_{t+1}^2\EE_{\Xi^t}\left[\left\|\theta_{t+1}^u\right\|^2\right]}_{\circledFour} + \underbrace{2\sum\limits_{t=0}^{T-1}\alpha_{t+1}^2\left\|\theta_{t+1}^b\right\|^2}_{\circledFive}. \label{eq:clipped_SSTM_technical_7}
    \end{eqnarray}
    
    To conclude our inductive proof successfully, we must obtain sufficiently strong upper bounds with high probability for the terms $\circledOne, \circledTwo, \circledThree, \circledFour, \circledFive$. In other words, we need to demonstrate that $\circledOne + \circledTwo + \circledThree + \circledFour + \circledFive \leq R^2$ with a high probability. In the subsequent stages of the proof, we will rely on the bounds for the norms and second moments of $\theta_{t}^u$ and $\theta_{t}^b$. First, as per the definition of the clipping operator, we can assert with probability $1$ that
    \begin{equation}
        \|\theta_{t+1}^u\| \leq 2\lambda_{t}. \label{eq:clipped_SSTM_norm_theta_u_bound}
    \end{equation}
    Moreover, since $E_{T-1}$ implies that $\|\EE_{\Xi^t}\nabla f_{\Xi^t}(x^{t+1})\| \leq \nicefrac{\lambda_t}{2}$ for $t = 0,1,\ldots,T-1$ (see \eqref{eq:clipped_SSTM_technical_4_bias}), then, according to Lemma~\ref{lem:bias_and_variance_clip}, we can deduce that $E_{T-1}$ implies
    \begin{eqnarray}
        \|\theta_{t+1}^b\| &\leq& \left\|\EE_{\Xi^t}\left[\clip(\nabla f_{\Xi^t}(x^{t+1}), \lambda_t)\right] - \EE_{\Xi^t}\left[\nabla f_{\Xi^t}(x^{t+1})\right]\right\| + \left\|\EE_{\Xi^t}\left[\nabla f_{\Xi^t}(x^{t+1})\right] - \nabla f(x^{t+1})\right\|\notag\\
        &\leq& \frac{4\sigma^2}{\lambda_{t}} + b, \label{eq:clipped_SSTM_norm_theta_b_bound} \\
        \EE_{\Xi^t}\left[\|\theta_{t+1}^u\|^2\right] &\leq& 18 \sigma^2. \label{eq:clipped_SSTM_second_moment_theta_u_bound}
    \end{eqnarray}

    \textbf{Upper bound for $\circledOne$.} By definition of $\theta_{t}^u$, we readily observe that $\EE_{\Xi^t}[\theta_{t}^u] = 0$ and
    \begin{equation}
        \EE_{\Xi^t}\left[\alpha_{t+1}\left\langle\theta_{t+1}^u, \eta_t\right\rangle\right] = 0. \notag
    \end{equation}
    Next, the sum $\circledOne$ contains only the terms that are bounded with probability $1$:
    \begin{equation}
        |\alpha_{t+1}\left\langle\theta_{t+1}^u, \eta_t\right\rangle| \leq \alpha_{t+1} \|\theta_{t+1}^u\| \cdot \|\eta_t\| \overset{\eqref{eq:clipped_SSTM_technical_6},\eqref{eq:clipped_SSTM_norm_theta_u_bound}}{\leq} 6\alpha_{t+1}\lambda_t R \overset{\eqref{eq:clipped_SSTM_clipping_level}}{=} \frac{R^2}{5\ln\frac{4K}{\delta}} \eqdef c. \label{eq:clipped_SSTM_technical_8} 
    \end{equation}
    The conditional variances $\sigma_t^2 \eqdef \EE_{\Xi^t}[\alpha_{t+1}^2\left\langle\theta_{t+1}^u, \eta_t\right\rangle^2]$ of the summands are bounded:
    \begin{equation}
        \sigma_t^2 \leq \EE_{\Xi^t}\left[\alpha_{t+1}^2\|\theta_{t+1}^u\|^2\cdot \|\eta_t\|^2\right] \overset{\eqref{eq:clipped_SSTM_technical_6}}{\leq} 9\alpha_{t+1}^2R^2 \EE_{\Xi^t}\left[\|\theta_{t+1}^u\|^2\right]. \label{eq:clipped_SSTM_technical_9}
    \end{equation}
    To summarize, we have demonstrated that $\{\alpha_{t+1}\left\langle\theta_{t+1}^u, \eta_t\right\rangle\}_{t=0}^{T-1}$ is a bounded martingale difference sequence with bounded conditional variances $\{\sigma_t^2\}_{t=0}^{T-1}$. Therefore, one can apply Bernstein's inequality (Lemma~\ref{lem:Bernstein_ineq}) with $X_t = \alpha_{t+1}\left\langle\theta_{t+1}^u, \eta_t\right\rangle$, parameter $c$ as in \eqref{eq:clipped_SSTM_technical_8}, $b = \frac{R^2}{5}$, $G = \frac{R^4}{150\ln\frac{4K}{\delta}}$ and get
    \begin{equation*}
        \Prob\left\{|\circledOne| > \frac{R^2}{5}\quad \text{and}\quad \sum\limits_{t=0}^{T-1} \sigma_{t}^2 \leq \frac{R^4}{150\ln\frac{4K}{\delta}}\right\} \leq 2\exp\left(- \frac{b^2}{2G + \nicefrac{2cb}{3}}\right) = \frac{\delta}{2K},
    \end{equation*}
    which is equivalent to
    \begin{equation}
        \Prob\left\{ E_{\circledOne} \right\} \geq 1 - \frac{\delta}{2K},\quad \text{for}\quad E_{\circledOne} = \left\{ \text{either} \quad  \sum\limits_{t=0}^{T-1} \sigma_{t}^2 > \frac{R^4}{150\ln\frac{4K}{\delta}} \quad \text{or}\quad |\circledOne| \leq \frac{R^2}{5}\right\}. \label{eq:clipped_SSTM_sum_1_upper_bound}
    \end{equation}
    In addition, $E_{T-1}$ implies that
    \begin{eqnarray}
        \sum\limits_{t=0}^{T-1} \sigma_{t}^2 &\overset{\eqref{eq:clipped_SSTM_technical_9}}{\leq}& 9R^2 \sum\limits_{t=0}^{T-1} \alpha_{t+1}^2 \EE_{\Xi^t}\left[\|\theta_{t+1}^u\|^2\right] \overset{\eqref{eq:clipped_SSTM_second_moment_theta_u_bound}}{\leq} 162\sigma^2 R^2 \sum\limits_{t=0}^{T-1} \alpha_{t+1}^2\notag\\
        &=& \frac{162\sigma^2 R^{2}}{4 a^2 L^2}\sum\limits_{t=0}^{T-1} (t+2)^2 \notag\\
        &\leq& \frac{1}{a^2} \cdot \frac{81\sigma^2 R^{2} T(T+1)^2}{2 a^2 L^2} \overset{\eqref{eq:clipped_SSTM_parameter_a}}{\leq} \frac{R^4}{150 \ln\frac{4K}{\delta}}. \label{eq:clipped_SSTM_sum_1_variance_bound}
    \end{eqnarray}

    \textbf{Upper bound for $\circledTwo$.} From $E_{T-1}$ it follows that
    \begin{eqnarray}
        \circledTwo &\leq& \sum\limits_{t=0}^{T-1}\alpha_{t+1}\|\theta_{t+1}^b\|\cdot \|\eta_t\| \overset{\eqref{eq:clipped_SSTM_technical_6},\eqref{eq:clipped_SSTM_norm_theta_b_bound}}{\leq} 3R \sum\limits_{t=0}^{T-1}\alpha_{t+1} \left( \frac{4\sigma^2}{\lambda_t} + b\right) \notag\\
        &\overset{\eqref{eq:clipped_SSTM_clipping_level}}{\leq}& 360 \sigma^2 \ln\frac{4K}{\delta} \sum\limits_{t=0}^{T-1}\alpha_{t+1}^{2} + 3bR \sum\limits_{t=0}^{T-1}\alpha_{t+1}\notag\\
        &\leq& \frac{360\sigma^2 \ln\frac{4K}{\delta}}{4 a^{2} L^2}\sum\limits_{t=0}^{T-1} (t+2)^2 + \frac{3bR}{2aL}\sum\limits_{t=0}^{T-1}(t+2)\notag\\
        &\leq& \frac{360\sigma^2 \ln\frac{4K}{\delta} T(T+1)^2}{4 a^{2} L^2} + \frac{3bRT(T+1)}{2aL} \overset{\eqref{eq:clipped_SSTM_parameter_a}}{\leq} \frac{R^2}{5}. \label{eq:clipped_SSTM_sum_2_upper_bound}
    \end{eqnarray}

    \textbf{Upper bound for $\circledThree$.} First, we have
    \begin{equation}
        \EE_{\Xi^t}\left[2\alpha_{t+1}^2\left(\left\|\theta_{t+1}^u\right\|^2 - \EE_{\Xi^t}\left[\left\|\theta_{t+1}^u\right\|^2\right]\right)\right] = 0. \notag
    \end{equation}
    Next, the sum $\circledThree$ contains only the terms that are bounded with probability $1$:
    \begin{eqnarray}
        \left|2\alpha_{t+1}^2\left(\left\|\theta_{t+1}^u\right\|^2 - \EE_{\Xi^t}\left[\left\|\theta_{t+1}^u\right\|^2\right]\right)\right| &\leq& 2\alpha_{t+1}^2\left( \|\theta_{t+1}^u\|^2 +   \EE_{\Xi^t}\left[\left\|\theta_{t+1}^u\right\|^2\right]\right)\notag\\
        &\overset{\eqref{eq:clipped_SSTM_norm_theta_u_bound}}{\leq}& 16\alpha_{t+1}^2\lambda_t^2 \overset{\eqref{eq:clipped_SSTM_clipping_level}}{\leq} \frac{R^2}{5\ln\frac{4K}{\delta}} \eqdef c. \label{eq:clipped_SSTM_technical_10}
    \end{eqnarray}
    The conditional variances $\widetilde\sigma_t^2 \eqdef \EE_{\Xi^t}\left[4\alpha_{t+1}^4\left(\left\|\theta_{t+1}^u\right\|^2 - \EE_{\Xi^t}\left[\left\|\theta_{t+1}^u\right\|^2\right]\right)^2\right]$ of the summands are bounded:
    \begin{eqnarray}
        \widetilde\sigma_t^2 &\overset{\eqref{eq:clipped_SSTM_technical_10}}{\leq}& \frac{R^2}{5 \ln\frac{4K}{\delta}} \EE_{\Xi^t}\left[2\alpha_{t+1}^2\left|\left\|\theta_{t+1}^u\right\|^2 - \EE_{\Xi^k}\left[\left\|\theta_{t+1}^u\right\|^2\right]\right|\right] \leq \alpha_{t+1}^2R^2 \EE_{\Xi^t}\left[\|\theta_{t+1}^u\|^2\right], \label{eq:clipped_SSTM_technical_11}
    \end{eqnarray}
    To summarize, we have demonstrated that $\left\{2\alpha_{t+1}^2\left(\left\|\theta_{t+1}^u\right\|^2 - \EE_{\Xi^t}\left[\left\|\theta_{t+1}^u\right\|^2\right]\right)\right\}_{t=0}^{T-1}$ is a bounded martingale difference sequence with bounded conditional variances $\{\widetilde\sigma_t^2\}_{t=0}^{T-1}$. Therefore, one can apply Bernstein's inequality (Lemma~\ref{lem:Bernstein_ineq}) with $X_t = 2\alpha_{t+1}^2\left(\left\|\theta_{t+1}^u\right\|^2 - \EE_{\Xi^t}\left[\left\|\theta_{t+1}^u\right\|^2\right]\right)$, parameter $c$ as in \eqref{eq:clipped_SSTM_technical_10}, $b = \frac{R^2}{5}$, $G = \frac{R^4}{150\ln\frac{4K}{\delta}}$ and get
    \begin{equation*}
        \Prob\left\{|\circledThree| > \frac{R^2}{5}\quad \text{and}\quad \sum\limits_{t=0}^{T-1} \widetilde\sigma_{t}^2 \leq \frac{R^4}{150\ln\frac{4K}{\delta}}\right\} \leq 2\exp\left(- \frac{b^2}{2G + \nicefrac{2cb}{3}}\right) = \frac{\delta}{2K},
    \end{equation*}
    which is equivalent to
    \begin{equation}
        \Prob\left\{ E_{\circledThree} \right\} \geq 1 - \frac{\delta}{2K},\quad \text{for}\quad E_{\circledThree} = \left\{ \text{either} \quad  \sum\limits_{t=0}^{T-1} \widetilde\sigma_{t}^2 > \frac{R^4}{150\ln\frac{4K}{\delta}} \quad \text{or}\quad |\circledThree| \leq \frac{R^2}{5}\right\}. \label{eq:clipped_SSTM_sum_3_upper_bound}
    \end{equation}
    In addition, $E_{T-1}$ implies that
    \begin{eqnarray}
        \sum\limits_{t=0}^{T-1} \widetilde\sigma_{t}^2 &\overset{\eqref{eq:clipped_SSTM_technical_11}}{\leq}& R^2 \sum\limits_{t=0}^{T-1} \alpha_{t+1}^2 \EE_{\Xi^t}\left[\|\theta_{t+1}^u\|^2\right] \leq  9R^2 \sum\limits_{t=0}^{T-1} \alpha_{t+1}^2 \EE_{\Xi^t}\left[\|\theta_{t+1}^u\|^2\right] \overset{\eqref{eq:clipped_SSTM_sum_1_variance_bound}}{\leq} \frac{R^4}{150 \ln\frac{4K}{\delta}}. \label{eq:clipped_SSTM_sum_3_variance_bound}
    \end{eqnarray}

    \textbf{Upper bound for $\circledFour$.} From $E_{T-1}$ it follows that
    \begin{eqnarray}
        \circledFour &=& 2\sum\limits_{t=0}^{T-1}\alpha_{t+1}^2\EE_{\Xi^t}\left[\left\|\theta_{t+1}^u\right\|^2\right] \leq \frac{1}{R^2}\cdot 9R^2\sum\limits_{t=0}^{T-1}\alpha_{t+1}^2\EE_{\Xi^t}\left[\left\|\theta_{t+1}^u\right\|^2\right] \overset{\eqref{eq:clipped_SSTM_sum_1_variance_bound}}{\leq} \frac{R^2}{150\ln\frac{4K}{\delta}} \leq \frac{R^2}{5}.\label{eq:clipped_SSTM_sum_4_upper_bound}
    \end{eqnarray}

    \textbf{Upper bound for $\circledFive$.} From $E_{T-1}$ it follows that
    \begin{eqnarray}
        \circledFive &=& 2\sum\limits_{t=0}^{T-1}\alpha_{t+1}^2\left\|\theta_{t+1}^b\right\|^2 \leq 4 \sum\limits_{t=0}^{T-1}\alpha_{t+1}^2\left(\frac{16\sigma^4}{\lambda_{t}^{2}} + b^2\right)\notag\\
        &\overset{\eqref{eq:clipped_SSTM_clipping_level}}{=}& \frac{57600\sigma^4\ln^2 \frac{4K}{\delta}}{R^2}\sum\limits_{t=0}^{T-1}\alpha_{t+1}^4 + 4b^2\sum\limits_{t=0}^{T-1}\alpha_{t+1}^2 =  \frac{57600\sigma^4\ln^2 \frac{4K}{\delta}}{16a^4 L^4 R^2}\sum\limits_{t=0}^{T-1} (t+2)^4 + \frac{4b^2}{4a^2L^2}\sum\limits_{t=0}^{T-1}(t+2)^2 \notag\\
        &\leq&  \frac{57600\sigma^4\ln^2 \frac{4K}{\delta} T(T+1)^4}{16a^4 L^4 R^2} + \frac{4b^2 T(T+1)^2}{4a^2L^2} \overset{\eqref{eq:clipped_SSTM_parameter_a}}{\leq} \frac{R^2}{5}.\label{eq:clipped_SSTM_sum_5_upper_bound}
    \end{eqnarray}

    That is, we derived the upper bounds for  $\circledOne, \circledTwo, \circledThree, \circledFour, \circledFive$. More specifically, the probability event $E_{T-1}$ implies:
    \begin{gather*}
        B_T \overset{\eqref{eq:clipped_SSTM_technical_7}}{\leq} R^2 + \circledOne + \circledTwo + \circledThree + \circledFour + \circledFive,\\
        \circledTwo \overset{\eqref{eq:clipped_SSTM_sum_2_upper_bound}}{\leq} \frac{R^2}{5},\quad \circledFour \overset{\eqref{eq:clipped_SSTM_sum_4_upper_bound}}{\leq} \frac{R^2}{5}, \quad \circledFive \overset{\eqref{eq:clipped_SSTM_sum_5_upper_bound}}{\leq} \frac{R^2}{5},\\
        \sum\limits_{t=0}^{T-1} \sigma_t^2 \overset{\eqref{eq:clipped_SSTM_sum_1_variance_bound}}{\leq} \frac{R^4}{150 \ln\frac{4K}{\delta}},\quad \sum\limits_{t=0}^{T-1} \widetilde\sigma_t^2 \overset{\eqref{eq:clipped_SSTM_sum_3_variance_bound}}{\leq} \frac{R^4}{150 \ln\frac{4K}{\delta}}.
    \end{gather*}
    In addition, we also have (see \eqref{eq:clipped_SSTM_sum_1_upper_bound}, \eqref{eq:clipped_SSTM_sum_3_upper_bound} and our induction assumption)
    \begin{equation*}
        \Prob\{E_{T-1}\} \geq 1 - \frac{(T-1)\delta}{K},\quad \Prob\{E_{\circledOne}\} \geq 1 - \frac{\delta}{2K},\quad \Prob\{E_{\circledThree}\} \geq 1 - \frac{\delta}{2K},
    \end{equation*}
    where 
    \begin{eqnarray*}
        E_{\circledOne} &=& \left\{ \text{either} \quad  \sum\limits_{t=0}^{T-1} \sigma_{t}^2 > \frac{R^4}{150\ln\frac{4K}{\delta}} \quad \text{or}\quad |\circledOne| \leq \frac{R^2}{5}\right\},\\
        E_{\circledThree} &=& \left\{ \text{either} \quad  \sum\limits_{t=0}^{T-1} \widetilde\sigma_{t}^2 > \frac{R^4}{150\ln\frac{4K}{\delta}} \quad \text{or}\quad |\circledThree| \leq \frac{R^2}{5}\right\}.
    \end{eqnarray*}
    Therefore, probability event $E_{T-1} \cap E_{\circledOne} \cap E_{\circledThree}$ implies
    \begin{eqnarray}
        B_T &\leq& R^2 + \frac{R^2}{5} + \frac{R^2}{5} + \frac{R^2}{5} + \frac{R^2}{5} + \frac{R^2}{5} = 2R^2, \notag\\
        R_T^2 &\overset{\eqref{eq:clipped_SSTM_technical_2}}{\leq}& R^2 + 2R^2 \leq (2R)^2, \notag
    \end{eqnarray}
    which is equivalent to \eqref{eq:clipped_SSTM_induction_inequality_1} and \eqref{eq:clipped_SSTM_induction_inequality_2} for $t = T$, and 
    \begin{equation*}
        \Prob\{E_T\} \geq \Prob\left\{E_{T-1} \cap E_{\circledOne} \cap E_{\circledThree}\right\} = 1 - \Prob\left\{\overline{E}_{T-1} \cup \overline{E}_{\circledOne} \cup \overline{E}_{\circledThree}\right\} \geq 1 - \Prob\{\overline{E}_{T-1}\} - \Prob\{\overline{E}_{\circledOne}\} - \Prob\{\overline{E}_{\circledThree}\} \geq 1 - \frac{T\delta}{K}.
    \end{equation*}
    We have now completed the inductive part of our proof. That is, for all $k = 0,1,\ldots,K$, we have $\Prob\{E_k\} \geq 1 - \nicefrac{k\delta}{K}$. Notably, when $k = K$, we can conclude that with a probability of at least $1 - \delta$:
    \begin{equation*}
        f(y^K) - f(x^*) \overset{\eqref{eq:clipped_SSTM_technical_1_1}}{\leq} \frac{6aLR^2}{K(K+3)}
    \end{equation*}
    and $\{x^k\}_{k=0}^{K+1}, \{z^k\}_{k=0}^{K}, \{y^k\}_{k=0}^K \subseteq B_{2R}(x^*)$, which follows from \eqref{eq:clipped_SSTM_induction_inequality_2}.
    
    Finally, if
    \begin{equation*}
        a = \max\left\{97200\ln^2\frac{4K}{\delta}, \frac{1800\sigma(K+1)\sqrt{K}\sqrt{\ln\frac{4K}{\delta}}}{LR}, \frac{4b(K+2)^2}{15LR}, \frac{60b(K+2)\ln\frac{4K}{\delta}}{LR}\right\},
    \end{equation*}
    then with probability at least $1-\delta$
    \begin{eqnarray*}
        f(y^K) - f(x^*) &\leq& \frac{6aLR^2}{K(K+3)}\\
        &=& \cO\left(\max\left\{\frac{LR^2\ln^2\frac{K}{\delta}}{K^2}, \frac{\sigma R \sqrt{\ln\frac{K}{\delta}}}{\sqrt{K}}, \frac{bR\ln\frac{K}{\delta}}{K}, bR\right\}\right).
    \end{eqnarray*}
\end{proof}

\subsection{Strongly Convex Case}\label{appendix:restarted_clipped_sstm}
In the strongly convex case, we consider a restarted version of \algname{SSTM}, see Algorithm~\ref{alg:R-clipped-SSTM}.

\begin{algorithm}[h]
\caption{Restarted \algname{clipped-SSTM} (\algname{R-clipped-SSTM}) \citep{gorbunov2020stochastic}}
\label{alg:R-clipped-SSTM}   
\begin{algorithmic}[1]
\REQUIRE starting point $x^0$, number of restarts $\tau$, number of steps of \algname{clipped-SSTM} between restarts $\{K_t\}_{t=1}^{\tau}$, stepsize parameters $\{a_t\}_{t=1}^\tau$, clipping levels $\{\lambda_{k}^1\}_{k=0}^{K_1-1}$, $\{\lambda_{k}^2\}_{k=0}^{K_2-1}$, \ldots, $\{\lambda_{k}^\tau\}_{k=0}^{K_\tau - 1}$, smoothness constant $L$.
\STATE $\hat{x}^0 = x^0$
\FOR{$t=1,\ldots, \tau$}
\STATE Run \algname{clipped-SSTM} for $K_t$ iterations with stepsize parameter $a_{t}$, clipping levels $\{\lambda_{k}^t\}_{k=0}^{K_t-1}$, and starting point $\hat{x}^{t-1}$. Define the output of \algname{clipped-SSTM} by $\hat{x}^{t}$.
\ENDFOR
\ENSURE $\hat{x}^\tau$ 
\end{algorithmic}
\end{algorithm}

The main result for \algname{R-clipped-SSTM} is given below.

\begin{theorem}\label{thm:R_clipped_SSTM_str_cvx_appendix}
    Let Assumptions~\ref{as:L_smoothness} and \ref{as:str_cvx} with $\mu > 0$ hold on $Q = B_{3R}(x^*)$, where $R \geq \|x^0 - x^*\|$ and \algname{R-clipped-SSTM} runs \algname{clipped-SSTM} $\tau$ times. Assume that estimator $\nabla f_{\Xi^{k,t}}(x^{k+1,t})$ used in \algname{clipped-SSTM} at $k$-th iteration of stage $t$ satisfies Assumption~\ref{as:bounded_bias_and_variance} with parameters $b_k^t$ and $\sigma_k^t$ such that
    \begin{equation}
        b_k^t \leq \frac{15\mu R}{24\cdot 2^{t+1}}. \label{eq:R_clipped_SSTM_condition_on_bias}
    \end{equation}
    Let
    \begin{gather}
        K_t = \left\lceil \max\left\{ 2160\sqrt{\frac{LR_{t-1}^2}{\varepsilon_t}}\ln\frac{4320\sqrt{LR_{t-1}^2}\tau}{\sqrt{\varepsilon_t}\delta}, 4\left(\frac{5400\sigma^t R_{t-1}}{\varepsilon_t}\right)^2 \ln\left(\frac{8\tau}{\delta}\left(\frac{5400\sigma^t R_{t-1}}{\varepsilon_t}\right)^2\right) \right\} \right\rceil, \label{eq:R_clipped_SSTM_K_t} \\
        \varepsilon_t = \frac{\mu R_{t-1}^2}{4},\quad R_{t-1} = \frac{R}{2^{\nicefrac{(t-1)}{2}}}, \quad \tau = \left\lceil \log_2 \frac{\mu R^2}{2\varepsilon} \right\rceil, \label{eq:R_clipped_SSTM_epsilon_R_t_tau} \\
        a_t = \max\left\{97200\ln^2\frac{4K_t\tau}{\delta}, \frac{1800\sigma(K_t+1)\sqrt{K_t}\sqrt{\ln\frac{4K_t\tau}{\delta}}}{L R_t}, \frac{4b_t(K_t+2)^2}{15LR_t}, \frac{60b_t(K_t+2)\ln\frac{4K_t}{\delta}}{LR_t}\right\}, \label{eq:R_clipped_SSTM_parameter_a}\\
        \lambda_{k}^t = \frac{R_t}{30\alpha_{k+1}^t \ln\frac{4K_t\tau}{\delta}} \label{eq:R_clipped_SSTM_clipping_level}
    \end{gather}
    for $t = 1, \ldots, \tau$. Then to guarantee $f(\hat x^\tau) - f(x^*) \leq \varepsilon$ with probability $\geq 1 - \delta$ \algname{R-clipped-SSTM} requires
    \begin{equation}
        \cO\left(\max\left\{ \sqrt{\frac{L}{\mu}}\ln\left(\frac{\mu R^2}{\varepsilon}\right)\ln\left(\frac{\sqrt{L}}{\sqrt{\mu}\delta}\ln\left(\frac{\mu R^2}{\varepsilon}\right)\right), \frac{\sigma^2}{\mu \varepsilon} \ln\left(\frac{\sigma^2}{\mu \varepsilon \delta}\ln\left(\frac{\mu R^2}{\varepsilon}\right)\right)\right\} \right) \label{eq:R_clipped_SSTM_main_result_appendix}
    \end{equation}
    iterations. Moreover, with probability $\geq 1-\delta$ the iterates of \algname{R-clipped-SSTM} at stage $t$ stay in the ball $B_{2R_{t-1}}(x^*)$.
\end{theorem}
\begin{proof}
    The proof of this theorem follows the same steps as the one given for Theorem~F.3 from \citep{sadiev2023high}. By induction we derive that for any $t = 1,\ldots, \tau$ with probability at least $1 - \nicefrac{t\delta}{\tau}$ inequalities
    \begin{equation}
        f(\hat{x}^l) - f(x^*) \leq \varepsilon_l, \quad \|\hat x^l - x^*\|^2 \leq  R_l^2 = \frac{R^2}{2^l}
    \end{equation}
    hold for $l = 1,\ldots, t$ simultaneously. We start with the base of the induction. Theorem~\ref{thm:clipped_SSTM_cvx_appendix} implies that with probability at least $1 - \nicefrac{\delta}{\tau}$
    \begin{eqnarray}
        f(\hat x^1) - f(x^*) &\leq& \frac{6a_1LR^2}{K_1(K_1+3)} \notag\\
        &\overset{\eqref{eq:R_clipped_SSTM_parameter_a}}{=}& \max\Bigg\{\frac{583200 LR^2\ln^2\frac{4K_1\tau}{\delta}}{K_1(K_1+3)}, \frac{10 800\sigma R(K_1+1)\sqrt{K_1 \ln\frac{4K_1\tau}{\delta}}}{K_1(K_1+3)},\notag\\
        &&\quad\quad\quad\quad \frac{24 b_1 R (K_1 + 2)^2}{15K_1(K_1+3)}, \frac{360 b_1 R (K_1 + 2) \ln \frac{4K_1}{\delta}}{K_1(K_1 + 3)}\Bigg\} \notag\\
        &\leq& \max\Bigg\{\frac{583200 LR^2\ln^2\frac{4K_1\tau}{\delta}}{K_1^2}, \frac{10 800\sigma R\sqrt{\ln\frac{4K_1\tau}{\delta}}}{\sqrt{K_1}},\notag\\
        &&\quad\quad\quad\quad \frac{24 b_1 R}{15}, \frac{360 b_1R \ln \frac{4K_1}{\delta}}{K_1}\Bigg\}\notag\\
        &\overset{\eqref{eq:R_clipped_SSTM_condition_on_bias},\eqref{eq:R_clipped_SSTM_K_t}}{\leq}& \varepsilon_1 = \frac{\mu R^2}{4} \notag
    \end{eqnarray}
    and, due to the strong convexity,
    \begin{equation*}
        \|\hat x^1 - x^*\|^2 \leq \frac{2(f(\hat x^1) - f(x^*))}{\mu} \leq \frac{R^2}{2} = R_1^2.
    \end{equation*}
    The base of the induction is proven. Now, assume that the statement holds for some $t = T < \tau$, i.e., with probability at least $1 - \nicefrac{T\delta}{\tau}$ inequalities
    \begin{equation}
        f(\hat{x}^l) - f(x^*) \leq \varepsilon_l, \quad \|\hat x^l - x^*\|^2 \leq  R_l^2 = \frac{R^2}{2^l}
    \end{equation}
    hold for $l = 1,\ldots, T$ simultaneously. In particular, with probability at least $1 - \nicefrac{T\delta}{\tau}$ we have $\|\hat x^T - x^*\|^2 \leq R_T^2$. Applying Theorem~\ref{thm:clipped_SSTM_cvx_appendix} and using union bound for probability events, we get that with probability at least $1 - \nicefrac{(T+1)\delta}{\tau}$
    \begin{eqnarray}
        f(\hat x^{T+1}) - f(x^*) &\leq& \frac{6a_{T+1}LR_{T}^2}{K_{T+1}(K_{T+1}+3)}\notag\\
        &\overset{\eqref{eq:R_clipped_SSTM_parameter_a}}{=}& \max\Bigg\{\frac{583200 LR_T^2\ln^2\frac{4K_{T+1}\tau}{\delta}}{K_{T+1}(K_{T+1}+3)}, \frac{10 800\sigma R_T(K_{T+1}+1)\sqrt{K_{T+1} \ln\frac{4K_{T+1}\tau}{\delta}}}{K_{T+1}(K_{T+1}+3)},\notag\\
        &&\quad\quad\quad\quad \frac{24 b_{T+1} R_T (K_{T+1} + 2)^2}{15K_{T+1}(K_{T+1}+3)}, \frac{360 b_{T+1} R_T (K_{T+1} + 2) \ln \frac{4K_{T+1}}{\delta}}{K_{T+1}(K_{T+1} + 3)}\Bigg\} \notag\\
        &\leq& \max\Bigg\{\frac{583200 LR_T^2\ln^2\frac{4K_{T+1}\tau}{\delta}}{K_{T+1}^2}, \frac{10 800\sigma R_T\sqrt{\ln\frac{4K_{T+1}\tau}{\delta}}}{\sqrt{K_{T+1}}},\notag\\
        &&\quad\quad\quad\quad \frac{24 b_{T+1} R_T}{15}, \frac{360 b_{T+1} R_T \ln \frac{4K_{T+1}}{\delta}}{K_{T+1}}\Bigg\}\notag\\
        &\overset{\eqref{eq:R_clipped_SSTM_condition_on_bias},\eqref{eq:R_clipped_SSTM_K_t}}{\leq}& \varepsilon_{T+1} = \frac{\mu R_T^2}{4} \notag
    \end{eqnarray}
    and, due to the strong convexity,
    \begin{equation*}
        \|\hat x^{T+1} - x^*\|^2 \leq \frac{2(f(\hat x^{T+1}) - f(x^*))}{\mu} \leq \frac{R_T^2}{2} = R_{T+1}^2.
    \end{equation*}
    Thus, we finished the inductive part of the proof. In particular, with probability at least $1 - \delta$ inequalities
    \begin{equation}
        f(\hat{x}^l) - f(x^*) \leq \varepsilon_l, \quad \|\hat x^l - x^*\|^2 \leq  R_l^2 = \frac{R^2}{2^l}\notag
    \end{equation}
    hold for $l = 1,\ldots, \tau$ simultaneously, which gives for $l = \tau$ that with probability at least $1 - \delta$
    \begin{equation*}
        f(\hat{x}^\tau) - f(x^*) \leq \varepsilon_\tau = \frac{\mu R_{\tau-1}^2}{4} = \frac{\mu R^2}{2^{\tau+1}} \overset{\eqref{eq:R_clipped_SSTM_epsilon_R_t_tau}}{\leq} \varepsilon.
    \end{equation*}
    It remains to calculate the overall number of iterations during all runs of \algname{clipped-SSTM}. We have
    \begin{eqnarray*}
        \sum\limits_{t=1}^\tau K_t &=& \cO\left(\sum\limits_{t=1}^\tau \max\left\{ \sqrt{\frac{LR_{t-1}^2}{\varepsilon_t}}\ln\left(\frac{\sqrt{LR_{t-1}^2}\tau}{\sqrt{\varepsilon_t}\delta}\right), \left(\frac{\sigma R_{t-1}}{\varepsilon_t}\right)^{2} \ln\left(\frac{\tau}{\delta}\left(\frac{\sigma R_{t-1}}{\varepsilon_t}\right)^{2}\right) \right\} \right)\\
        &=& \cO\left(\sum\limits_{t=1}^\tau \max\left\{ \sqrt{\frac{L}{\mu}}\ln\left(\frac{\sqrt{L}\tau}{\sqrt{\mu}\delta}\right), \left(\frac{\sigma}{\mu R_{t-1}}\right)^{2} \ln\left(\frac{\tau}{\delta}\left(\frac{\sigma }{\mu R_{t-1}}\right)^{2}\right) \right\} \right)\\
        &=& \cO\left(\max\left\{ \tau \sqrt{\frac{L}{\mu}}\ln\left(\frac{\sqrt{L}\tau}{\sqrt{\mu}\delta}\right), \sum\limits_{t=1}^\tau\left(\frac{\sigma \cdot 2^{\nicefrac{t}{2}}}{\mu R}\right)^{2} \ln\left(\frac{\tau}{\delta}\left(\frac{\sigma \cdot 2^{\nicefrac{t}{2}}}{\mu R}\right)^{2}\right) \right\} \right)\\
        &=& \cO\left(\max\left\{ \sqrt{\frac{L}{\mu}}\ln\left(\frac{\mu R^2}{\varepsilon}\right)\ln\left(\frac{\sqrt{L}}{\sqrt{\mu}\delta}\ln\left(\frac{\mu R^2}{\varepsilon}\right)\right), \left(\frac{\sigma}{\mu R}\right)^{2} \ln\left(\frac{\tau}{\delta}\left(\frac{\sigma \cdot 2^{\nicefrac{\tau}{2}}}{\mu R}\right)^{2}\right)\sum\limits_{t=1}^\tau 2^t\right\} \right)\\
        &=& \cO\left(\max\left\{ \sqrt{\frac{L}{\mu}}\ln\left(\frac{\mu R^2}{\varepsilon}\right)\ln\left(\frac{\sqrt{L}}{\sqrt{\mu}\delta}\ln\left(\frac{\mu R^2}{\varepsilon}\right)\right), \left(\frac{\sigma}{\mu R}\right)^{2} \ln\left(\frac{\tau}{\delta}\left(\frac{\sigma}{\mu R}\right)^{2}\cdot 2\right)2^\tau\right\} \right)\\
        &=& \cO\left(\max\left\{ \sqrt{\frac{L}{\mu}}\ln\left(\frac{\mu R^2}{\varepsilon}\right)\ln\left(\frac{\sqrt{L}}{\sqrt{\mu}\delta}\ln\left(\frac{\mu R^2}{\varepsilon}\right)\right), \left(\frac{\sigma^2}{\mu \varepsilon}\right) \ln\left(\frac{1}{\delta}\left(\frac{\sigma^2}{\mu \varepsilon}\right)\ln\left(\frac{\mu R^2}{\varepsilon}\right)\right)\right\} \right),
    \end{eqnarray*}
    which concludes the proof.
\end{proof}
\newpage

\section{PROPERTIES OF HERMITE POLYNOMIALS}
\label{sec:hermite}

This section collects some properties of Hermite polynomials which are used in the proof of Lemma \ref{lem:f-g_non-uniform_bound}. First, let us recall the definition. There are two versions of Hermite polynomials, which are referred to as ``physicist's'' and ``probabilist's'',  given by
\[
    \scH_n(x) = (-1)^n e^{x^2} \frac{\rmd^n}{\rmd x^n} e^{-x^2}
    \quad \text{and} \quad
    \cH_n(x) = (-1)^n e^{x^2 / 2} \frac{\rmd^n}{\rmd x^n} e^{-x^2 / 2}
    \quad n \in \mathbb N,
\]
respectively. Obviously, for any positive integer $n$, the following relation holds true:
\begin{equation}
    \label{eq:probabilist_physicist_relation}
    \cH_n(x) \equiv 2^{-n / 2} \scH_n\left( \frac{x}{\sqrt 2} \right).
\end{equation}
Hence, ``probabilist's'' Hermite polynomials inherit all the properties of ``physicist's'' ones. In \citep{indritz1961inequality}, the author proved that
\[
    \max\limits_{x \in \R} \left| \scH_n (x) \, e^{-x^2 / 2} \right|
    \leq \sqrt{2^n \cdot n!}
    \quad \text{for all $n \in \mathbb N$.}
\]
In view of \eqref{eq:probabilist_physicist_relation}, this implies that
\[
    \max\limits_{x \in \R} \left| \cH_n (x) \, e^{-x^2 / 4} \right|
    = 2^{-n / 2} \, \max\limits_{x \in \R} \left| \scH_n \left( \frac{x}{\sqrt 2} \right) \, e^{-x^2 / 4} \right|
    \leq \sqrt{n!}
    \quad \text{for all $n \in \mathbb N$.}
\]


\section{NUMERICAL EXPERIMENTS: ADDITIONAL DETAILS}\label{appendix:experiments_details}
For every combination of noise distribution and method, we tuned optimal parameters for 70000 steps and ran methods on 95000 steps, where one step is one oracle call. 

The optimal values of the learning rate and the clipping parameter were selected via grid search over the sets $\{0.002, 0.004, 0.008, 0.01, 0.02, 0.04\}$ and $\{0.75, 1, 1.5, 2, 4, 8\}$, respectively.

\begin{table}[h]
\begin{center}
\begin{tabular}{ | m{3cm} | m{5cm}| m{2cm} | m{1.75cm} | } 
  \hline
  Distribution & Method & Learning Rate & Clipping parameter \\ 
  \hline
  Cauchy & clipped-MB-SGD \newline Med-MB-SGD \newline MB-clipped-SGD \newline clipped-Med-MB-SGD \newline SMoM-MB-SGD \newline clipped-SMoM-MB-SGD & 0.004 \newline 0.002 \newline 0.01 \newline 0.002 \newline 0.002 \newline 0.008 & 4 \newline - \newline 4 \newline 2 \newline - \newline 1 \\ 
  \hline
  Cauchy + Exponential & clipped-MB-SGD \newline Med-MB-SGD \newline MB-clipped-SGD \newline clipped-Med-MB-SGD \newline SMoM-MB-SGD \newline clipped-SMoM-MB-SGD & 0.008 \newline 0.002 \newline 0.04 \newline 0.002 \newline 0.002 \newline 0.002 & 1.5 \newline - \newline 4 \newline 8 \newline - \newline 8 \\ 
  \hline
  Cauchy + Pareto & clipped-MB-SGD \newline Med-MB-SGD \newline MB-clipped-SGD \newline clipped-Med-MB-SGD \newline SMoM-MB-SGD \newline clipped-SMoM-MB-SGD & 0.008 \newline 0.002 \newline 0.02 \newline 0.002 \newline 0.002 \newline 0.008 & 1.5 \newline - \newline 0.75 \newline 8 \newline - \newline 1 \\
  \hline
\end{tabular}
\end{center}
\caption{Optimal parameters for different distributions and methods.}
\end{table}

We also provide plots, reflecting the dependence of the error on the number of iterations, where one iteration is one method's update, see Figure \ref{fig:enter-label} below. As we can see, \algname{clipped-SMoM-MB-SGD} converges much faster than the competitors due to the larger batch size.
\begin{figure*}[h]
    \includegraphics[width=\textwidth]{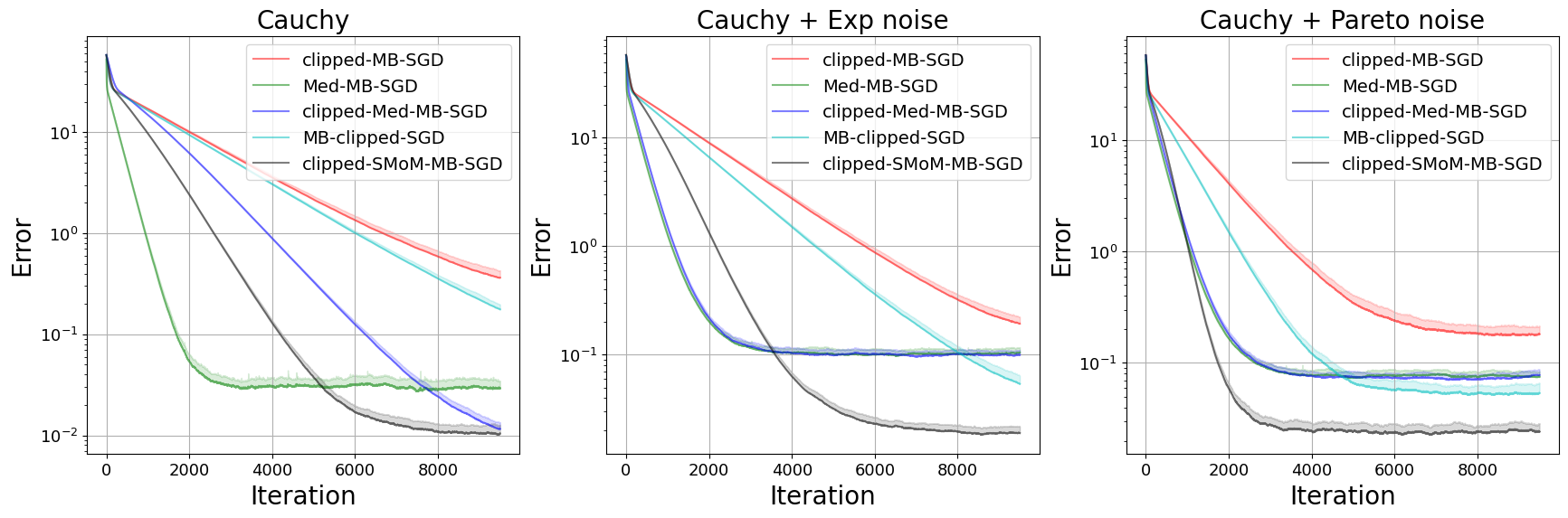}
    \centering
    \caption{Dependence of the mean error on the number of iterations with a standard deviation upper bound.}
    \label{fig:enter-label}
\end{figure*}

Finally, according to our theoretical findings, the error bound for the mini-batched SGD with clipped smoothed median of means grows logarithmically with $1/\delta$. In Figure \ref{fig:percentile_diff}, we plot the dependence of the confidence interval width on the number of iterations to illustrate this point.

\begin{figure*}[h]
    \includegraphics[width=\textwidth]{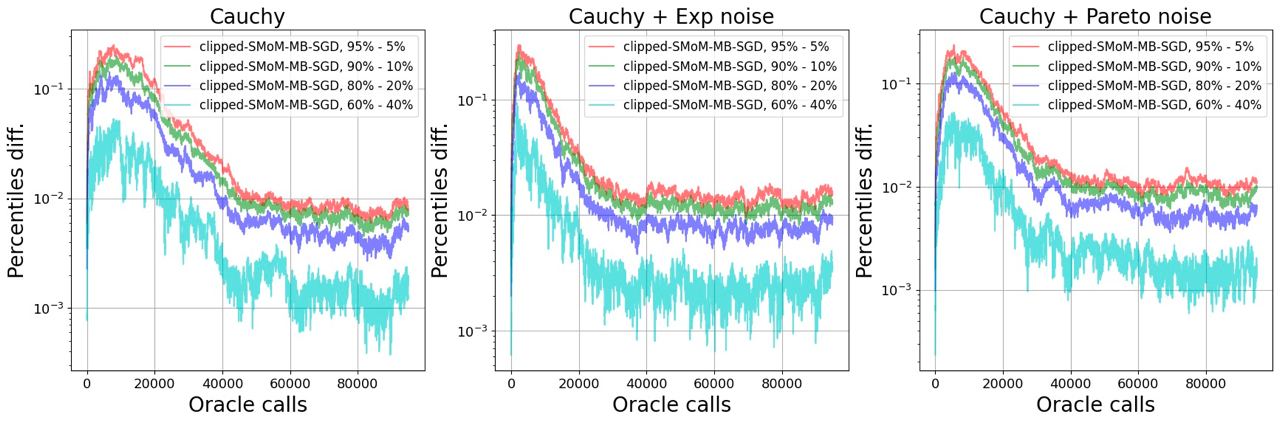}
    \centering
    \caption{Dependence of the confidence interval width for the error of mini-batched SGD with clipped smoothed median of means on the number of iterations.}
    \label{fig:percentile_diff}
\end{figure*}

\end{document}